\documentclass[a4paper, 11pt, reqno]{amsart}

\let\oldtocsection=\tocsection
\let\oldtocsubsection=\tocsubsection
\let\oldtocsubsubsection=\tocsubsubsection
\renewcommand{\tocsection}[2]{\hspace{0em}\oldtocsection{#1}{#2}}
\renewcommand{\tocsubsection}[2]{\hspace{1em}\oldtocsubsection{#1}{#2}}
\renewcommand{\tocsubsubsection}[2]{\hspace{2em}\oldtocsubsubsection{#1}{#2}}

\usepackage{amsthm,amssymb,amsmath,enumerate,graphicx,psfrag,enumitem,mathrsfs}

\usepackage[margin=2.55cm]{geometry}
\usepackage{hyperref}
\hypersetup{colorlinks=true,
   citecolor=blue,
   filecolor=blue,
   linkcolor=blue,
   urlcolor=blue
}

\usepackage[textsize=footnotesize]{todonotes}

\usepackage{algorithm}
\usepackage{tikz}
\usetikzlibrary{arrows,snakes}

\newtheorem{definition}{Definition}[section]

\newtheorem{proposition}[definition]{Proposition}
\newtheorem{theorem}[definition]{Theorem}
\newtheorem{corollary}[definition]{Corollary}
\newtheorem{lemma}[definition]{Lemma}

\numberwithin{equation}{section}

\newcommand{\comment}[1]{}
\newcommand{\N}{\mathbb N}
\newcommand{\R}{\mathbb R}

\newcommand{\es}{\emptyset}

\newcommand{\cl}{{\rm cl}}

\newcommand{\cA}{\mathcal{A}}
\newcommand{\cB}{\mathcal{B}}

\newcommand{\cE}{\mathcal{E}}
\newcommand{\cF}{\mathcal{F}}

\newcommand{\cH}{\mathcal{H}}

\newcommand{\cK}{\mathcal{K}}
\newcommand{\cL}{\mathcal{L}}

\newcommand{\cO}{\mathcal{O}}
\newcommand{\cP}{\mathcal{P}}
\newcommand{\cQ}{\mathcal{Q}}
\newcommand{\cR}{\mathcal{R}}

\newcommand{\bP}{\mathbf{P}}
\newcommand{\bF}{\mathbf{F}}
\newcommand{\bT}{\mathbf{T}}

\newcommand{\bR}{\mathbf{R}}
\newcommand{\bI}{\mathbf{I}}

\newcommand{\ba}{\mathbf{a}}
\newcommand{\bb}{\mathbf{b}}

\newcommand{\bd}{\mathbf{d}}

\newcommand{\bx}{\mathbf{x}}
\newcommand{\by}{\mathbf{y}}
\newcommand{\bz}{\mathbf{z}}
\newcommand{\bw}{\mathbf{w}}
\newcommand{\bu}{\mathbf{u}}

\newcommand{\dist}{{\rm dist}}

\newcommand{\sA}{\mathscr{A}}
\newcommand{\sB}{\mathscr{B}}

\newcommand{\sG}{\mathscr{G}}
\newcommand{\sH}{\mathscr{H}}

\newcommand{\sO}{\mathscr{O}}
\newcommand{\sP}{\mathscr{P}}
\newcommand{\sQ}{\mathscr{Q}}

\newcommand{\Pro}{\mathbb{P}}

\newcommand{\kk}{{(k)}}

\newcommand{\inj}{{\rm inj}}
\newcommand{\cut}{{\rm cut}}
\newcommand{\maxcut}{{\rm maxcut}}
\newcommand{\bPr}{\mathbf{Pr}}

\newcommand{\norm}[1]{\|#1\|_\infty}

\renewcommand{\epsilon}{\varepsilon}

\newcommand{\sm}{\setminus}
\newcommand{\sub}{\subseteq}

\newcommand{\COMMENT}[1]{}
\title{A characterization of testable hypergraph properties}

\author[F.~Joos]{Felix Joos}
\address{Institut f\"ur Informatik, Universit\"at Heidelberg, Germany}
\email{joos@informatik.uni-heidelberg.de}

\author[J.~Kim]{Jaehoon Kim}
\address{Department of Mathematical Sciences, KAIST, South Korea 34141}
\email{jaehoon.kim@kaist.ac.kr}

\author[D.~K\"uhn]{Daniela K\"uhn}
\author[D.~Osthus]{Deryk Osthus}
\address{School of Mathematics, University of Birmingham, Edgbaston, Birmingham, B15 2TT, United Kingdom}
\email{d.kuhn@bham.ac.uk, d.osthus@bham.ac.uk}

\thanks{
 The research leading to these results was partially supported by the EPSRC, grant nos. EP/M009408/1 (F.~Joos, D.~K\"uhn and D.~Osthus), 
and EP/N019504/1 (D.~K\"uhn).
The research was  also partially supported by the European Research Council under the European Union's Seventh Framework Programme (FP/2007--2013) / ERC Grant 306349 (J.~Kim and D.~Osthus).\\
An extended abstract of this work appeared in  \textit{FOCS 2017}, 859--867.
}

\date{\today}

\begin{document}

\begin{abstract}
We provide a combinatorial characterization of all testable properties of $k$-uniform hypergraphs ($k$-graphs for short).
Here, a $k$-graph property $\bP$ is testable if there is a randomized algorithm which makes a bounded number of edge queries and 
distinguishes with probability $2/3$ between $k$-graphs that satisfy $\bP$ and those that are far from satisfying $\bP$.
For the $2$-graph case, such a combinatorial characterization was obtained by Alon, Fischer, Newman and Shapira.
Our results for the $k$-graph setting are in contrast to those of Austin and Tao, who showed that for the somewhat stronger concept of local repairability,
the testability results for graphs do not extend to the $3$-graph setting.
\end{abstract}
\maketitle

%\newpage

%\tableofcontents

%\newpage
\section{Introduction}
The universal question in the area of property testing is the following:
By considering a small (random) sample $S$ of a combinatorial object $\cO$,
can we distinguish (with high probability) whether $\cO$ has a specific property $\bP$
or whether it is far from satisfying $\bP$?
In this paper we answer this question for $k$-uniform hypergraphs,
where a hypergraph $H$ is \emph{$k$-uniform} if all edges of $H$ have size  $k\in \N$.
For brevity, we usually refer to $k$-uniform hypergraphs as \emph{$k$-graphs}  (so $2$-graphs are graphs).

We now formalize the notion of testability (throughout, we consider only properties $\bP$ which are decidable).
For this, we say that two $k$-graphs $G$ and $H$ on vertex set $V$ with $|V|=n$ are \emph{$\alpha$-close} 
if $|G\triangle H| \leq \alpha \binom{n}{k}$, 
and \emph{$\alpha$-far} otherwise\footnote{We identify hypergraphs with their edge set and for two sets~$A,B$ we denote by $A\triangle B$ the symmetric difference of $A$ and $B$.}. 
We say that $H$ is \emph{$\alpha$-close to satisfying a property $\bP$} 
if there exists a $k$-graph $G$ that satisfies $\bP$ and is $\alpha$-close to $H$, 
and we say that $H$ is \emph{$\alpha$-far from satisfying $\bP$} otherwise.

\begin{definition}[Testability]\label{def: testable property}
Let $k\in \mathbb{N}\setminus\{1\}$ be fixed and let $q_{k}:(0,1)\rightarrow \mathbb{N}$ be a function.
A $k$-graph property $\bP$ is \emph{testable with query complexity at most $q_k$} if 
for every $n\in \mathbb{N}$ and every $\alpha \in (0,1)$, 
there are an integer $q'_k= q'_k(n,\alpha)\leq q_k(\alpha)$ and a randomized algorithm $\bT=\bT(n,\alpha)$ 
that can distinguish with probability at least $2/3$ between $n$-vertex $k$-graphs satisfying $\bP$ and $n$-vertex $k$-graphs that are $\alpha$-far from satisfying $\bP$, while making $q'_k$ edge queries:
\begin{enumerate}[label=\rm (\roman*)]
\item if $H$ satisfies $\bP$, 
then $\bT$ accepts $H$ with probability at least $2/3$,
\item if $H$ is $\alpha$-far from satisfying $\bP$, 
then $\bT$ rejects $H$ with probability at least $2/3$.
\end{enumerate} 
In this case, we say $\bT$ is a \emph{tester}, or an \emph{$(n,\alpha)$-tester} for $\bP$.
We also say that $\bT$ has query complexity $q'_k$. 
The property $\bP$ is \emph{testable} if it is testable with query complexity at most $q_k$ for some function $q_{k}: (0,1)\rightarrow \mathbb{N}$. 
\end{definition}

Property testing was introduced by Rubinfeld and Sudan~\cite{RS96}.
In the graph setting, the earliest systematic results were obtained in a seminal paper of Goldreich, Goldwasser and Ron~\cite{GGR98}. 
These included
$k$-colourability, max-cut and more general graph partitioning problems. 
(In fact, these results are preceded by the famous triangle removal lemma of Ruzsa and Szemer\'edi~\cite{RS78}, which can be rephrased in terms of testability of triangle-freeness.)
This list of problems was greatly extended (e.g.~via a description in terms of first order logic by Alon, Fischer, Krivelevich, and Szegedy~\cite{AFKS00})
and generalized first to monotone properties (which are closed under vertex and edge deletion) by Alon and Shapira~\cite{AS08a} 
and then to hereditary properties  (which are closed under vertex deletion), again by Alon and Shapira~\cite{AS08}.
Examples of non-testable properties include some properties which are closed under edge deletion~\cite{GT03}
and the property of being isomorphic to a given graph $G$~\cite{AFNS09,Fish05}, provided the local structure of $G$ is sufficiently `complex'
(e.g.~$G$ is obtained as a binomial random graph).
This sequence of papers culminated in the result of Alon, Fischer, Newman and Shapira~\cite{AFNS09} 
who obtained a combinatorial characterization of all testable graph properties.
This solved a problem posed already by~\cite{GGR98}, which was regarded as one of the main open problems in the area.

The characterization proved in~\cite{AFNS09}  states that a $2$-graph property $\bP$ is testable if and only if it is `regular reducible'.
Roughly speaking, the latter means that $\bP$ can be characterized by being close to one of a bounded number of 
(weighted) Szemer\'edi-partitions (which arise from an application of Szemer\'edi's regularity lemma).
Our main theorem (Theorem~\ref{thm:main}) shows that this can be extended to hypergraphs of higher uniformity.
Our characterization is based on the concept of (strong) hypergraph regularity,
which was introduced in the ground-breaking work of R\"odl together with Frankl, Schacht, and Skokan~\cite{FR92,RS07count,RS07,RS04}, Gowers~\cite{Gow07}, see also Tao~\cite{Tao06}.
We defer the precise definition of regular reducibility for $k$-graphs to Section~\ref{sec: regularity instances},
as the concept of (strong) hypergraph regularity involves additional features compared to the graph setting
(in particular, one needs to consider an entire (suitably nested) family of regular partitions, one for each $j \in [ k ]$).
Accordingly, our argument relies on the so-called `regular approximation lemma' due to R\"odl and Schacht~\cite{RS07}, which can be viewed as a powerful variant of the hypergraph regularity lemma.
In turn, we derive a strengthening of this result in the companion paper \cite[Lemma~6.1]{JKKO2}.
This is used in~\cite{JKKO2} to prove a random subhypergraph sampling result (Lemma~\ref{lem: random choice2}) which is a key tool in the proof of our main result.

Instead of testing whether $H$ satisfies $\bP$ or is $\alpha$-far from $\bP$,
it is natural to consider the more general task of estimating the distance between $H$ and $\bP$:
given $\alpha>\beta>0$,
is $H$ $(\alpha-\beta)$-close to satisfying $\bP$ or is $H$ $\alpha$-far from satisfying $\bP$?
In this case we refer to $\bP$ as being \emph{estimable}.
The formal definition is as follows.
\begin{definition}[Estimability]\label{def: estimable property def}
Let $k\in \mathbb{N}\setminus\{1\}$ be fixed and let $q_k:(0,1)^2\rightarrow \mathbb{N}$ be a function.
A $k$-graph property $\bP$ is \emph{estimable with query complexity at most $q_k$} 
if for every $n\in \mathbb{N}$ and all $\alpha,\beta \in (0,1)$ with $0<\beta<\alpha$ there are an integer $q'_k=q'_k(n,\alpha,\beta) \leq q_k(\alpha,\beta)$ and a randomized
algorithm $\bT=\bT(n,\alpha,\beta)$ that can distinguish with probability $2/3$ 
between $n$-vertex $k$-graphs that are $(\alpha-\beta)$-close to satisfying $\bP$ 
and $n$-vertex $k$-graphs that are $\alpha$-far from satisfying $\bP$ 
while making $q'_k$ edge queries:  
\begin{itemize}
\item if $H$ is $(\alpha-\beta)$-close to satisfying $\bP$, 
then $\bT$ accepts $H$ with probability at least $2/3$, 
\item if $H$ is $\alpha$-far from satisfying $\bP$, 
then $\bT$ rejects $H$ with probability at least $2/3$. 
\end{itemize}
In this case, we say $\bT$ is an \emph{estimator}, or \emph{$(n,\alpha,\beta)$-estimator} for $\bP$. 
We also say that $\bT$ has query complexity $q'_k$. 
The property $\bP$ is \emph{estimable} if it is estimable with query complexity at most $q_k$ for some function $q_{k}: (0,1)^2\rightarrow \mathbb{N}$. 
\end{definition}

We show that testability and estimability are in fact equivalent. 
For graphs this goes back to Fischer and Newman~\cite{FN07}.

\begin{theorem}\label{thm:main}
Suppose $k\in\N\sm\{1\}$ and suppose $\bP$ is a $k$-graph property.
Then the following three statements are equivalent:
\begin{itemize}
\item[\rm (a)] $\bP$ is testable.
\item[\rm (b)] $\bP$ is estimable.
\item[\rm (c)] $\bP$ is regular reducible. 
\end{itemize}
\end{theorem}
In Section~\ref{sec:application}, we illustrate how Theorem~\ref{thm:main} can be used to prove testability of a given property:
firstly to test the injective homomorphism density of a given subgraph (which includes the classical example of $H$-freeness)
and secondly to test the size of a maximum $\ell$-way cut (which includes testing $\ell$-colourability).\COMMENT{
First sample a random set of vertices in order to estimate $c$ such that $e(H)=c\binom{n}{k}$.
Then use max $\ell$-way cut tester to test whether $\maxcut_\ell\geq c-\alpha/2$.}
The proof of the latter relies also on our strengthening of the regular approximation lemma~\cite[Lemma~5.1]{JKKO2}.

Previously, the most general result on hypergraph property testing was the testability of hereditary properties, which was 
proved by R\"odl and Schacht~\cite{RS07STOC,RS09}, based on deep results on hypergraph regularity.
In fact, they showed that hereditary $k$-graph properties can be even tested with one-sided error
 (which means that the `2/3' is replaced by `1' in Definition~\ref{def: testable property}(i)).
 This generalized earlier results in~\cite{ARS07,KNR02}.

The result of Alon and Shapira on the testability of hereditary graph properties was strengthened by Austin and Tao~\cite{AT10} in another direction:
 they showed that hereditary properties of graphs are not only testable with  one-sided error,
but they are also \emph{locally repairable}\footnote{
Suppose $\bP$ is a hereditary graph property and $\epsilon>0$.
We say that a graph $G$ is locally $\delta$-close to $\bP$ if a random sample $S$ satisfies $\bP$ with probability at least $1-\delta$.
A result of Alon and Shapira~\cite{AS08}
shows that whenever $G$ is locally $\delta$-close to $\bP$ for some $\delta(\epsilon)>0$,
then $G$ is $\epsilon$-close to $\bP$.
The concept of being locally repairable strengthens this by requiring a
rule that generates $G'\in \bP$ only based on $S$ such that $|G\triangle G'|<\epsilon n^2$
with probability at least $1-\delta$.
}
(one may think of this as a strengthening of testability).
On the other hand, they showed that hereditary properties of $3$-graphs are not necessarily locally repairable.
Note that this is in contrast to Theorem~\ref{thm:main}.%

An intimate connection between property testing and graph limits was established by 
Borgs, Chayes, Lov\'asz, S\'os, Szegedy and Vesztergombi~\cite{BCLSSV06}.
In particular, they showed that a graph property $\bP$ is testable if and only if
for all sequences $(G_n)$ of graphs with $|V(G_n)|\to \infty$
and $\delta_\square(G_n,\bP)\to 0$,
we have $d_1(G_n,\bP)\to 0$.
Here $\delta_\square(G,\bP)$ denotes the cut-distance of $G$ and the closest graph satisfying $\bP$
and $d_1(G,\bP)$ is the normalized edit-distance between $G$ and $\bP$ (see also~\cite{Lov12} for more background and discussion on this).
Another characterization (in terms of localized samples) using the graph limit framework was given by Lov\'asz and Szegedy~\cite{LS10}.
Similarly, the result of R\"odl and Schacht~\cite{RS07STOC} on testing hereditary hypergraph properties 
was reproven via hypergraph limits by Elek and Szegedy~\cite{ES12} 
as well as Austin and Tao~\cite{AT10}. 
The latter further extended this to directed pre-coloured hypergraphs
(none of these results however yield  effective bounds on the query complexity).

Lov\'asz and Vesztergombi~\cite{LV13} introduced the notion of `non-deterministic' property testing,  
where the tester also has access to a `certificate' for the property $\bP$.
By considering the graph limit setting, they proved the striking result that any non-deterministically testable graph property is also deterministically testable 
(one could think of their result as the graph property testing analogue of proving that $\text{P}=\text{NP}$). 
Karpinski and Mark\'o~\cite{KM15} generalized the Lov\'asz-Vesztergombi result to hypergraphs, also via the notion of \mbox{(hyper-)graph} limits.
However, these proofs do not give explicit bounds on the query complexity -- this was achieved by Gishboliner and Shapira~\cite{GS14}
for graphs and Karpinski and Mark\'o~\cite{KM15a}
for hypergraphs.

Another direction of research concerns \emph{easily testable properties}, where
we require that the size of the sample is bounded from above by a polynomial in $1/\alpha$.
(The bounds coming from Theorem~\ref{thm:main} can be made explicit but are quite large, as the approach via the (hyper-)graph regularity lemma incurs at least a tower-type dependence on  $1/\alpha$, see~\cite{Gow97}.)
For $k$-graphs, Alon and Shapira~\cite{AS05} as well as Alon and Fox~\cite{AF15} obtained positive and negative results for the property of containing a given $k$-graph as an (induced) subgraph.
For an approach via a `polynomial' version of the regularity lemma see~\cite{FPS16}.

More recent progress on property testing includes many questions beyond the hypergraph setting.
Instances include property testing of matrices~\cite{AB16}, 
Boolean functions~\cite{AB10,AS03}, geometric objects~\cite{ADPR03},
and algebraic structures~\cite{BCLR08,FPS16,FIS05}.
Moreover, property testing in the sparse (graph) setting gives rise to many interesting results and questions (see e.g.~\cite{BSS10,NS13}).
Little is known for hypergraphs in this case (see~\cite{EJKO:19} for partial results on testing for sub(hyper-)graph-freeness in (hyper-)graphs which are not dense).
For much more details on property testing, we refer to~\cite{Gol:17}.

\medskip

The paper is organized as follows.
In the next section,
we outline the main steps of the argument.
In Section~\ref{sec:concepts}, we explain the relevant concepts of hypergraph regularity and collect a number of tools related to hypergraph regularity. 
In particular, we introduce the regular approximation lemma of R\"odl and Schacht (Theorem~\ref{thm: RAL}) and we describe a suitable `induced' version of the hypergraph counting lemma.
In Section~\ref{sec:counting}, we prove and derive a number of tools related to hypergraph regularity.
In Section~\ref{sec:testred},
we use this counting lemma to show that testable properties are regular reducible.
 %In Section~\ref{sec:testred}, we use this counting lemma to show that testable properties are regular reducible.
In Section~\ref{sec:regtest},
we show how our random sampling result from~\cite{JKKO2} (Lemma~\ref{lem: random choice2}) implies that satisfying a given regularity instance is testable.
In Section~\ref{sec:testest},
we then show that estimability is equivalent to testability.
In Section~\ref{sec:regredtest},
we combine the previous results to show that regular reducible properties are testable.
%Sections~\ref{sec: Regular approximations of partitions and hypergraphs} and~\ref{sec: Sampling a regular partition} are then devoted to the proof of Lemma~\ref{lem: random choice2}.
Finally, in Section~\ref{sec:application} we discuss applications of our main result and illustrate in detail how to apply Theorem~\ref{thm:main}.

\section{Proof sketch}\label{sec:proofsketch}

In the following,
we describe the main steps leading to the proof of Theorem~\ref{thm:main}.
While the general strategy emulates that of~\cite{AFNS09},
the hypergraph setting leads to many additional challenges.

\subsection{Testable properties are regular reducible}

We first discuss the implication (a)$\Rightarrow$(c) in Theorem~\ref{thm:main}.
(Note that the statement of (c) is formalized in Section~\ref{sec: regularity instances}.) 
The detailed proof is given in Section~\ref{sec:testred}.
The argument involves the following concepts.
A regularity instance $R=(\epsilon,\ba,d_{\ba,k})$ consists of a regularity parameter $\epsilon$, 
a vector $\ba\in \N^{k-1}$ determining the `address space' of $R$, 
 and a density function $d_{\ba,k}$
on the address space described by $\ba$.
(In the graph case, 
$\ba$ equals the number of parts of the regularity partition and the address space consists of all pairs of parts.)
We say a $k$-graph $H$ satisfies $R$ if there is a family of partitions $\sP=\{\sP^{(i)}\}_{i=1}^{k-1}$
(where $\sP^{(1)}$ is a partition of $V(H)$ and $\sP^{(i)}$ is a partition of all those $i$-sets which `cross' $\sP^{(1)}$)
so that $\sP$ is an $\epsilon$-equitable partition of $H$ with density function $d_{\ba,k}$.
(In the graph case this means that $\sP=\sP^{(1)}$ is a vertex partition so that all pairs of partition classes induce $\epsilon$-regular bipartite graphs.)
Then a property $\bP$ is regular reducible if there is a bounded size set $\cR$ of regularity instances so that $H$ is close to satisfying some $R\in \cR$
if and only if $H$ is close to satisfying $\bP$ (see Definition~\ref{def: regular reducible}).

Goldreich and Trevisan~\cite{GT03} proved that every testable graph property
is also testable in some canonical way (and their results translate to the hypergraph setting in a straightforward way).
Thus we may restrict ourselves to such canonical testers.
More precisely, an $(n,\alpha)$-tester $\bT=\bT(n,\alpha)$ is \emph{canonical}
if, given an $n$-vertex $k$-graph $H$, it chooses
a set $Q$ of $q_k' = q_k'(n,\alpha)$ vertices of $H$ uniformly at random, 
queries all $k$-sets in $Q$,
and then accepts or rejects $H$ (deterministically) according to (the isomorphism class of) $H[Q]$. 
In particular, $\bT$ has query complexity $\binom{q_k'}{k}$. 
Moreover, every canonical tester is non-adaptive.

Let $\bP$ be a testable $k$-graph property. 
Thus there exists a function $q_k:(0,1)\rightarrow \mathbb{N}$ such that for every $n\in \mathbb{N}$ and $\alpha \in (0,1)$,
there exists a canonical $(n,\alpha)$-tester $\bT = \bT(n,\alpha)$ for $\bP$ with query complexity at most $q_k(\alpha)$. 
So $\bT$ samples a set $Q$ of $q\leq q_k(\alpha)$ vertices, 
considers $H[Q]$, and then deterministically accepts or rejects $H$ based on $H[Q]$.
Let $\cQ$ be the set of all the $k$-graphs on $q$ vertices such that $\bT$ accepts $H$ if and only if there is $Q'\in \cQ$ that is isomorphic to $H[Q]$.

Now let $\bPr(\cQ,H)$ denote the `density' of copies of $k$-graphs  $Q\in \cQ$ in $H$ (see Section~\ref{sec: basic notation}).
As $\bT$ is an $(n,\alpha)$-tester,
$\bPr(\cQ,H)\geq 2/3$ if $H$ satisfies $\bP$ and $\bPr(\cQ,H)\leq 1/3$ if $H$ is $\alpha$-far from $\bP$.
The strategy is now to 
apply a suitable `induced' version (Corollary~\ref{cor: counting collection}) of the hypergraph counting lemma (Lemma~\ref{lem: counting}).
Corollary~\ref{cor: counting collection} shows that $\bPr(\cQ,H)$ can be approximated by a function $IC(\cQ,d_{\ba,k})$,
where $d_{\ba,k}$ is the density function of an equitable partition $\sP$ of $H$.
Accordingly, 
for a suitable small $\epsilon>0$ and all $\ba\in \N^{k-1}$ in a specified range (in terms of $\alpha$, $q_k(\alpha)$ and $k$),
we define a `discretized' set $\bI$ of regularity instances $(\epsilon,\ba,d_{\ba,k})$ such that $d_{\ba,k}(\cdot)$ only attains a bounded number of possible values.
Now setting $\cR(n,\alpha) := \{ R \in \bI :  IC(\cQ,d_{\ba,k}) \geq 1/2\}$ leads to the desired result,
as Corollary~\ref{cor: counting collection} 
implies $IC(\cQ,d_{\ba,k})\approx \bPr(\cQ,H)$ if $H$ satisfies $(\epsilon,\ba,d_{\ba,k})$.
(In the actual argument,
we consider some $k$-graph $G$ obtained from the regular approximation lemma (Theorem~\ref{thm: RAL}) rather than $H$ itself.)

\subsection{Satisfying a regularity instance is testable}\label{subsec:regtestable}

In this subsection we sketch how we prove that the property of satisfying a particular regularity instance is testable.
This forms the main part of the proof of Theorem~\ref{thm:main} and is described in Section~\ref{sec:regtest}.
Suppose $H$ is a $k$-graph and $Q$ is a subset of the vertices chosen uniformly at random.
First we show that if $H$ satisfies a regularity instance $R$,
then with high probability $H[Q]$ is close to satisfying $R$.
Also the converse is true:
if $H$ is far from satisfying $R$,
then with high probability $H[Q]$ is also far from satisfying $R$.

The main tool for this is Lemma~\ref{lem: random choice2} (which is proven in \cite{JKKO2}). 
Roughly speaking,
it states the following.
\begin{equation*}
\begin{minipage}[c]{0.90\textwidth}\em
Suppose $H$ is a $k$-graph and $Q$ a random subset of $V(H)$.
Then with high probability, the following hold  (where $\delta \ll \epsilon_0$).
\begin{itemize}
	\item If $\sO_1$ is an $\epsilon_0$-equitable partition of $H$ with density function $d_{\ba,k}$,
	then there is an $(\epsilon_0+\delta)$-equitable partition of $H[Q]$ with the same density function~$d_{\ba,k}$.
	\item If $\sO_2$ is an $\epsilon_0$-equitable partition of $H[Q]$ with density function $d_{\ba,k}$,
	then there is an $(\epsilon_0+\delta)$-equitable partition of $H$ with the same density function~$d_{\ba,k}$.
\end{itemize}
\end{minipage}
\end{equation*}
The key point here is that the transfer between $H$ and $H[Q]$ incurs only an additive increase in the regularity parameter $\epsilon_0$.
This additive increase can then be eliminated by slightly adjusting $H$ (or $H[Q]$).

It is not difficult to deduce from Lemma~\ref{lem: random choice2} that satisfying a given regularity instance is testable (Theorem~\ref{thm: regularity instance is testable}).

\subsection{The final step}
We now aim to use Theorem~\ref{thm: regularity instance is testable} to show that (c)$\Rightarrow$(a) in Theorem~\ref{thm:main},
i.e.~to prove that a regular reducible property $\bP$ is also testable (see Section~\ref{sec:regredtest}).
As $\bP$ is regular reducible,
we can decide whether $H$ satisfies $\bP$ if we can test whether $H$ is close to some regularity instance in a certain set $\cR$.
To achieve this,
we strengthen Theorem~\ref{thm: regularity instance is testable} to show that the property of satisfying a given regularity instance $R$ is actually estimable
(the equivalence (a)$\Leftrightarrow$(b) is a by-product of this argument, see Section~\ref{sec:testest}).
Having proved this, it is straightforward to construct a tester for $\bP$
by appropriately combining $|\cR|$ estimators which estimate the distance of $H$ and a given $R\in \cR$.

\section{Concepts and tools}\label{sec:concepts}

In this section we introduce the main concepts and tools (mainly concerning hypergraph regularity partitions) which form the basis of our approach. The constants in the hierarchies used to state our results have to be chosen from right to left. More precisely, if we claim that a result holds whenever $1/n \ll a \ll b \leq 1$ (where $n\in \N$ is typically the number of vertices of a hypergraph), 
then this means that there are non-decreasing functions $f : (0, 1] \rightarrow (0, 1]$ and $g : (0, 1] \rightarrow (0, 1]$ such that the result holds for all $0 < a, b \leq 1 $ and all $n \in \mathbb{N}$ with $a \leq f(b)$ and $1/n \leq g(a)$.
For a vector $\bx =(\alpha_1,\dots, \alpha_\ell)$, we let $\bx_*:= \{ \alpha_1,\dots, \alpha_\ell\}$ and write $\norm{\bx} = \max_{i\in [\ell]} \alpha_i$. We say a set $E$ is an \emph{$i$-set} if $|E|=i$.
Unless stated otherwise, in the partitions considered in this paper, we allow some of the parts to be empty.

\subsection{Hypergraphs}\label{sec: basic notation}

In the following we introduce several concepts about a hypergraph $H$.
We typically refer to $V=V(H)$ as the vertex set of $H$
and usually let $n:=|V|$.
We often identify $H$ with the collection $E(H)$ of its edges and write $|H|$ to denote the number of edges of $H$.
Given a hypergraph $H$ and a set $Q\subseteq V(H)$, we denote by $H[Q]$ the hypergraph induced on $H$ by $Q$.
For two $k$-graphs $G,H$ on the same vertex set, we often refer to $|G\triangle H|$  as the \emph{distance} between $G$ and $H$.
If the vertex set of $H$  has a partition $\{V_1,\ldots,V_\ell\}$,
we simply refer to $H$ as a hypergraph on $\{V_1,\ldots,V_\ell\}$.

A partition $\{V_1,\ldots,V_\ell\}$ of $V$ is an \emph{equipartition} if $|V_i|=|V_j|\pm 1$ for all $i,j\in [\ell]$. For a partition $\{V_1,\ldots,V_\ell\}$ of $V$ and $k\in [\ell]$,
we denote by $K^{(k)}_\ell(V_1,\dots, V_\ell)$ the \emph{complete $\ell$-partite} $k$-graph with vertex classes $V_1,\ldots, V_\ell$.
Let $0\leq \lambda<1$.
If $|V_i| = (1\pm \lambda)m$ for every $i\in [\ell]$, 
then an \emph{$(m,\ell,k,\lambda)$-graph} $H$ on $\{V_1, \dots, V_{\ell}\}$ 
is a spanning subgraph of $K^\kk_{\ell}(V_1,\dots, V_{\ell})$.
For notational convenience, we consider the vertex partition $\{V_1,\ldots ,V_\ell\}$ as an $(m,\ell,1,\lambda)$-graph. 
If $|V_i| \in \{m,m+1\}$, 
we drop $\lambda$ and simply refer to $(m,\ell,k)$-graphs.
Similarly, if the value of $\lambda$ is not relevant, then we say $H\subseteq K^\kk_{\ell}(V_1,\dots, V_{\ell})$ is an $(m,\ell,k,*)$-graph.

Given an $(m,\ell,k,*)$-graph $H$ on $\{V_1,\dots, V_\ell\}$, an integer $k\leq i\leq \ell$ and a set $\Lambda_i \in \binom{[\ell]}{i}$, 
we set $H[\Lambda_i] := H[\bigcup_{\lambda'\in \Lambda_i} V_{\lambda'}]$. If $2\leq k\leq i \leq \ell$ and $H$ is an $(m,\ell,k,*)$-graph,
we denote by $\cK_i(H)$ the family of all $i$-element subsets $I$ of $V(H)$ for which $H[I]\cong K^{\kk}_i$,
where $K^{\kk}_i$ denotes the complete $k$-graph on $i$ vertices.

If $H^{(1)}$ is an $(m,\ell,1,*)$-graph and $i\in [\ell]$, 
we denote by $\cK_i(H^{(1)})$ the family of all $i$-element subsets $I$ of $V(H^{(1)})$ which `cross' the partition $\{V_1,\dots,V_\ell\}$; 
that is, $I\in \cK_i(H^{(1)})$ if and only if $|I\cap V_s|\leq 1$ for all $s\in [\ell]$.

We will consider hypergraphs of different uniformity on the same vertex set.
Given an $(m,\ell,k-1,\lambda)$-graph $H^{(k-1)}$ and an $(m,\ell,k,\lambda)$-graph $H^\kk$ on the same vertex set, 
we say $H^{(k-1)}$ \emph{underlies} $H^\kk$ if $H^\kk\subseteq \cK_k(H^{(k-1)})$;
that is, for every edge $e\in H^\kk$ and every $(k-1)$-subset $f$ of $e$,
we have $f\in H^{(k-1)}$.
If we have an entire cascade of underlying hypergraphs we refer to this as a complex.
More precisely, 
let $m\geq 1$ and $\ell\geq k\geq 1$ be integers. 
An \emph{$(m,\ell,k,\lambda)$-complex}  $\cH$ on $\{V_1,\ldots,V_\ell\}$ 
is a collection of $(m,\ell,j,\lambda)$-graphs $\{H^{(j)}\}_{j=1}^{k}$ on $\{V_1,\ldots,V_\ell\}$ such that $H^{(j-1)}$ underlies $H^{(j)}$ for all $j\in [k]\sm\{1\}$, that is, $H^{(j)} \subseteq \cK_{j}(H^{(j-1)})$.
Again, if $|V_i| \in \{m,m+1\}$, then we simply drop $\lambda$ and refer to such a complex as an $(m,\ell,k)$-complex. 
If the value of $\lambda$ is not relevant, then we say that $\{H^{(j)}\}_{j=1}^{k}$  is an $(m,\ell,k,*)$-complex. A collection of hypergraphs is a \emph{complex} if it is an $(m,\ell,k,*)$-complex for some integers $m,\ell,k$.

When $m$ is not of primary concern, we refer to $(m,\ell,k,\lambda)$-graphs and $(m,\ell,k,\lambda)$-complexes simply as 
$(\ell,k,\lambda)$-graphs and $(\ell,k,\lambda)$-complexes, respectively.
Again, we also omit $\lambda$ if $|V_i|\in \{m,m+1\}$ and refer to $(\ell,k)$-graphs and $(\ell,k)$-complexes and we write the symbol `$*$' instead of $\lambda$ if $\lambda$ is not relevant.

Note that there is no ambiguity between an $(\ell,k,\lambda)$-graph and an $(m,\ell,k)$-graph (and similarly for complexes) as $\lambda<1$.

Suppose $n\geq \ell\geq k$ and
suppose $H$ is an $n$-vertex $k$-graph and $F$ is an $\ell$-vertex $k$-graph.
We define $\mathbf{Pr}(F,H)$ 
such that $\mathbf{Pr}(F,H)\binom{n}{\ell}$ equals the number of induced copies of $F$ in~$H$. 
For a collection $\cF$ of $\ell$-vertex $k$-graphs, 
we define $\mathbf{Pr}(\cF, H)$ 
such that $\mathbf{Pr}(\cF,H)\binom{n}{\ell}$ equals the number of induced $\ell$-vertex $k$-graphs $F$ in $H$ such that $F\in \cF$.
Note that the following proposition holds.
\begin{proposition}\label{prop: mathbf Pr doesn't change much}
Suppose $n,k,q\in \mathbb{N}$ with $k\leq q\leq n$ and $G$ and $H$ are $n$-vertex $k$-graphs on vertex set $V$ and $\cF$ is a collection of $q$-vertex $k$-graphs.
If $|G \triangle H| \leq \nu \binom{n}{k}$, then 
$$\mathbf{Pr}(\cF, G) = \mathbf{Pr}(\cF, H) \pm q^k \nu.$$
\end{proposition}
\COMMENT{
Note that adding or removing an edge can decrease the number of induced copies of members of $\cF$ by at most $\binom{n-k}{q-k}$. (Note that the size of $\cF$ is irrelevant as the $k$-vertices in the added/removed edge with $q-k$ other vertices can form at most one graph in $\cF$.) Thus adding or removing $\nu \binom{n}{k}$ edges can decrease the number of induced copies of the members of $\cF$ by at most $\nu \binom{n}{k}\binom{n-k}{q-k} \leq \nu q^k \binom{n}{q}$.
}

\subsection{Probabilistic tools}
%For $n\in \mathbb{N}$ and $0\leq p\leq 1$ we write $Bin(n,p)$ to denote the binomial distribution with parameters $n$ and $p$. 
For $m,n,N\in \mathbb{N}$ with $m,n<N$ the \emph{hypergeometric distribution} with parameters $N$, $n$ and $m$ is the distribution of the random variable $X$ defined as follows. Let $S$ be a random subset of $\{1,2, \dots, N\}$ of size $n$ and let $X:=|S\cap \{1,2,\dots, m\}|$. We will use the following bound, which is a simple form of Chernoff-Hoeffding's inequality.

\begin{lemma}
[See {\cite[Remark 2.5, Theorem 2.8 and Theorem 2.10]{JLR00}}]  \label{lem: chernoff} 
Suppose $X_1,\dots, X_n$ are independent random variables such that $X_i\in \{0,1\}$ for all $i\in [n]$. 
Let $X:= X_1+\dots + X_n$. Then for all $t>0$, $\mathbb{P}[|X - \mathbb{E}[X]| \geq t] \leq 2e^{-2t^2/n}$. 
Suppose $Y$ has a hypergeometric distribution with parameters $N,n,m$,
then
$\mathbb{P}[|Y - \mathbb{E}[Y]| \geq t] \leq 2e^{-2t^2/n}$.
\end{lemma}

\COMMENT{For the proof of Lemma~\ref{lem: random subset edge size} we need to introduce the concept of martingales.
A sequence $X_0,\dots, X_N$ of random variables is a {\em martingale} 
if $X_0$ is a fixed real number and $\mathbb{E}[X_{n}\mid X_0,\dots,X_{n-1}] = X_{n-1}$ for all $n\in [N]$. 
We say that the martingale $X_0,\dots, X_N$ is {\em $c$-Lipschitz} if $|X_{n}-X_{n-1}| \leq c$ holds for all $n\in[N]$. 
We apply Azuma's inequality with martingales of the form $X_i:=\mathbb{E}[X\mid Y_1,\dots, Y_i]$, where $X$ and $Y_1,\dots,Y_i$ are some previously defined random variables.
\begin{theorem}[Azuma's inequality \cite{Azu67, Hoe63}]%\label{Azuma} 
Suppose that $\lambda, c >0$ and that $X_0,\dots, X_N$ is a $c$-Lipschitz martingale. 
Then 
\begin{align*}
\mathbb{P}[\left|X_N-X_0\right|\geq \lambda]\leq 2e^{\frac{-\lambda^2}{2Nc^2}}.
\end{align*}
\end{theorem}
}

The next lemma is easy to show, e.g.~using Azuma's inequality.
We omit the proof.

\begin{lemma}\label{lem: random subset edge size}
Suppose $0<1/n \leq  1/q \ll 1/k \leq 1/2$ and $1/q \ll \nu$.
Let $H$ be an $n$-vertex $k$-graph on vertex set $V$. Let $Q\in \binom{V}{q}$ be a $q$-vertex subset of $V$ chosen uniformly at random. Then 
$$\mathbb{P}\left[|H[Q]| = \frac{q^{k}}{n^{k}}|H| \pm \nu \binom{q}{k} \right] \geq 1- 2 e^{\frac{-\nu^2q}{8k^2}}.$$
\end{lemma}
\COMMENT{
\begin{proof}
Suppose we reveal $q$ vertices $v_1,\dots, v_q$ of $V$ one by one and let $Q:=\{v_1,\dots, v_q\}$.
Consider an exposure martingale $X_0,\dots, X_q$ such that 
$X_i:=\mathbb{E}[ |H[Q]| \mid v_1,\dots, v_i]$. 
Then it is easy to check that this forms a $\binom{q}{k-1}$-Lipschitz martingale. Moreover, $X_0 = \mathbb{E}[|H[Q]|] = \binom{q}{k}|H|/\binom{n}{k}$.
Therefore, by Azuma's inequality, we obtain
$$
\mathbb{P}\left[|H[Q]| 
= \frac{q^{k}}{n^{k}}|H| \pm \nu \binom{q}{k} \right] 
\geq 1 - 2e^{-\nu^2 \binom{q}{k}^2 /(3\binom{q}{k-1}^2 q) } \geq 1- 2 e^{-\nu^2q/(8k^2)}.$$
\end{proof}
}

\subsection{Hypergraph regularity}\label{sec: 2 hypergraph regularity}
In this subsection we introduce $\epsilon$-regularity for hypergraphs.
Suppose $\ell \geq k\geq 2$ and $V_1,\dots, V_{\ell}$ are pairwise disjoint vertex sets.
Let $H^\kk$ be an $(\ell,k,*)$-graph on $\{V_1,\dots, V_{\ell}\}$, 
let $\{i_1,\dots, i_{k} \} \in \binom{[\ell]}{k}$, 
and let $H^{(k-1)}$ be a $(k,k-1,*)$-graph on $\{V_{i_1},\dots, V_{i_k}\}$. 
We define the \emph{density of $H^\kk$ with respect to $H^{(k-1)}$} as
$$d(H^{(k)} \mid H^{(k-1)}) := \left\{ \begin{array}{ll} \frac{|H^{(k)}\cap \cK_k(H^{(k-1)})|}{|\cK_k(H^{(k-1)})| } & \text{ if } |\cK_k(H^{(k-1)})|>0, \\
0 & \text{otherwise.} 
\end{array}\right.$$
Suppose $\epsilon>0$ and $d\geq 0$.
We say $H^{(k)}$ is \emph{$(\epsilon,d)$-regular with respect to $H^{(k-1)}$} 
if for all $Q^{(k-1)}\subseteq H^{(k-1)}$ with  
$$| \cK_k(Q^{(k-1)})| \geq \epsilon |\cK_k(H^{(k-1)})|, 
\text{ we have } |H^{(k)}\cap \cK_k(Q^{(k-1)})| = (d\pm \epsilon) |\cK_k(Q^{(k-1)})|.$$\COMMENT{We don't want to write $d(H^{(k)}\mid Q^{(k-1)}) = d\pm \epsilon$ here since we later want that if $\cK_{k}(H^{(k-1)})=\emptyset$ then any $H^{(k)}$  is $(\epsilon,d)$-regular with respect to $H^{(k-1)}$ for all $d,\epsilon \in [0,1]$.}
Note that if $H^{(k)}$ is $(\epsilon,d)$-regular with respect to $H^{(k-1)}$ and $H^{(k-1)}\neq \emptyset$, then we have $d(H^{(k)} \mid H^{(k-1)}) = d\pm \epsilon$. We say $H^{(k)}$ is \emph{$\epsilon$-regular with respect to $H^{(k-1)}$} if it is $(\epsilon,d)$-regular with respect to $H^{(k-1)}$ for some $d\geq 0$.

We say an $(\ell,k,*)$-graph $H^{(k)}$ 
on $\{V_1,\dots, V_\ell\}$ is \emph{$(\epsilon,d)$-regular with respect to an $(\ell,k-1,*)$-graph $H^{(k-1)}$} 
on $\{V_1,\dots, V_\ell\}$
if for every $\Lambda \in \binom{[\ell]}{k}$ %the restriction $H^{(k)}[\Lambda]$ 
$H^{(k)}$ is $(\epsilon,d)$-regular with respect to the restriction $H^{(k-1)}[\Lambda]$. 

Let $\bd=(d_2,\dots, d_{k})\in \R^{k-1}_{\geq 0}$.
We say an $(\ell,k,*)$-complex $\cH = \{H^{(j)}\}_{j=1}^{k}$ is \emph{$(\epsilon,\bd)$-regular} if $H^{(j)}$ is $(\epsilon,d_j)$-regular with respect to $H^{(j-1)}$ for every $j\in [k]\sm \{1\}$. 
We sometimes simply refer to a complex as being $\epsilon$-regular if it is $(\epsilon,\bd)$-regular for some vector $\bd$.

\subsection{Partitions of hypergraphs and the regular approximation lemma}\label{sec: partitions of hypergraphs and RAL}

The regular approximation lemma of R\"odl and Schacht implies that 
for all $k$-graphs $H$, there exists a $k$-graph $G$
which is very close to $H$ and so that $G$ has a very `high quality' partition into $\epsilon$-regular subgraphs.
To state this formally we need to introduce further concepts involving partitions of hypergraphs.

Suppose $A\supseteq B$ are finite sets, 
$\sA$ is a partition of $A$, 
and $\sB$ is a partition of $B$.
We say $\sA$ \emph{refines} $\sB$ and write $\sA \prec \sB$ 
if for every $\cA\in \sA$ there either exists $\cB \in \sB$ such that $\cA \subseteq \cB$ or $\cA \subseteq A\setminus B$.
The following definition concerns `approximate' refinements.
Let $\nu\geq 0$.
We say that $\sA$ \emph{$\nu$-refines} $\sB$ and write $\sA \prec_{\nu} \sB$ 
if there exists a function $f: \sA \rightarrow \sB\cup \{A\setminus B\}$ such that 
$$\sum_{\cA \in \sA}|\cA \setminus f(\cA)| \leq \nu |A|.$$
We make the following observations.
\begin{equation}\label{eq: prec triangle}
\begin{minipage}[c]{0.9\textwidth}\em
\begin{itemize}
\item  $\sA\prec \sB$ if and only if $\sA\prec_{0} \sB$. 
\item Suppose $\sA,\sA',\sA''$ are partitions of $A,A',A''$ respectively and $A''\subseteq A' \subseteq A$. 
If $\sA \prec_{\nu} \sA'$ and $\sA' \prec_{\nu'} \sA''$, 
then $\sA\prec_{\nu+\nu'}\sA''$.
\end{itemize} 
\end{minipage}\ignorespacesafterend 
\end{equation}

We now introduce the concept of a polyad.
Roughly speaking, given a vertex partition $\sP^{(1)}$, an $i$-polyad is an $i$-graph which arises from a partition $\sP^{(i)}$ of the complete partite $i$-graph $\cK_i(\sP^{(1)})$.
The $(i+1)$-cliques spanned by all the $i$-polyads give rise to a partition $\sP^{(i+1)}$ of $\cK_{i+1}(\sP^{(1)})$ (see Definition~\ref{def: family of partitions}).
Such a `family of partitions' then provides a suitable framework for describing a regularity partition (see Definition~\ref{def: equitable family of partitions}).

Suppose we have a vertex partition $\sP^{(1)} = \{V_1,\ldots,V_\ell\}$ and $\ell\geq k$.
For integers $k\leq \ell'\leq \ell$, we say that a hypergraph $H$ is an \emph{$(\ell',k,*)$-graph with respect to $\sP^{(1)}$} if it is an $(\ell',k,*)$-graph on $\{ V_i : i\in \Lambda\}$ for some $\Lambda \in \binom{[\ell]}{\ell'}$.

Recall that $\cK_{j}(\sP^{(1)})$ is the family of all crossing $j$-sets with respect to $\sP^{(1)}$.
Suppose that for all  $i\in[k-1]\sm \{1\}$, 
we have partitions $\sP^{(i)}$ of $\cK_{i}(\sP^{(1)})$ such that each part of $\sP^{(i)}$ is an $(i,i)$-graph with respect to $\sP^{(1)}$.
By definition,
for each $i$-set $I\in \cK_{i}(\sP^{(1)})$, 
there exists exactly one $P^{(i)}=P^{(i)}(I)\in \sP^{(i)}$ so that $I\in P^{(i)}$.
Consider $j\in [\ell]$ and any $J\in \cK_j(\sP^{(1)})$.
For each $i\in[ \max\{j,k-1\}]$,
the \emph{$i$-polyad} $\hat{P}^{(i)}(J)$ of $J$ is defined by
\begin{align}\label{def:polyad}
	\hat{P}^{(i)}(J) := \bigcup\left\{P^{(i)}(I) : I\in \binom{J}{i}\right\}.
\end{align}
Thus $\hat{P}^{(i)}(J)$ is a $(j,i)$-graph with respect to $\sP^{(1)}$.
Moreover, let
\begin{align}\label{eq: hat cP}
\hat{\cP}(J):=\left\{\hat{P}^{(i)}(J)\right\}_{i=1}^{\max\{j,k-1\}},\end{align}
\COMMENT{Until now, this does not have to be complex, since we do not know that $\cK_{i}(\hat{P}^{(i-1)})$ underlies $P^{(i)} \in \sP^{(i)}$ or not.
Because since $\sP$ is arbitrary, $\hat{P}^{(i)}(J) = \hat{P}^{(i)}(J')$ while $\hat{P}^{(i-1)}(J) \neq \hat{P}^{(i-1)}(J')$ might happen.
}
and for $j\in [k-1]$, let
\begin{align}\label{eq:sP}
	\hat{\sP}^{(j)} := \left\{\hat{P}^{(j)}(J): J\in \cK_{j+1}(\sP^{(1)})\right\}.
\end{align}
We note that $\hat{\sP}^{(1)}$ is the set consisting of all $(2,1)$-graphs with vertex classes $V_s, V_t$ (for all distinct $s,t\in [\ell]$).
Moreover, note that if $\hat{P}^{(j)} \in \hat{\sP}^{(j)}$, it follows that there is a set $J\in \cK_{j+1}(\sP^{(1)})$ such that $\hat{P}^{(j)}=\hat{P}^{(j)}(J)$. Since $J\in \cK_{j+1}(\hat{P}^{(j)}(J))$, we obtain that $\cK_{j+1}(\hat{P}^{(j)})\neq \emptyset$ for any $\hat{P}^{(j)}\in \hat{\sP}^{(j)}$.

The above definitions apply to arbitrary partitions $\sP^{(i)}$ of $\cK_i(\sP^{(1)})$.
However, it will be useful to consider partitions with more structure.
\begin{definition}[Family of partitions]\label{def: family of partitions}
Suppose $k\in \N\sm \{1\}$ and $\ba=(a_1,\dots, a_{k-1})\in \N^{k-1}$.
We say $\sP= \sP(k-1,\ba)= \{\sP^{(1)},\dots, \sP^{(k-1)}\}$ is a \emph{family of partitions on $V$} 
if it satisfies the following for each $j\in [k-1]\setminus \{1\}$:
\begin{enumerate}[label={ \rm(\roman*)}]
\item $\sP^{(1)}$ is a partition of $V$ into $a_1 \geq k$ nonempty classes,
\item $\sP^{(j)}$ is a partition of $\cK_{j}(\sP^{(1)})$ into nonempty\COMMENT{Here, I added the word `nonempty' because of the injectivity of $\ba$-labelling $\phi^{(j)}$ we use in section 2.6.1. Since the parts of $\sP^{(k-1)}$ are non-empty this also implies that $a_1\geq k-1$.} 
$j$-graphs such that
\begin{itemize}
	\item $\sP^{(j)}\prec \{\cK_j(\hat{P}^{(j-1)}): \hat{P}^{(j-1)}\in \hat{\sP}^{(j-1)}\}$ and
	\item $|\{P^{(j)}\in \sP^{(j)} : P^{(j)} \subseteq \cK_j(\hat{P}^{(j-1)})\}|=a_j$ for every $\hat{P}^{(j-1)}\in \hat{\sP}^{(j-1)}$.
\end{itemize}
\end{enumerate}
\end{definition}
We say $\sP = \sP(k-1,\ba)$ is \emph{$T$-bounded} if $\norm{\ba}\leq T$. 
For two families of partitions $\sP = \sP(k-1,\ba^{\sP})$ and $\sQ =\sQ(k-1,\ba^{\sQ})$, we say \emph{$\sP \prec \sQ$} if $\sP^{(j)} \prec \sQ^{(j)}$ for all $j\in [k-1]$. We say \emph{$\sP \prec_{\nu} \sQ$} if $\sP^{(j)} \prec_{\nu} \sQ^{(j)}$ for all $j\in [k-1]$.

As the concept of polyads is central to this paper,
we emphasize the following:
\begin{proposition}\label{prop: (i,i)-hypergraph}
Let $k\in \mathbb{N}\setminus \{1\}$, $\ba\in \mathbb{N}^{(k-1)}$ and $\sP = \sP(k-1,\ba)$ be a family of partitions. Then for all $i\in [k-1]$ and $j\in [a_1]$, the following hold.
\begin{enumerate}[label={ \rm(\roman*)}]
\item if $i>1$, then $\sP^{(i)}$ is a partition of $\cK_{i}(\sP^{(1)})$ into $(i,i,*)$-graphs with respect to $\sP^{(1)}$, 
\item each $\hat{P}^{(i)}\in \hat{\sP}^{(i)}$ is an $(i+1,i,*)$-graph with respect to $\sP^{(1)}$,
\item for each $j$-set $J\in \cK_j(\sP^{(1)})$, $\hat{\cP}(J)$ as defined in \eqref{eq: hat cP} is a complex.
\end{enumerate}
\end{proposition}

We now extend the concept of $\epsilon$-regularity to families of partitions.
\begin{definition}[Equitable family of partitions]\label{def: equitable family of partitions}
Let $k\in \N\sm \{1\}$.
Suppose $\eta>0$ and $\ba=(a_1,\dots, a_{k-1})\in \N^{k-1}$.
Let $V$ be a vertex set of size $n$.
We say a family of partitions $\sP= \sP(k-1,\ba)$ on $V$ is \emph{$(\eta,\epsilon,\ba,\lambda)$-equitable} if it satisfies the following:
\begin{enumerate}[label={ \rm(\roman*)}]
\item $a_1 \geq \eta^{-1}$,
\item $\sP^{(1)} = \{V_i : i\in [a_1]\}$ satisfies $|V_i|= (1\pm \lambda)n/a_1$ for all $i\in[a_1]$, and
\item if $k\geq 3$, then for every $k$-set $K\in \cK_{k}(\sP^{(1)})$ 
the collection $\hat{\cP}(K)=\{\hat{P}^{(j)}(K)\}_{j=1}^{k-1}$ is an $(\epsilon,\bd)$-regular $(k,k-1,*)$-complex, where $\bd= (1/a_2,\dots, 1/a_{k-1})$. 
\end{enumerate}
\end{definition}
As before we drop $\lambda$ if $|V_i| \in \{ \lfloor n/a_1\rfloor, \lfloor n/a_1\rfloor +1\}$ and say $\sP$ is $(\eta,\epsilon,\ba)$-equitable.
Note that for any $\lambda\leq 1/3$, every $(\eta,\epsilon,\ba,\lambda)$-equitable family of partitions $\sP$ satisfies
\begin{align}\label{eq: eta a1}
\left|\binom{V}{k}\setminus \cK_{k}(\sP^{(1)})\right| \leq k^2 \eta \binom{n}{k}.
\end{align}
\COMMENT{A $k$-set $K$ is not in $\cK_{k}(\sP^{(1)})$ if $|K\cap V_i|\geq 2$ for some $i\in [a_1]$.
Thus we choose $i$ with $\leq a_1$ ways, and choose two vertices in $\binom{V_i}{2}$ with $\leq \binom{(1+\lambda)n/a_1}{2}$ ways, and choose $k-2$ other vertices arbitrarily with $\leq \binom{n}{k-2}$ ways.
Thus we have 
$$\left|\binom{V}{k}\setminus \cK_{k}(\sP^{(1)})\right| \leq a_1 \cdot \binom{(1+\lambda)n/a_1}{2} \cdot \binom{n}{k-2} \leq  \frac{1}{2 a_1} (1+\lambda)^2k(k-1)\binom{n}{k} \leq k^2 \eta \binom{n}{k}.$$
}

We next introduce the concept of perfect $\epsilon$-regularity with respect to a family of partitions.
\begin{definition}[Perfectly regular]
Suppose $\epsilon>0$ and $k\in \N\sm\{1\}$. 
Let $H^{(k)}$ be a $k$-graph with vertex set $V$ and let $\sP= \sP(k-1,\ba)$ be a family of partitions on $V$. 
We say $H^{(k)}$ is perfectly $\epsilon$-regular with respect to $\sP$ 
if for every $\hat{P}^{(k-1)}\in \hat{\sP}^{(k-1)}$ the graph $H^{(k)}$ is $\epsilon$-regular with respect to $\hat{P}^{(k-1)}$.
\end{definition}

Having introduced the necessary notation, we are now ready to state the regular approximation lemma due to R\"odl and Schacht.
It states that for every $k$-graph $H$, 
there is a $k$-graph $G$ that is close to $H$ and that has very good regularity properties.

\begin{theorem}[Regular approximation lemma~\cite{RS07}]\label{thm: RAL}
Let $k\in \N\sm\{1\}$. 
For all $\eta,\nu>0$ and every function $\epsilon: \mathbb{N}^{k-1}\rightarrow (0,1]$, 
there are integers $t_0:= t_{\ref{thm: RAL}}(\eta,\nu,\epsilon)$ and $n_0:=n_{\ref{thm: RAL}}(\eta,\nu,\epsilon)$ 
so that the following holds:

For every $k$-graph $H$ on at least $n\geq n_0$ vertices,
there exists a $k$-graph $G$ on $V(H)$ and a family of partitions $\sP=\sP(k-1,\ba^{\sP})$ on $V(H)$ so that 
\begin{enumerate}[label={\rm (\roman*)}]
\item $\sP$ is $(\eta,\epsilon(\ba^{\sP}),\ba^{\sP})$-equitable and $t_0$-bounded,
\item $G$ is perfectly $\epsilon(\ba^{\sP})$-regular with respect to $\sP$, and
\item $|G\triangle H|\leq \nu \binom{n}{k}$.
\end{enumerate}
\end{theorem}

The crucial point here is that in applications we may apply Theorem~\ref{thm: RAL} with a function $\epsilon$ such that $\epsilon(\ba^{\sP})\ll \norm{\ba^\sP}^{-1}$.
This is in contrast to other versions (see e.g.~\cite{Gow07,RS04,Tao06}) where (roughly speaking) in (iii) we have $G=H$ but in (ii) we have an error parameter $\epsilon'$ which may be large compared to $\norm{\ba^\sP}^{-1}$.

We next state a generalization (Lemma~\ref{RAL(k)}) of the regular approximation lemma
which was also proved by R\"odl and Schacht (see Lemma~25 in~\cite{RS07}).
Lemma~\ref{RAL(k)} 
has two additional features in comparison to Theorem~\ref{thm: RAL}.
Firstly, we can prescribe a family of partitions $\sQ$ and obtain a refinement $\sP$ of $\sQ$,
and secondly, we are not only given one $k$-graph $H$ but a collection of $k$-graphs $H_i$ that partitions the complete $k$-graph.
Thus we may view Lemma~\ref{RAL(k)} as a `partition version' of Theorem~\ref{thm: RAL}.

\begin{lemma}[R\"odl and Schacht~\cite{RS07}]\label{RAL(k)}
For all $o,s\in \N$, $k\in \N\sm \{1\}$, all $\eta,\nu>0$, 
and every function $\epsilon : \mathbb{N}^{k-1}\rightarrow (0,1]$, 
there are $\mu=\mu_{\ref{RAL(k)}}(k,o,s,\eta,\nu,\epsilon)>0$ 
and $t=t_{\ref{RAL(k)}}(k,o,s,\eta,\nu,\epsilon)\in \N$ and $n_0=n_{\ref{RAL(k)}}(k,o,s,\eta,\nu,\epsilon)\in\N$ such that the following hold. 
Suppose
\begin{itemize}
\item[{\rm (O1)$_{\ref{RAL(k)}}$}] $V$ is a set and $|V|=n\geq n_0$,
\item[{\rm (O2)$_{\ref{RAL(k)}}$}] $\sQ=\sQ(k,\ba^{\sQ})$ is a $(1/a_1^{\sQ},\mu,\ba^{\sQ})$-equitable $o$-bounded family of partitions on $V$,
\item[{\rm (O3)$_{\ref{RAL(k)}}$}] $\sH^{(k)}=\{H^{(k)}_1,\dots,H^{(k)}_{s}\}$ is a partition of $\binom{V}{k}$ so that $\sH^{(k)} \prec \sQ^{(k)}$.
\end{itemize}
Then there exist a family of partitions $\sP = \sP(k-1,\ba^{\sP})$ and a partition $\sG^{(k)}= \{ G^{(k)}_1,\dots, G^{(k)}_s\}$ of $\binom{V}{k}$ satisfying the following for every $i\in [s]$ and $j\in [k-1]$.
\begin{itemize}
\item[{\rm (P1)$_{\ref{RAL(k)}}$}] $\sP$ is a $t$-bounded $(\eta,\epsilon(\ba^{\sP}),\ba^{\sP})$-equitable family of partitions, and $a^{\sQ}_j$ divides $a^{\sP}_j$, 
\item[{\rm (P2)$_{\ref{RAL(k)}}$}] $\sP \prec \{\sQ^{(j)}\}_{j=1}^{k-1}$,
\item[{\rm (P3)$_{\ref{RAL(k)}}$}] $G^{(k)}_i$ is perfectly $\epsilon(\ba^{\sP})$-regular with respect to $\sP$,
\item[{\rm (P4)$_{\ref{RAL(k)}}$}] $\sum_{i=1}^{s} |G^{(k)}_i\triangle H^{(k)}_i| \leq \nu \binom{n}{k}$, and
\item[{\rm (P5)$_{\ref{RAL(k)}}$}] $\sG^{(k)} \prec \sQ^{(k)}$ and if $H^{(k)}_i \subseteq \cK_k(\sQ^{(1)})$, then $G_i^{(k)} \sub \cK_k(\sQ^{(1)})$.
\end{itemize}
\end{lemma}

In Lemma~\ref{RAL(k)} we may assume without loss of generality that $1/\mu,t,n_0$ are non-decreasing in $k,o,s$ and non-increasing in $\eta, \nu$.

\subsection{The address space}\label{sec: address space}
Later on, we will need to explicitly refer to the densities arising, for example, in Theorem~\ref{thm: RAL}(ii).
For this (and other reasons) it is convenient to consider the `address space'.
Roughly speaking the address space consists of a collection of vectors where each vector identifies a polyad.

For $a,s\in \N$, 
we recursively define 
$[a]^{s}$ by $[a]^{s} := [a]^{s-1}\times [a]$ and $[a]^{1}:= [a]$. 
To define the address space,
let us write $\binom{[a_1]}{\ell}_{<} := \{(\alpha_1,\dots, \alpha_\ell)\in[a_1]^\ell: \alpha_1 < \dots <\alpha_\ell \}$.

Suppose $k',\ell,p\in \N$, $\ell\geq k'$, and $p\geq\max\{ k'-1,1\}$, and $\ba=(a_1,\dots,a_p )\in \N^{p}$. 
We define 
$$\hat{A}(\ell,k'-1,\ba):= 
\binom{[a_1]}{\ell}_{<} \times \prod_{j=2}^{k'-1}[a_j]^{\binom{\ell}{j}}$$ to be the \emph{$(\ell,k')$-address space}. 
Observe that $\hat{A}(1,0,\ba)=[a_1]$ and $\hat{A}(2,1,\ba)=\binom{[a_1]}{2}_{<}$. 
Recall that for a vector $\bx$, the set $\bx_*$ was defined at the beginning of Section~\ref{sec:concepts}.
Note that if $k'>1$, then each $\hat{\bx}\in \hat{A}(\ell,k'-1,\ba)$ can be written as $\hat{\bx} = (\bx^{(1)},\dots, \bx^{(k'-1)})$, 
where $\bx^{(1)}\in \binom{[a_1]}{\ell}_{<}$ and $\bx^{(j)} \in [a_j]^{\binom{\ell}{j}}$ for each $j\in [k'-1]\sm\{1\}$. 
Thus each entry of the vector $\bx^{(j)}$ corresponds to (i.e.~is indexed by) an element of $\binom{[\ell]}{j}$. 
We order the elements of both $\binom{[\ell]}{j}$ and $\binom{\bx_*^{(1)}}{j}$ lexicographically and consider the bijection $g: \binom{\bx_*^{(1)}}{j} \rightarrow \binom{[\ell]}{j}$ which preserves this ordering.
For each $\Lambda \in \binom{\bx^{(1)}_*}{j}$ and $j\in [k'-1]$, 
we denote by $\bx^{(j)}_{\Lambda}$ the entry of $\bx^{(j)}$ which corresponds to the set $g(\Lambda)$.

\subsubsection{Basic properties of the address space}\label{sec: Basic properties of the address space}
Let $k\in\N\sm\{1\}$ and let $V$ be a vertex set of size $n$. 
Let $\sP(k-1,\ba)$ be a family of partitions on $V$. 
For each crossing $\ell$-set $L\in \cK_{\ell}(\sP^{(1)})$, 
the address space allows us to identify (and thus refer to) the set of polyads `supporting' $L$. 
We will achieve this by defining a suitable operator $\hat{\bx}(L)$ which maps $L$ to the address space.

To do this, write $\sP^{(1)}= \{V_i: i\in [a_1]\}$.  
Recall from Definition~\ref{def: family of partitions}(ii) that for $j\in[k-1]\setminus\{1\}$, we partition $\cK_j(\hat{P}^{(j-1)})$ of every $(j-1)$-polyad $\hat{P}^{(j-1)}\in\hat{\sP}^{(j-1)}$ into $a_j$ nonempty parts in such a way that $\sP^{(j)}$ is the collection of all these parts.
Thus, there is a labelling $\phi^{(j)}:\sP^{(j)} \rightarrow [a_j]$ such that for every polyad $\hat{P}^{(j-1)} \in \hat{\sP}^{(j-1)}$, the restriction of $\phi^{(j)}$ to $\{P^{(j)} \in \sP^{(j)}: P^{(j)}\subseteq \cK_j(\hat{P}^{(j-1)})\}$ is injective.\COMMENT{ To make this true, we added the work `nonempty' in Definition~\ref{def: family of partitions} which ensures that every element of $\sP^{(j)}$ are all distinct. }
The set $\mathbf{\Phi} := \{\phi^{(2)},\dots, \phi^{(k-1)}\}$ is called an \emph{$\ba$-labelling} of $\sP(k-1,\ba)$.
For a given set $L\in \cK_{\ell}(\sP^{(1)})$, we denote $\cl(L):= \{ i: V_i\cap L \neq \emptyset\}.$

Consider any $\ell \in [a_1]$. 
Let $j':=\min\{k-1,\ell-1\}$  and let $j'':=\max\{j',1\}$. 
For every $\ell$-set $L\in \cK_{\ell}(\sP^{(1)})$ we define an integer vector $\hat{\bx}(L) = (\bx^{(1)}(L),\dots, \bx^{(j'')}(L))$ by
\begin{equation}\label{eq: bx J def}
\begin{minipage}[c]{0.9\textwidth}\em
\begin{itemize}
\item $\bx^{(1)}(L) := (\alpha_1,\ldots,\alpha_\ell)$, where  $\alpha_1<\ldots<\alpha_\ell$ and $L \cap V_{\alpha_i}=\{v_{\alpha_i}\}$,
\item and for $i\in [j']\setminus\{1\}$ we set
$$\bx^{(i)}({L}) := \left(\phi^{(i)}(P^{(i)}): \{v_{\lambda} :\lambda \in \Lambda\}\in P^{(i)}, P^{(i)}\in \sP^{(i)}\right)_{\Lambda \in \binom{\cl(L)}{i}}.$$
\end{itemize}
\end{minipage}
\end{equation}
Here, we order $\binom{\cl(L)}{i}$ lexicographically. In particular, $\bx^{(i)}(L)$ is a vector of length $\binom{\ell}{i}$.

By definition, $\hat{\bx}(L) \in \hat{A}(\ell,j',\ba)$ for every $L\in \cK_{\ell}(\sP^{(1)})$ with $\ell, j'$ as above. 
Our next aim is to define an operator $\hat{\bx}(\cdot)$ which maps the set $\hat{\sP}^{(j-1)}$ of $(j-1)$-polyads injectively into the address space $\hat{A}(j,j-1,\ba)$ (see \eqref{eq: hat bx maps}). We will then extend this further into a bijection between elements of the address spaces and their corresponding hypergraphs.
However, before we can define $\hat{\bx}(\cdot)$, we need to introduce some more notation.

Suppose $j\in [k'-1]$. 
For $\hat{\bx} \in \hat{A}(\ell,k'-1,\ba)$ and $J\in \cK_{j}(\sP^{(1)})$ with $\cl(J) \subseteq \bx^{(1)}_{*}$, we define $\bx^{(j)}_{J} := \bx^{(j)}_{\cl(J)}$. Thus from now on, we may refer to the entries of $\bx^{(j)}$ either by an index set $\Lambda \in \binom{\bx^{(1)}_*}{j}$ or by a set $J \in \cK_{j}(\sP^{(1)})$.

Next we introduce a relation on the elements of (possibly different) address spaces.
Consider $\hat{\bx}= (\bx^{(1)}, \bx^{(2)},\dots,\bx^{(k'-1)}) \in \hat{A}(\ell,k'-1,\ba)$ with $\ell'\leq \ell$ and $k''\leq k'$.
We define $\hat{\by} \leq_{\ell',k''-1} \hat{\bx}$ if
\begin{itemize}
\item $\hat{\by} = (\by^{(1)}, \by^{(2)},\dots,\by^{(k''-1)})\in \hat{A}(\ell', k''-1,\ba)$,
\item $\by^{(1)}_*\subseteq \bx^{(1)}_*$ and 
\item $\bx^{(j)}_{\Lambda}=\by^{(j)}_{\Lambda}$ for any $\Lambda \in \binom{\by^{(1)}_*}{j}$ and $j\in [k''-1]\sm\{1\}$.
 \end{itemize}
Thus any $\hat{\by}\in \hat{A}(\ell', k''-1,\ba)$ with $\hat{\by} \leq_{\ell',k''-1} \hat{\bx}$ 
can be viewed as the restriction of $\hat{\bx}$ to an $\ell'$-subset of the $\ell$-set $\bx^{(1)}_*$.
%(where $k''$ indicates about how many levels below the $\ell'$-subset $\hat{\by}$ carries information).
Hence for $\hat{\bx} \in \hat{A}(\ell,k'-1,\ba)$, there are exactly $\binom{\ell}{\ell'}$ distinct integer vectors $\hat{\by}\in \hat{A}(\ell',k''-1,\ba)$ such that $\hat{\by} \leq_{\ell',k''-1} \hat{\bx}.$ 
Also it is easy to check the following properties.
\begin{proposition}\label{eq: I subset J then leq}
Suppose $\sP= \sP(k-1,\ba)$ is a family of partitions, $i\in [a_1]$ and $i':= \min\{i,k\}$. 
\begin{enumerate}[label={\rm (\roman*)}]
	\item Whenever $I\in \cK_{i}(\sP^{(1)})$ and $J\in \cK_{j}(\sP^{(1)})$ with $I\subseteq J$, then $\hat{\bx}(I) \leq_{i,i'-1} \hat{\bx}(J)$.
	\item If $J \in \cK_j(\sP^{(1)})$ and $\hat{\by} \leq_{i,i'-1} \hat{\bx}(J)$, then there exists a unique $I\in \binom{J}{i}$ such that $\hat{\by} = \hat{\bx}(I)$.
\end{enumerate}
\end{proposition}

Now we are ready to introduce the promised bijection between the elements of address spaces and their corresponding hypergraphs.

Consider $j\in [k]\setminus\{1\}$. Recall that for every $j$-set $J\in \cK_{j}(\sP^{(1)})$, 
we have $\hat{\bx}(J) \in \hat{A}(j,j-1,\ba)$. 
Moreover, recall that $\cK_{j}(\hat{P}^{(j-1)})\neq \emptyset$ for any $\hat{P}^{(j-1)}\in \hat{\sP}^{(j-1)}$, and note that $\hat{\bx}(J) = \hat{\bx}(J')$ for all $J,J' \in \cK_{j}(\hat{P}^{(j-1)})$ and all $\hat{P}^{(j-1)}\in \hat{\sP}^{(j-1)}$.
Hence, for each $\hat{P}^{(j-1)} \in \hat{\sP}^{(j-1)}$ we can define 
\begin{align}\label{eq: hat bx maps}
\hat{\bx}(\hat{P}^{(j-1)}):= \hat{\bx}(J)\text{ for some } J\in \cK_{j}(\hat{P}^{(j-1)}).
\end{align}
Let 
\begin{align*}
\hat{A}(j,j-1,\ba)_{\neq \emptyset} &:= \{ \hat{\bx}\in \hat{A}(j,j-1,\ba): \exists \hat{P}^{(j-1)}\in \hat{\sP}^{(j-1)} \text{ such that }\hat{\bx}(\hat{P}^{(j-1)}) = \hat{\bx}\}\\
&= \{ \hat{\bx} \in \hat{A}(j,j-1,\ba) : \exists J\in \cK_j(\sP^{(1)}) \text{ such that } \hat{\bx} = \hat{\bx}(J)\},\\
\hat{A}(j,j-1,\ba)_{\emptyset}&:= \hat{A}(j,j-1,\ba)\setminus \hat{A}(j,j-1,\ba)_{\neq \emptyset}.
\end{align*}
Clearly \eqref{eq: hat bx maps} gives rise to a bijection between $\hat{\sP}^{(j-1)}$ and $\hat{A}(j,j-1,\ba)_{\neq \emptyset}$.
Thus for each $\hat{\bx} \in \hat{A}(j,j-1,\ba)_{\neq \emptyset}$, we can define the polyad  $\hat{P}^{(j-1)}(\hat{\bx})$ of $\hat{\bx}$ by
\begin{align}\label{eq: hat P def}
\hat{P}^{(j-1)}(\hat{\bx}):= \hat{P}^{(j-1)} \text{ such that }\hat{P}^{(j-1)} \in \hat{\sP}^{(j-1)} \text{ with } \hat{\bx}= \hat{\bx}(\hat{P}^{(j-1)}).
\end{align}
Note that for any $J\in \cK_j(\sP^{(1)})$, we have $\hat{P}^{(j-1)}(\hat{\bx}(J)) = \hat{P}^{(j-1)}(J).$ 

We will frequently make use of an explicit description of a polyad in terms of the partition classes it contains (see \eqref{eq:hatPconsistsofP(x,b)}). For this, we proceed as follows.
For each $b\in [a_1]$, let $P^{(1)}(b,b):= V_b$.
For each $j\in [k-1]\setminus\{1\}$ and $(\hat{\bx},b)  \in \hat{A}(j,j-1,\ba)_{\neq\emptyset}\times [a_j]$, we let 
\begin{align}\label{eq: Pj operator define}
	P^{(j)}(\hat{\bx},b):=P^{(j)}\in \sP^{(j)} \text{ such that } \phi^{(j)}(P^{(j)})=b \text{ and } P^{(j)} \subseteq \cK_{j}(\hat{P}^{(j-1)}(\hat{\bx})). 
\end{align}
Using Definition~\ref{def: family of partitions}(ii), 
we conclude that so far $P^{(j)}(\hat{\bx},b)$ is well-defined for each $(\hat{\bx},b)  \in \hat{A}(j,j-1,\ba)_{\neq\emptyset}\times [a_j]$ and all $j \in [k-1]\setminus\{1\}$.

For convenience we now extend the domain of the above definitions to cover the `trivial' cases.
For $(\hat{\bx},b) \in  \hat{A}(j,j-1,\ba)_{\emptyset}\times [a_j]$, we let
\begin{align}\label{eq: Pj operator define empty}
P^{(j)}(\hat{\bx},b) := \emptyset.
\end{align}
We also let $P^{(1)}(a,b):= \emptyset$ for all $a,b\in [a_1]$ with $a\neq b$.
For all $j\in [k-1]$ and $\hat{\bx}\in \hat{A}(j+1,j,\ba)_{\emptyset}$, we define
\begin{align}\label{eq: hat P P relations emptyset}
\hat{P}^{(j)}(\hat{\bx}):= 
\bigcup_{\hat{\by}\leq_{j,j-1} \hat{\bx}} P^{(j)}(\hat{\by},\bx^{(j)}_{\by^{(1)}_* }).
\end{align}

To summarize, given a family of partitions $\sP=\sP(k-1,\ba)$ and an $\ba$-labelling $\mathbf{\Phi}$, 
for each $j\in [k-1]$, this defines $ P^{(j)}(\hat{\bx},b)$ for $\hat{\bx}\in \hat{A}(j,j-1,\ba)$ and $b\in [a_j]$ and $\hat{P}^{(j)}(\hat{\bx})$ for all $\hat{\bx}\in \hat{A}(j+1,j,\ba)$.
For later reference, we collect the relevant properties of these objects below.
For each $j\in [k-1]\setminus \{1\}$, it will be convenient to extend the domain of the $\ba$-labelling $\phi^{(j)}$ of $\sP^{(j)}$ to all $j$-sets $J\in \cK_j(\sP^{(1)})$ by
setting $\phi^{(j)}(J) := \phi^{(j)}(P^{(j)})$, where $P^{(j)}\in \sP^{(j)}$ is the unique $j$-graph that contains $J$. The proof of the following proposition is provided in \cite{JKKO2}.

\begin{proposition}\label{prop: hat relation}
For a given family of partitions $\sP=\sP(k-1,\ba)$ and an $\ba$-labelling~$\mathbf{\Phi}$, the following hold for all $j\in [k-1]$.
\begin{enumerate}[label={\rm (\roman*)}]
\item $\hat{P}^{(j)}(\cdot):\hat{A}(j+1,j,\ba)_{\neq\emptyset}\to \hat{\sP}^{(j)}$ is a bijection.
\item For $j\geq 2$, the restriction of $P^{(j)}(\cdot,\cdot)$ onto $\hat{A}(j,j-1,\ba)_{\neq \emptyset} \times [a_j]$ is a bijection onto $\sP^{(j)}$.

\item $\hat{\bx} \in \hat{A}(j+1,j,\ba)_{\neq\emptyset}$ if and only if $\cK_{j+1}(\hat{P}^{(j)}(\hat{\bx}))\neq \emptyset$.

\item Each $\hat{\bx}\in \hat{A}(j+1,j,\ba)$ satisfies
\begin{align}\label{eq:hatPconsistsofP(x,b)}
\hat{P}^{(j)}(\hat{\bx})= 
\bigcup_{\hat{\by}\leq_{j,j-1} \hat{\bx}} P^{(j)}(\hat{\by},\bx^{(j)}_{\by^{(1)}_* }).
\end{align}
\item $\{P^{(j)}(\hat{\bx},b) : \hat{\bx}\in \hat{A}(j,j-1,\ba), b\in [a_j]\}$ forms a partition of $\cK_{j}(\sP^{(1)})$.

\item $\{\cK_{j+1}(\hat{P}^{(j)}(\hat{\bx})):\hat{\bx}\in \hat{A}(j+1,j,\ba)\}$ forms a partition of $\cK_{j+1}(\sP^{(1)})$.
\item $\{P^{(j+1)}(\hat{\bx},b) : \hat{\bx}\in \hat{A}(j+1,j,\ba), b\in [a_{j+1}]\}\prec \{\cK_{j+1}(\hat{P}^{(j)}(\hat{\bx})):\hat{\bx}\in \hat{A}(j+1,j,\ba)\}$.
\item If $\sP(k-1,\ba)$ is $T$-bounded, then $|\hat{\sP}^{(j)}|\leq |\hat{A}(j+1,j,\ba)|\leq T^{2^{j+1}-1}$ and $|\sP^{(j)}| \leq T^{2^{j}}$.

\item If $\cK_{j+1}(\hat{P}^{(j)}(\hat{\bx}))\neq \emptyset$ for all $\hat{\bx}\in \hat{A}(j+1,j,\ba)$, then $\hat{P}^{(j)}(\cdot):\hat{A}(j+1,j,\ba)\to \hat{\sP}^{(j)}$ is a bijection and, if in addition $j<k-1$, then $P^{(j+1)}(\cdot,\cdot):\hat{A}(j+1,j,\ba)\times [a_{j+1}]\to \sP^{(j+1)}$ is also a bijection.

\item $\hat{A}(j,j-1,\ba)_{\emptyset} = \emptyset$ for all $j\in [2]$ and thus $\hat{P}^{(1)}(\cdot)$ and $P^{(2)}(\cdot,\cdot)$ are always bijections.

\item If $\sP \prec \sQ$ for another family of partitions $\sQ=\sQ(k-1,\ba^{\sQ})$, then $\{ \cK_{j+1}(\hat{P}^{(j)}(\hat{\bx})) : \hat{\bx}\in \hat{A}(j+1,j,\ba) \} \prec \{ \cK_{j+1}(\hat{Q}^{(j)}(\hat{\bx})) : \hat{\bx}\in \hat{A}(j+1,j,\ba^{\sQ})\}$.
\end{enumerate}
\end{proposition}

% As a result, we have bijective maps $\hat{P}^{(j-1)}(\cdot):\hat{A}(j,j-1,\ba)_{\neq\emptyset}\to \hat{\sP}^{(j-1)}$. Moreover, we have surjective maps $P^{(j)}(\cdot):\hat{A}(j,j-1,\ba)\times [a_j]\to \sP^{(j)}\cup \{\emptyset\}$  and the two maps satisfy the relation \eqref{eq:hatPconsistsofP(x,b)}. Note that $\{P^{(j)}(\hat{\bx},b) : \hat{\bx}\in \hat{A}(j,j-1,\ba), b\in [a_j]\}$ forms a partition of $\cK_{j}(\sP^{(1)})$ which may contain several empty classes and  $\{\cK_{j}(\hat{P}^{(j-1)}(\hat{\bx})):\hat{\bx}\in \hat{A}(j,j-1,\ba)\}$ also forms a partition of $\cK_{j}(\sP^{(1)})$. In addition, $\{P^{(j)}(\hat{\bx},b) : \hat{\bx}\in \hat{A}(j,j-1,\ba), b\in [a_j]\}\prec \{\cK_{j}(\hat{P}^{(j-1)}(\hat{\bx}):\hat{\bx}\in \hat{A}(j,j-1,\ba)\}$.
%Also we have \begin{align*} |\hat{\sP}^{(j-1)}| \leq |\hat{A}(j,j-1,\ba)| = \prod_{i=1}^{j-1} a_i^{\binom{j}{i}} \end{align*} This implies that $|\hat{\sP}^{(j-1)}|\leq \prod_{i=1}^{j-1} T^{\binom{j}{i} }\leq T^{2^j}$ if $\sP(k-1,\ba)$ is $T$-bounded.

 We remark that the counting lemma (see Lemma~\ref{lem: counting}) will enable us to restrict our attention to families of partitions as in Proposition~\ref{prop: hat relation}(ix). This is formalized in Lemma~\ref{lem: maps bijections}.
%In fact, in our paper, we will only deal with family of partitions that are `equitable', which is a term we define in the next subsection.
%In those `equitable' families of partitions, we have $\cK_{j}(\hat{P}^{(j-1)}(\hat{\bx}))\neq\emptyset$ for every $\hat{\bx}\in \hat{A}(j,j-1,\ba)$.\COMMENT{This can be ensured from Lemma~\ref{lem: counting} when the parameter $\epsilon$ for regularity in `equitability' is small enough.}   Note that in this case,  $\hat{P}^{(j-1)}(\cdot):\hat{A}(j,j-1,\ba)\to\hat{\sP}^{(j-1)}$ and $P^{(j)}(\cdot):\hat{A}(j,j-1,\ba)\times [a_j]\to \sP^{(j)}$ are bijections.%

For $j\in [k-1]$, $\ell \geq j+1$ and for each $\hat{\bx} \in \hat{A}(\ell,j,\ba)$, 
we define the \emph{polyad of $\hat{\bx}$} by 
\begin{align}\label{eq: larger polyad def}
\hat{P}^{(j)}(\hat{\bx}) 
:= \bigcup_{\hat{\by}\leq_{j+1,j} \hat{\bx}} \hat{P}^{(j)}(\hat{\by})
\stackrel{(\ref{eq:hatPconsistsofP(x,b)})}{=}\bigcup_{\hat{\bz}\leq_{j,j-1} \hat{\bx}} P^{(j)}(\hat{\bz},\bx^{(j)}_{\bz^{(1)}_* }).
\end{align}\COMMENT{Here, we have this because
$\hat{\bz}\leq_{j,j-1} \hat{\by} \leq_{j+1,j} \hat{\bx}$ implies $\hat{\bz}\leq_{j,j-1}\hat{\bx}$.}
(Note that this generalizes the definition made in \eqref{eq: hat P def} for the case $\ell=j+1$.)
The following fact follows easily from the definition.
\begin{proposition}\label{prop:unique}
Let $\sP= \sP(k-1,\ba)$ be a family of partitions. 
Let $j\in [k-1]$ and $\ell\geq j+1$. Then for every $L\in \cK_{\ell}(\sP^{(1)})$, 
there exists a unique $\hat{\bx}\in \hat{A}(\ell,j,\ba)$ such that $L\in \cK_{\ell}(\hat{P}^{(j)}(\hat{\bx}))$.
\end{proposition}
\COMMENT{
\begin{proof}
Indeed, recall that for every $J\in \binom{L}{j+1}$, 
the vector $\hat{\bx}(J)$ is the unique element of $\hat{A}(j+1,j,\ba)$ which satisfies $J\in \cK_{j+1}(\hat{P}^{(j)}(\hat{\bx}(J)))$. 
Let $\hat{\bx}\in \hat{A}(\ell,j,\ba)$ be the vector which satisfies $\bx^{(i)}_{I'}= (\bx^{(i)}(J))_{I'}$ for every $I'\subsetneq J$ with $|I'|=i\leq j$ and every $J\in \binom{L}{j+1}$.\COMMENT{Here, subscript $I'$ denotes the coordinate of the vector.}
Then it is easy to see that the vector $\hat{\bx}= (\bx^{(1)},\dots, \bx^{(j)})$ is well-defined and $\hat{\bx}(J)\leq_{j+1,j} \hat{\bx}$ for all $J\in \binom{L}{j+1}$.
Thus by \eqref{eq: larger polyad def}, every $J\in \binom{L}{j+1}$ satisfies $J\in \cK_{j+1}(\hat{P}^{(j)}(\hat{\bx}))$, and hence $L\in \cK_{\ell}(\hat{P}^{(j)}(\hat{\bx}))$. 
This proves our claim. %Thus this shows that for every $L\in \cK_{\ell}(\sP^{(1)})$, there exists unique $\hat{\bx}\in \hat{A}(\ell,j,\ba)$ such that $L\in \cK_{\ell}(\hat{P}^{(j)}(\hat{\bx}))$.
\end{proof}
}
Note that \eqref{eq: Pj operator define} and \eqref{eq: larger polyad def} together imply that, for all $j\in [k-1]$ and $\hat{\bx}\in \hat{A}(j+1,j,\ba)$,
\begin{align}\label{eq: complex definition by address}
\hat{\cP}(\hat{\bx}) := \left\{ \bigcup_{\hat{\by}\leq_{j+1,i} \hat{\bx}} \hat{P}^{(i)}(\hat{\by})\right\}_{i\in [j]}
\end{align} is a $(j+1,j)$-complex.\COMMENT{
If it is a complex, then it is easy to see that it is a $(j+1,j)$-complex.
Suppose $\hat{\by} \leq_{j+1,i} \hat{\bx}$ for $i\leq j$.
It suffice to show that if $J\in \hat{P}^{(i)}(\hat{\by})$, then $J\in \cK_{i}( \hat{P}^{(i-1)}(\hat{\bz}))$ for some $\hat{\bz}$ with $\hat{\bz} \leq_{j+1,i-1} \hat{\bx}$. \newline
Suppose $J\in \hat{P}^{(i)}(\hat{\by})$.
By \eqref{eq: larger polyad def}, $J\in P^{(i)}(\hat{\bz}, \by^{(i)}_{\bz^{(1)}_*})$ for some $\hat{\bz}$ with $\hat{\bz}\leq_{i,i-1} \hat{\by}$.
By the definition of $ P^{(i)}(\hat{\bz}, \by^{(i)}_{\bz^{(1)}_*})$, $J$ belongs to $\cK_{i}(\hat{P}^{(i-1)}(\hat{\bz}))$. Thus we get what we want.
}
Moreover, using Proposition~\ref{prop: hat relation}(iii) it is easy to check that for each $\hat{\bx}\in \hat{A}(j+1,j,\ba)$ with $\cK_{j+1}(\hat{P}^{(j)}(\hat{\bx}))\neq \emptyset$, we have (for $\hat{\cP}(J)$ as defined in~\eqref{eq: hat cP})
\begin{align}\label{eq: hat cP bx is hat cP J is for some J}
\hat{\cP}(\hat{\bx}) = \hat{\cP}(J)\text{ for some } J\in \cK_{j+1}(\sP^{(1)}).
\end{align}

\medskip

\subsubsection{Constructing families of partitions using the address space}
On several occasions we will construct $P^{(j)}(\hat{\bx},b)$ and $\hat{P}^{(j)}(\hat{\bx})$ first and 
then  show that they actually give rise to a family of partitions
for which we can use the properties listed in Proposition~\ref{prop: hat relation}.
The following lemma,
which can easily be proved by induction, 
provides a criterion to show that this is indeed the case.

\begin{lemma}\label{lem: family of partitions construction}
Suppose $k\in \N\sm\{1\}$ and $\ba\in \N^{k-1}$. Suppose $\sP^{(1)}= \{V_1,\dots, V_{a_1}\}$ is a partition of a vertex set $V$.
Suppose that for each $j\in [k-1]\setminus\{1\}$ and each $(\hat{\bx},b)\in \hat{A}(j,j-1,\ba) \times [a_j]$, 
we are given a $j$-graph $P'^{(j)}(\hat{\bx},b)$, and
for each $j\in [k]\setminus \{1\}$ and $\hat{\bx}\in \hat{A}(j,j-1,\ba)$,
we are given a $(j-1)$-graph $\hat{P'}^{(j-1)}(\hat{\bx})$. 
Let 
\begin{align*} P'^{(1)}(b,b)&:=V_b \text{ for all } b\in [a_1], \text{ and }\\ 
   \sP^{(j)}&:= \{P'^{(j)}(\hat{\bx},b) :(\hat{\bx},b)\in \hat{A}(j,j-1,\ba) \times [a_j]\} \text{ for all } j \in [k-1]\setminus\{1\}.
   \end{align*}
Suppose the following conditions hold: 
\begin{enumerate}[label=\rm (FP\arabic*)]
\item\label{item:FP1} $P'^{(1)}(b,b) \neq \emptyset$ for each $b\in [a_1]$; moreover for each $j\in [k-1]\setminus \{1\}$ and each $(\hat{\bx},b)\in \hat{A}(j,j-1,\ba) \times [a_j]$, we have $P'^{(j)}(\hat{\bx},b)\neq \emptyset$.

\item\label{item:FP2} For each $j\in [k-1]\setminus\{1\}$ and $\hat{\bx}\in \hat{A}(j,j-1,\ba)$,
the set $\{P'^{(j)}(\hat{\bx},b):b\in [a_j]\}$ has size $a_j$\COMMENT{
Need this to get that the $P'^{(j)}(\hat{\bx},b)$ are distinct.} and forms a partition of $\cK_j(\hat{P'}^{(j-1)}(\hat{\bx}))$.
\item\label{item:FP3} For each $j\in [k-1]$ and $\hat{\bx}\in \hat{A}(j+1,j,\ba)$, we have 
$$\hat{P'}^{(j)}(\hat{\bx})= 
\bigcup_{\hat{\by}\leq_{j,j-1} \hat{\bx}} P'^{(j)}(\hat{\by},\bx^{(j)}_{\by^{(1)}_* }).$$
\end{enumerate}
Then the $\ba$-labelling $\mathbf{\Phi} = \{\phi^{(i)}\}_{i=2}^{k-1}$ given by $\phi^{(i)}(P'^{(i)}(\hat{\bx},b))=b$ for each $(\hat{\bx},b)\in \hat{A}(i,i-1,\ba)\times[a_i]$ is well-defined and satisfies the following:
\begin{enumerate}[label={\rm (FQ\arabic*)}]
\item\label{item:FQ1} $\sP = \{\sP^{(i)}\}_{i=1}^{k-1}$ is a family of partitions on $V$.
\item\label{item:FQ2} The maps $P^{(j)}(\cdot,\cdot)$ and $\hat{P}^{(j)}(\cdot)$ defined in \eqref{eq: hat P def}--\eqref{eq: hat P P relations emptyset} for $\sP$, $\mathbf{\Phi}$ satisfy that for each $j\in [k-1]\setminus\{1\}$ and $(\hat{\bx},b)\in \hat{A}(j,j-1,\ba) \times [a_j]$, we have 
$$P^{(j)}(\hat{\bx},b) = P'^{(j)}(\hat{\bx},b),$$ 
and for each $j \in [k-1]$ and $\hat{\bx} \in \hat{A}(j+1,j,\ba)$ we have
$$\hat{P}^{(j)}(\hat{\bx}) = \hat{P'}^{(j)}(\hat{\bx}).$$
\end{enumerate}
\end{lemma}

\subsubsection{Density functions of address spaces}\label{sec: subsub density function}

For $k\in \N\sm\{1\}$ 
and $\ba\in \N^{k-1}$,
we say a function $d_{\ba,k}:\hat{A}(k,k-1,\ba)\rightarrow[0,1]$ is a \emph{density function} of $\hat{A}(k,k-1,\ba)$.
For two density functions $d^1_{\ba,k}$ and $d^2_{\ba,k}$,
we define the \emph{distance} between $d^1_{\ba,k}$ and $d^2_{\ba,k}$ by
$$\dist(d^1_{\ba,k}, d^2_{\ba,k}) 
:= k!\prod_{i=1}^{k-1} a_i^{-\binom{k}{i}}\sum_{\hat{\bx}\in \hat{A}(k,k-1,\ba)} |d^1_{\ba,k}(\hat{\bx})-d^2_{\ba,k}(\hat{\bx})| .$$
Since $|\hat{A}(k,k-1,\ba)| = \binom{a_1}{k} \prod_{i=2}^{k-1} a_i^{\binom{k}{i}}$, we always have that $\dist(d^1_{\ba,k}, d^2_{\ba,k})\leq 1.$
Suppose we are given a density function $d_{\ba,k}$, a real $\epsilon>0$, and a $k$-graph $H^\kk$.
\begin{equation*}
\begin{minipage}[c]{0.9\textwidth}\em 	
We say a family of partitions $\sP=\sP(k-1,\ba)$ on $V(H^{(k)})$ is an \emph{$(\epsilon,d_{\ba,k})$-partition} of $H^{(k)}$ 
if for every $\hat{\bx}\in \hat{A}(k,k-1,\ba)$ the $k$-graph
$H^{(k)}$ is $(\epsilon,d_{\ba,k}(\hat{\bx}))$-regular with respect to $\hat{P}^{(k-1)}(\hat{\bx})$. %(Recall that $H^{(k)}[\hat{\bx}^{(1)}_*] = H^{(k)}[ \bigcup_{s\in \hat{\bx}^{(1)}_*} V_s].$)
If $\sP$ is also $(1/a_1,\epsilon,\ba)$-equitable (as specified in Definition~\ref{def: equitable family of partitions}), 
we say $\sP$ is an \emph{$(\epsilon,\ba,d_{\ba,k})$-equitable partition} of $H^{(k)}$.	
\end{minipage}
\end{equation*}
 Note that
\begin{equation}\label{eq: perfectly regular is regular}
\begin{minipage}[c]{0.9\textwidth}\em 
 if $\hat{P}^{(k-1)}(\cdot): \hat{A}(k,k-1,\ba)\rightarrow \hat{\sP}^{(k-1)}$ is a bijection, then $H^{(k)}$ is perfectly $\epsilon$-regular with respect to $\sP$ if and only if there exists a density function $d_{\ba,k}$ such that $\sP$ is an $(\epsilon,d_{\ba,k})$-partition of $H^{(k)}$.
\end{minipage}
\end{equation}

\subsection{Regularity instances}\label{sec: regularity instances}
A regularity instance $R$ encodes an address space, an associated density function and a regularity parameter.
Roughly speaking,
a regularity instance can be thought of as encoding a weighted `reduced multihypergraph' obtained from an application of the regularity lemma for hypergraphs.
To formalize this, we fix a function $\epsilon_{\ref{def: regularity instance}}(\cdot,\cdot) : \mathbb{N} \times \mathbb{N} \rightarrow (0,1]$ which satisfies the following.
\begin{itemize}
\item $\epsilon_{\ref{def: regularity instance}}( \cdot, k)$ is a decreasing function for any fixed  $k\in \mathbb{N}$ with $\lim_{t\rightarrow \infty} \epsilon_{\ref{def: regularity instance}}(t, k) = 0$, 
\item $\epsilon_{\ref{def: regularity instance}}( t, \cdot)$ is a decreasing function for any fixed $t\in \mathbb{N}$,
\item $\epsilon_{\ref{def: regularity instance}}(t,k) < t^{- 4^k} \epsilon_{\ref{lem: counting}}(1/t,1/t,k-1,k)/4$, where $\epsilon_{\ref{lem: counting}}
$ is defined in Lemma~\ref{lem: counting}.
\end{itemize}\COMMENT{Here, note that we divide $\epsilon_{\ref{lem: counting}}(1/t,1/t,k-1,k)$ by $4 t^{4^k}$. There are two reason for this. We divide it by three because
since we want to apply Lemma~\ref{lem: counting} even when $(\epsilon,\ba,d_{\ba,k})$ is not a regularity instance but  $(\epsilon/3,\ba,d_{\ba,k})$ is a regularity instance. For this purposes, in some lemmas we assume $(\epsilon/3, \ba, d_{\ba,k})$ is a regularity instance, instead of $(\epsilon,\ba,d_{\ba,k})$ is a regularity instance.
We also divide it by $t^{4^{k}}$ to make sure that $\epsilon $ is much less than $1/t$.
}
In fact, there are many choices of such functions, we fix one such function $\epsilon_{\ref{def: regularity instance}}$ throughout the paper. In particular, if we write $0<\epsilon \ll 1/k, 1/t$, it is often implied that $\epsilon < \epsilon_{\ref{def: regularity instance}}(t,k)$. With this fixed choice of the function, we define regularity instances as follows.

\begin{definition}[Regularity instance]\label{def: regularity instance}
A \emph{regularity instance} $R= (\epsilon,\ba,d_{\ba,k})$ 
is a triple, where
$\ba = (a_1,\dots, a_{k-1})\in \N^{k-1}$ with $0<\epsilon \leq \epsilon_{\ref{def: regularity instance}}(\|\ba\|_\infty,k)$,\COMMENT{Here, the condition that $\epsilon \ll \epsilon_{\ref{def: regularity instance}}(t)$ is required. If $\epsilon$ is too big compare to $1/t$, we lose the tack of $|\cK_{j}(\hat{P}(\hat{\bx}))|$ for each $\hat{\bx}\in \hat{A}(j,j-1,\ba)$. Then it causes a problem in the proof of Lemma~\ref{lem: slightly different partition regularity} } 
and $d_{\ba,k}$ is a density function of $\hat{A}(k,k-1,\ba)$. 
A $k$-graph $H$ satisfies 
the regularity instance $R$ if there exists a family of partitions $\sP=\sP(k-1,\ba)$ such that 
$\sP$ is an $(\epsilon,\ba,d_{\ba,k})$-equitable partition of~$H$.
The \emph{complexity} of $R$ is $1/\epsilon$.
\end{definition}

Since $\epsilon_{\ref{def: regularity instance}}(\norm{\ba},k)$ depends only on $\norm{\ba}$ and $k$, it follows that for given $r$ and fixed $k$, the number of vectors $\ba$ which could belong to a regularity instance $R$ with complexity $r$ is bounded by a function of $r$.

Note that in the above definition, we have chosen $\epsilon_{\ref{def: regularity instance}}(t,k)$ to be sufficiently small so that we can apply several lemmas to the hypergraph $H$ whenever it satisfies a regularity instance. In fact, in many cases, these applications remain valid even if $\epsilon$ is slightly larger. For example, Lemmas~\ref{lem: maps bijections}, \ref{lem: improve regularity partition}, and~\ref{lem: slightly different partition regularity} allow for the case where $\epsilon \leq 3\cdot \epsilon_{\ref{def: regularity instance}}(t,k)$—that is, when $R = (\epsilon, \ba, d_{\ba,k})$ is close to a regularity instance, but may not exactly be one. Later, we will modify a hypergraph that satisfies a regularity instance $(\epsilon, \ba, d_{\ba,k})$, resulting in a new hypergraph that satisfies $(\epsilon', \ba, d_{\ba,k})$ with a larger $\epsilon'$. Although this modified triple may no longer be a regularity instance, the buffer provided by our choice of $\epsilon_{\ref{def: regularity instance}}(t,k)$ ensures that the lemmas can still be applied.

\begin{definition}[Regular reducible]\label{def: regular reducible}
A $k$-graph property $\bP$ is \emph{regular reducible} if for any $\beta>0$, 
there exists an $r=r_{\ref{def: regular reducible}}(\beta,\bP)$ such that 
for any integer $n\geq k$, 
there is a family $\cR=\cR(n,\beta,\bP)$ of at most $r$ regularity instances, each of complexity at most $r$,
such that the following hold for every $\alpha>\beta$ and every $n$-vertex $k$-graph $H$:
\begin{itemize}
\item If $H$ satisfies $\bP$, then there exists $R\in \cR$ such that $H$ is $\beta$-close to satisfying $R$.
\item If $H$ is $\alpha$-far from satisfying $\bP$, then for any $R\in \cR$  the $k$-graph $H$ is $(\alpha-\beta)$-far from satisfying $R$.
\end{itemize}
\end{definition}

Thus a property is regular reducible if it can be (approximately) encoded by a bounded number of regularity instances of bounded complexity.
We will often make use of the fact that if we apply the regular approximation lemma (Theorem~\ref{thm: RAL}) to a $k$-graph $H$
to obtain $G$ and $\sP$,
then $\ba^\sP$ together with the densities of $G$ with respect to the polyads in $\hat{\sP}^{(k-1)}$ naturally give rise to a regularity instance $R$
where $G$ satisfies $R$ and $H$ is close to satisfying $R$.

Note that different choices of $\epsilon_{\ref{def: regularity instance}}$ lead to a different definition of regularity instances and thus might lead to a different definition of being regular reducible. However, our main result implies that for \emph{any} appropriate choice of $\epsilon_{\ref{def: regularity instance}}$, being regular reducible and testability are equivalent. In particular, if a property is regular reducible for an appropriate choice of $\epsilon_{\ref{def: regularity instance}}$, then it is regular reducible for all appropriate choices of $\epsilon_{\ref{def: regularity instance}}$, and so `regular reducibility' is well defined.

\section{Hypergraph regularity: counting lemmas and approximation}\label{sec:counting}

In this section we present several results about hypergraph regularity. 
The first few results are simple observations which follow either from the definition of $\epsilon$-regularity or can be easily proved by standard probabilistic arguments.
We omit the proofs.
In Section~\ref{sec:3.2}
we introduce an induced version of the `counting lemma' that is suitable for our needs (see Lemma~\ref{lem: counting hypergraph}).

In Section~\ref{sec:3.3} we show that for every $k$-graph $H$, 
there is a $k$-graph $G$ that is close to $H$ and has better regularity parameters.
As a qualitative statement of this is trivial, the crucial point of our statement is the exact relation of the parameters. In Section~\ref{sec:3.4} we make two simple observations on refinements of partitions and in Section~\ref{sec:3.5} we consider small perturbations of a given family of partitions. In Section~\ref{sec:3.6} we relate the distance between two $k$-graphs and the distance of their density functions.

\subsection{Hypergraph regularity results}
We will use the following results which follow easily from the definition of hypergraph regularity (see Section~\ref{sec: 2 hypergraph regularity}).

\begin{lemma}\label{lem: simple facts 1}
Suppose $m \in \mathbb{N}$,  $0< \epsilon \leq \alpha^2 <1$ and $d\in [0,1]$.
Suppose $H^{(k)}$ is an $(m,k,k,1/2)$-graph which is $(\epsilon,d)$-regular  with respect to an $(m,k,k-1,1/2)$-graph $H^{(k-1)}$.
Suppose $Q^{(k-1)} \subseteq H^{(k-1)}$ and $H'^{(k)}\subseteq H^{(k)}$ such that $|\cK_k(Q^{(k-1)})| \geq \alpha |\cK_k(H^{(k-1)})|$ and $H'^{(k)}$ is $(\epsilon,d')$-regular with respect to $H^{(k-1)}$ for some $d'\leq d$.  Then 
\begin{enumerate}[label=\rm(\roman*)]
\item $\cK_k(H^{(k-1)}) \setminus H^{(k)}$ is $(\epsilon,1-d)$-regular with respect to $H^{(k-1)}$, 
\item $H^{(k)}$ is $(\epsilon/\alpha,d)$-regular with respect to $Q^{(k-1)}$, and 
\item $H^{(k)}\setminus H'^{(k)}$ is $(2\epsilon,d-d')$-regular with respect to $H^{(k-1)}$.
\end{enumerate}
\end{lemma}

\begin{lemma}\label{lem: simple facts 2}
Suppose $m \in \mathbb{N}$, $0< \epsilon \leq 1/100$, $d\in [0,1]$ and $\nu \leq \epsilon^{10}$.
Suppose $H^{(k)}$ and $G^{(k)}$ are $(m,k,k,1/2)$-graphs on $\{V_1,\dots, V_k\}$, and
$H^{(k-1)}$ and $G^{(k-1)}$ are $(m,k,k-1,1/2)$-graphs on $\{V_1,\dots, V_k\}$.
Suppose $H^{(k)}$ is $(\epsilon,d)$-regular with respect to $H^{(k-1)}$. 
If $|\cK_k(H^{(k-1)})| \geq \nu^{1/2} m^{k}$, $|H^{(k)}\triangle G^{(k)}| \leq \nu m^k$ and $|H^{(k-1)}\triangle G^{(k-1)}| \leq \nu m^{k-1}$, then $G^{(k)}$ is $(\epsilon+ \nu^{1/3}, d)$-regular with respect to $G^{(k-1)}$. 
\end{lemma}

\begin{lemma}\label{lem: union regularity}
Suppose $0<\epsilon \ll 1/k,1/s.$
Suppose that $H^{(k)}_1,\dots, H^{(k)}_s$ are edge-disjoint $(k,k,*)$-graphs such that each $H^{(k)}_i$ is $\epsilon$-regular with respect to a $(k,k-1,*)$-graph $H^{(k-1)}$. 
Then $\bigcup_{i=1}^{s} H^{(k)}_i$ is $s\epsilon$-regular with respect to $H^{(k-1)}$.
\end{lemma}

We will also use the following observation (see for example \cite{RS07}), which can be easily proved using Chernoff's inequality.

\begin{lemma}[Slicing lemma \cite{RS07}]\label{lem: slicing}
Suppose $0< 1/m \ll d, \epsilon , p_0, 1/s$ and $d\geq 2\epsilon$. Suppose that
\begin{itemize}
\item $H^{(k)}$ is an $(\epsilon,d)$-regular $k$-graph with respect to a $(k-1)$-graph $H^{(k-1)}$,
\item $|\cK_{k}(H^{(k-1)})| \geq  m^{k}/\log m$,
\item $p_1,\dots, p_s \geq p_0$ and $\sum_{i=1}^{s} p_i \leq 1$.
 \end{itemize}
 Then 
there exists a partition $\{H^{(k)}_0,H^{(k)}_1,\dots, H^{(k)}_s\}$ of $H^{(k)}$ such that $H^{(k)}_i$ is $(3\epsilon,p_id)$-regular with respect to $H^{(k-1)}$ for every $i\in [s]$, and $H^{(k)}_0$ is $(3\epsilon,(1-\sum p_i)d)$-regular with respect to $H^{(k-1)}$.
\end{lemma}

\subsection{Counting lemmas}\label{sec:3.2}
Kohayakawa, R\"odl and Skokan proved the following `counting lemma' (Theorem~6.5 in~\cite{KRS02}), which asserts that the number of copies of a given $K^{(k)}_{\ell}$ in an $(\epsilon,\bd)$-regular complex is close to what one could expect in a corresponding random complex. We will deduce several versions of this which suit our needs.

\begin{lemma}[Counting lemma  \cite{KRS02}]\label{lem: counting}
For all $ \gamma, d_0>0$ and  $k, \ell \in \N\sm \{1\}$ with $k\leq \ell$,\COMMENT{We don't use the hierarchy notation here since we refer to $\epsilon_{\ref{lem: counting}}$ in the definition of regularity instance.}
there exist
$\epsilon_0 := \epsilon_{\ref{lem: counting}}(\gamma,d_0,k,\ell)\leq 1$ and $m_0:= n_{\ref{lem: counting}}(\gamma,d_0,k,\ell)$
such that the following holds:
Suppose $0\leq \lambda <1/4$.
Suppose 
$0<\epsilon \leq \epsilon_0$ and $m_0 \leq m$ and
$\bd=(d_2,\dots, d_{k})\in \mathbb{R}^{k-1}$ such that 
$d_j\geq d_0$ for every $j\in [k]\setminus\{1\}$.
Suppose that
$\cH= \{H^{(j)}\}_{j=1}^{k}$ is an $(\epsilon,\bd)$-regular $(m,\ell,k,\lambda)$-complex, 
and $H^{(1)}=\{V_1,\ldots,V_\ell\}$ with $m_i=|V_i|$ for every $i\in [\ell]$.
Then 
$$|\cK_{\ell}(H^{(k)})| = (1\pm \gamma) \prod_{j=2}^{k}d_j^{\binom{\ell}{j}} \cdot \prod_{i=1}^{\ell} m_i.$$
\end{lemma}
\COMMENT{It is possible to delete the condition $1/m \ll \epsilon$ from all counting lemmas.}

Recall that equitable families of partitions were defined in Section~\ref{sec: partitions of hypergraphs and RAL}. 
Based on the counting lemma, 
it is easy to show that for an equitable family of partitions $\sP$ and an $\ba$-labelling $\mathbf{\Phi}$, 
the maps $\hat{P}^{(j-1)}(\cdot):\hat{A}(j,j-1,\ba)\to\hat{\sP}^{(j-1)}$ and $P^{(j)}(\cdot,\cdot):\hat{A}(j,j-1,\ba)\times [a_j]\to \sP^{(j)}$ defined in Section~\ref{sec: address space} are bijections.
We will frequently make use of this fact in subsequent sections,
often without referring to Lemma~\ref{lem: maps bijections} explicitly.

\begin{lemma}\label{lem: maps bijections}
Suppose that $k,t \in \mathbb{N}\setminus \{1\}$, $0\leq \lambda < 1/4$ and $\epsilon/3 \leq \epsilon_{\ref{def: regularity instance}}(t,k)$ and $\ba = (a_1,\dots, a_{k-1}) \in [t]^{k-1}$ and $|V|=n$ with $1/n \ll 1/t, 1/k$.\COMMENT{In other words,  $R=(\epsilon/3, \ba, d_{\ba,k})$ is a regularity instance for any density function $d_{\ba,k}$.}  If $\sP = \sP(k-1,\ba)$ is a $(1/a_1,  \epsilon,\ba,\lambda)$-equitable family of partitions on $V$, and $\sP$ with an $\ba$-labelling $\mathbf{\Phi}$ defines maps $\hat{P}^{(j-1)}(\cdot)$ and $P^{(j-1)}(\cdot,\cdot)$, then the following hold.
\begin{enumerate}[label={\rm(\roman*)}]
\item For each $j\in [k-1]$, $\hat{P}^{(j)}(\cdot):\hat{A}(j+1,j,\ba)\to\hat{\sP}^{(j)}$ is a bijection and if $j>1$, then $P^{(j)}(\cdot,\cdot):\hat{A}(j,j-1,\ba)\times [a_j]\to \sP^{(j)}$ is also a bijection. In particular, 
$\hat{A}(j,j-1,\ba) = \hat{A}(j,j-1,\ba)_{\neq \emptyset}$.
\item For each $j\in [k-1]\setminus \{1\}$ and $\hat{\bx} \in \hat{A}(j+1,j,\ba)$, 
$\hat{\cP}(\hat{\bx})$ is an $(\epsilon, (1/a_2,\dots, 1/a_{j}))$-regular $(j+1,j,\lambda)$-complex.
\end{enumerate}
\end{lemma}

Let us introduce the necessary notation which will be convenient to count induced copies of a given hypergraph $F$.

Suppose $k,\ell\in \mathbb{N}\sm\{1\}$ such that $\ell\geq k$
and suppose $\ba \in \N^{k-1}$.
Suppose that $d_{\ba,k}:\hat{A}(k,k-1,\ba)\to[0,1]$ is a density function.
Suppose $F$ is a $k$-graph on $\ell$ vertices.
Suppose $\hat{\bx}\in \hat{A}(\ell,k-1,\ba)$
and $\sigma:V(F)\to\bx^{(1)}_*$ is a bijection.
Let $A(F)$ be the size of the automorphism group of $F$.
We now define three functions in terms of the parameters above that will estimate the number of induced copies of $F$ in certain parts of an $\epsilon$-regular $k$-graph.
Let
\begin{align*}
	IC(F,d_{\ba,k},\hat{\bx},\sigma) &:=  \prod_{\substack{\hat{\by}\leq_{k,k-1} \hat{\bx}, \\ \by^{(1)}_*\in \sigma(F)}} d_{\ba,k}(\hat{\by}) \prod_{\substack{\hat{\by}\leq_{k,k-1} \hat{\bx}, \\ \by^{(1)}_*\notin \sigma(F)}}(1-d_{\ba,k}(\hat{\by}))  \prod_{j=2}^{k-1}a_j^{-\binom{\ell}{j}},\\
IC(F,d_{\ba,k},\hat{\bx}) &:=  \frac{1}{A(F)} \sum_{\sigma} IC(F,d_{\ba,k},\hat{\bx},\sigma),\\
IC(F,d_{\ba,k}) &:= \binom{a_1}{\ell}^{-1}\sum_{\hat{\bx} \in \hat{A}(\ell,k-1,\ba)} IC(F,d_{\ba,k},\hat{\bx}).
\end{align*}

We will now show that for a $k$-graph $H$ satisfying a suitable regularity instance $R=(\epsilon,\ba,d_{\ba,k})$,
the value $IC(F,d_{\ba,k})$ is a very accurate estimate for $\mathbf{Pr}(F, H)$ (recall the latter was introduced in Section~\ref{sec: basic notation}). The proof of this lemma is given in \cite{JKKO2}.
%The same is true if $F$ is replaced by a finite family of $k$-graphs (see Corollary~\ref{cor: counting collection}).

\begin{lemma}[Induced counting lemma for general hypergraphs]\label{lem: counting hypergraph}
Suppose $0<1/n\ll \epsilon \ll 1/t, 1/a_1 \ll \gamma,  1/k, 1/\ell$ with $2\leq k \leq \ell$.
Suppose $F$ is an $\ell$-vertex $k$-graph and $\ba\in [t]^{k-1}$.
Suppose $H$ is an $n$-vertex $k$-graph satisfying a regularity instance $R= (\epsilon,\ba,d_{\ba,k})$.
Then 
$$\mathbf{Pr}(F, H)= IC(F,d_{\ba,k}) \pm \gamma.$$
\end{lemma}

In the previous lemma we counted the number of induced copies of a single $k$-graph $F$ in $H$.
It is not difficult to extend this approach to a finite family of $k$-graphs.
For a finite family $\cF$ of $k$-graphs, 
we define 
\begin{align}\label{eq: def IC cF d a k}
IC(\cF, d_{\ba,k}):= \sum_{F\in \cF} IC(F,d_{\ba,k}).
\end{align}

\begin{corollary}\label{cor: counting collection}
Suppose $0<1/n\ll \epsilon \ll 1/t, 1/a_1 \ll \gamma,  1/k, 1/\ell$ with $2\leq k \leq \ell$.
Let $\cF$ be a collection of $k$-graphs on $\ell$ vertices.
Suppose $H$ is an $n$-vertex $k$-graph satisfying a regularity instance $R= (\epsilon,\ba,d_{\ba,k})$ where $\ba\in [ t]^{k-1}$. Then 
$$\mathbf{Pr}(\cF, H)=IC(\cF, d_{\ba,k}) \pm \gamma.$$
\end{corollary}
\begin{proof}
For each $F\in \cF$, we apply Lemma~\ref{lem: counting hypergraph} with $\gamma/2^{\binom{\ell}{k}}$ playing the role of $\gamma$.
As $|\cF|\leq 2^{\binom{\ell}{k}}$, this completes the proof.
\end{proof}

\subsection{Close hypergraphs with better regularity}\label{sec:3.3}

In this subsection we present several results that show in their simplest form that for an $(\epsilon+\delta)$-regular $k$-graph $H$,
there is a $k$-graph $G$ on the same vertex set such that $|H\triangle G|$ is small and $G$ is $\epsilon$-regular.
The key point is that we seek an additive improvement in the regularity parameter.
Our approach in the following sections relies heavily on this (see the proof of Lemmas~\ref{lem:7.1} and~\ref{lem:7.2}).

\begin{lemma}\label{lem: improve regularity}
Suppose $0< 1/m\ll \delta \ll \nu \ll \epsilon, 1/k, 1/s, d_0 \leq 1/2$. 
Suppose that $d_i\geq d_0$ for all $i\in [s]$ and $\sum_{i\in [s]}d_i=1$.
Suppose that
\begin{itemize}
	\item $H^{(k-1)}$ is an $(m,k,k-1)$-graph such that $|\cK_{k}(H^{(k-1)})|\geq  \epsilon m^{k}$,
	\item $H^{(k)}_1,\dots, H^{(k)}_s$ are $(m,k,k)$-graphs that form a partition of $\cK_{k}(H^{(k-1)})$, and that
	\item $H^{(k)}_i$ is $(\epsilon+\delta,d_i)$-regular with respect to $H^{(k-1)}$ for each $i\in [s]$. 
\end{itemize}
Then there exist $(m,k,k)$-graphs $G^{(k)}_1,\dots, G^{(k)}_s$ such that 
\begin{itemize}
\item[{\rm (G1)$_{\ref{lem: improve regularity}}$}] $G^{(k)}_1,\dots, G^{(k)}_s$ form a partition of $\cK_{k}(H^{(k-1)})$,
\item[{\rm (G2)$_{\ref{lem: improve regularity}}$}] $G^{(k)}_i$ is $(\epsilon,d_i)$-regular with respect to $H^{(k-1)}$ for each $i\in [s]$, and
\item[{\rm (G3)$_{\ref{lem: improve regularity}}$}] $|G^{(k)}_i\triangle H^{(k)}_i| \leq \nu m^{k}$ for each $i\in [s]$.
\end{itemize}
\end{lemma}

Roughly speaking, the idea of the proof is to construct $G_i^\kk$ from $H_i^\kk$ by randomly redistributing a small proportion of the $k$-edges of each $H_i^\kk$.

\begin{proof}[Proof of Lemma~\ref{lem: improve regularity}]
\COMMENT{`$1/m\ll \delta $' is unnecessary condition here.
We can just put $0< 1/m,\delta \ll \nu$ in the statement and add the following sentence at the beginning of the proof.
`By increasing the value of $\delta$ if necessary, we may assume that we have
$0< 1/m \ll \delta \ll \nu$.'
Note that $H^{(k)}_i$ is $(\epsilon+\delta,d_i)$-regular with respect to $H^{(k-1)}$ even after increasing the value of $\delta$.
}
First we claim that for  $Q^{(k-1)}\sub H^{(k-1)}$ with $|\cK_k(Q^{(k-1)})| \geq \epsilon |\cK_k(H^{(k-1)})|$, we have
\begin{eqnarray}\label{eq: Q density}
 d(H^{(k)}_i\mid Q^{(k-1)}) = d_i\pm (\epsilon+\delta^{1/2}).
 \end{eqnarray}
Observe that if $|\cK_k(Q^{(k-1)}) |\geq (\epsilon+\delta) |\cK_k(H^{(k-1)})|$, 
then this follows directly from the fact that $H_i^\kk$ is $(\epsilon+\delta,d_i)$-regular with respect to $H^{(k-1)}$.

So suppose that $|\cK_{k}(Q^{(k-1)})| < ( \epsilon+\delta) |\cK_{k}(H^{(k-1)})|$.
Choose a $(k-1)$-graph $Q'^{(k-1)}$ such that $Q^{(k-1)}\subseteq Q'^{(k-1)} \subseteq H^{(k-1)}$ as well as
$$|\cK_{k}(Q'^{(k-1)})| \geq (\epsilon+\delta) |\cK_{k}(H^{(k-1)})| \text{ and } |\cK_{k}(Q'^{(k-1)})\setminus \cK_{k}(Q^{(k-1)})|\leq 2\delta |\cK_{k}(H^{(k-1)})|.$$
Then for each $i\in [s]$, we conclude
\begin{align*}%\label{eq: Q Q'}
|H^{(k)}_i\cap \cK_{k}(Q^{(k-1)})| = |H^{(k)}_i\cap \cK_{k}(Q'^{(k-1)})|\pm 2\delta|\cK_{k}(H^{(k-1)})|.
\end{align*}
The fact that $H^{(k)}_i$ is $(\epsilon+\delta,d_i)$-regular with respect to $H^{(k-1)}$ implies that 
\begin{align*}%\label{eq: Q' density}
|H^{(k)}_i\cap \cK_{k}(Q'^{(k-1)})| = (d_i\pm (\epsilon +\delta))|\cK_{k}(Q'^{(k-1)})|.
\end{align*}
From this we obtain \eqref{eq: Q density}.\COMMENT{
\begin{eqnarray*}
 d(H^{(k)}_i\mid Q^{(k-1)}) &=& \frac{ |H^{(k)}_i\cap \cK_{k}(Q^{(k-1)})|}{ |\cK_{k}(Q^{(k-1)})|} 
= \frac{ |H^{(k)}_i\cap \cK_{k}(Q'^{(k-1)})|\pm 2\delta|\cK_{k}(H^{(k-1)})| }{ |\cK_{k}(Q'^{(k-1)})| \pm 2\delta|\cK_{k}(H^{(k-1)})| }\nonumber\\
&= &\frac{ (d_i\pm (\epsilon +\delta))|\cK_{k}(Q'^{(k-1)})| \pm 2\delta|\cK_{k}(H^{(k-1)})| }{|\cK_{k}(Q'^{(k-1)})| \pm 2\delta|\cK_{k}(H^{(k-1)})|}\nonumber\\
&=& d_i\pm (\epsilon+\delta^{1/2}).
\end{eqnarray*}
We obtain the final equality since $\delta \ll \epsilon \leq |\cK_{k}(Q'^{(k-1)})|/|\cK_{k}(H^{(k-1)})|$.
}
\noindent Our next step is to construct suitable random sets $\cA,\cB_1,\ldots, \cB_s\sub \cK_k(H^{(k-1)})$.
We define $G_i^\kk$ based on these sets and show that (G1)$_{\ref{lem: improve regularity}}$--(G3)$_{\ref{lem: improve regularity}}$ hold with positive probability.

We assign each $e\in \cK_{k}(H^{(k-1)})$ to be in $\cA:=\bigcup_{i\in [s]}\cB_i$ independently with probability $\delta^{1/3}$
and assign every edge in $\cA$ independently to one $\cB_i$  with probability $d_i$.
%and for each $E\in \cA$ we assign a number $i\in [s]$ to $E$ independently at random so that $i$ is assigned to $E$ with probability $d_i$.  
%We include $E$ into $\cB_i$ if $i$ is assigned to $E$. Note that $\cB_1,\dots, \cB_s$ forms a partition of $\cA$.
For each $i\in [s]$, we define
$$G^{(k)}_i:= (H^{(k)}_i\setminus \cA)\cup \cB_i.$$
%It is enough to show that $G^{(k)}_1,\dots, G^{(k)}_s$ satisfies (G1)$_{\ref{lem: improve regularity}}$--(G3)$_{\ref{lem: improve regularity}}$ with probability at least $1/2$.
Observe that (G1)$_{\ref{lem: improve regularity}}$ holds by construction.
Moreover,
$$\mathbb{E}[|H^{(k)}_i\cap \cA|] = \delta^{1/3} |H^{(k)}_i |
\enspace\text{ and }\enspace \mathbb{E}[|\cB_i| ] = \delta^{1/3} d_i |\cK_{k}(H^{(k-1)})|. $$
Thus Lemma~\ref{lem: chernoff} with the fact that $|H^{(k)}_i| \leq m^{k}$ implies that  
$$\mathbb{P}[|H^{(k)}_i\cap \cA| \leq 2\delta^{1/3} m^{k}] \geq 1 - 2e^{ -\delta^{2/3} m^{k}} \geq 1 - 1/(6s),$$
and
$$ \mathbb{P}[|\cB_i| \leq 2\delta^{1/3} m^{k}] \geq 1 - 2e^{ -\delta^{2/3} m^{k}} \geq 1 - 1/(6s).$$
Hence with probability at least $2/3$,
we have $|H^{(k)}_i\cap \cA| \leq 2\delta^{1/3} m^{k}$ for all $i\in [s]$ and $|\cB_i|\leq 2\delta^{1/3} m^{k}$.
This implies 
\begin{align*}%\label{eq: GH triangle}
|G^{(k)}_i\triangle H^{(k)}_i| \leq  |H^{(k)}_i\cap \cA|+ |\cB_i| \leq \nu m^{k}.
\end{align*}
Thus (G3)$_{\ref{lem: improve regularity}}$ holds with probability at least $2/3$.

Furthermore, 
for each $i\in [s]$ and $Q^{(k-1)}\subseteq H^{(k-1)}$ with $|\cK_{k}(Q^{(k-1)})| \geq \epsilon |\cK_{k}(H^{(k-1)})| \geq \epsilon^2 m^k$,
we obtain
\begin{eqnarray}
\mathbb{E}[|G^{(k)}_i\cap \cK_{k}(Q^{(k-1)})| ] 
&=& \mathbb{E}[| ( H^{(k)}_i\cap \cK_{k}(Q^{(k-1)}) ) \setminus \cA|]  + \mathbb{E}[|\cB_i \cap \cK_{k}(Q^{(k-1)})|] \nonumber\\
&\stackrel{\eqref{eq: Q density}}{=}& \left((1-\delta^{1/3})(d_i\pm (\epsilon+\delta^{1/2})) + \delta^{1/3} d_i  \right)|\cK_{k}(Q^{(k-1)})| \nonumber\\
&=
& (d_i \pm (\epsilon - \delta^{1/2}))|\cK_{k}(Q^{(k-1)})|.
\end{eqnarray}\COMMENT{ $ \left((1-\delta^{1/3})(d_i\pm (\epsilon+\delta^{1/2})) + \delta^{1/3} d_i  \right) = d_i - \delta^{1/3} d_i \pm (1-\delta^{1/3})(\epsilon+\delta^{1/2}) + \delta^{1/3}d_i = d_i \pm  (1-\delta^{1/3})(\epsilon+\delta^{1/2}) = d_i \pm (\epsilon - \delta^{1/3}\epsilon + \delta^{1/2} - \delta^{5/6}) = d_i \pm (\epsilon - \delta^{1/2})$ since $\delta \ll \epsilon$.  }
Thus Lemma~\ref{lem: chernoff} implies that
$$\mathbb{P}[ |G^{(k)}_i\cap \cK_{k}(Q^{(k-1)})|
 = (d_i\pm \epsilon)|\cK_{k}(Q^{(k-1)}| ] 
\geq 1 - 2e^{-\delta |\cK_{k}(Q^{(k-1)})|^2/ m^{k}} 
\geq 1 - e^{-\delta^2 m^{k} }.$$
Since there are at most $2^{|H^{(k-1)}|} \leq 2^{m^{k-1}}$ distinct choices of $Q^{(k-1)}$, 
using a union bound, 
we conclude that with probability at least $1- s 2^{m^{k-1}} e^{-\delta^2 m^{k} } \geq 2/3$, 
for all $i\in [s]$ and $Q^{(k-1)}\subseteq H^{(k-1)}$ with $\cK_{k}(Q^{(k-1)})\geq \epsilon |\cK_{k}(H^{(k-1)})|$, 
we have
$$d(G^{(k)}_i \mid Q^{(k-1)}) 
=\frac{(d_i \pm \epsilon)|\cK_{k}(Q^{(k-1)})|}{|\cK_{k}(Q^{(k-1)})|} 
= d_i \pm \epsilon.$$
This implies that $G^{(k)}_i$ is $(\epsilon,d_i)$-regular with respect to $H^{(k-1)}$, and so (G2)$_{\ref{lem: improve regularity}}$ holds with probability at least $2/3$.
Hence $G^{(k)}_1,\dots, G^{(k)}_s$ satisfy (G1)$_{\ref{lem: improve regularity}}$--(G3)$_{\ref{lem: improve regularity}}$ with probability at least $1/3$.
\end{proof}

We now generalize Lemma~\ref{lem: improve regularity} to the setting where
 we consider a family of partitions instead of only $k$-graphs with one common underlying $(k-1)$-graph.

\begin{lemma}\label{lem: improve regularity partition}
Suppose $0 < 1/n \ll \delta \ll \nu \ll \epsilon \leq 1/2 $\COMMENT{
$R$ being regularity instance already contains information that $\epsilon \leq \norm{\ba}^{-4^k}\leq 2^{-4^k}$ and thus it follows that $\nu\ll 1/k$.
} and $R=(\epsilon/3,\ba,d_{\ba,k})$\COMMENT{ The $\epsilon/3$ is needed in the proof of Lemma~\ref{lem:7.2}.} is a regularity instance.\COMMENT{Note that $R$ being regularity instance implies that $\epsilon/2  < \epsilon_{\ref{def: regularity instance}}(\|\ba \|_{\infty},k)$. So we don't need to write $\epsilon \ll 1/k,\|\ba \|_{\infty}^{-1}$ here. }
Suppose $H^{(k)}$ is an $n$-vertex  $k$-graph on vertex set $V$.
Suppose that $\sP=\sP(k-1,\ba)$ is an $(\epsilon+\delta,\ba,d_{\ba,k})$-equitable partition of $H^{(k)}$.
Then there exists a family of partitions $\sQ=\sQ(k-1,\ba)$ and a $k$-graph $G^{(k)}$ on $V$ such that 
\begin{itemize}
\item[{\rm (G1)$_{\ref{lem: improve regularity partition}}$}] $\sQ$ is an $(\epsilon,\ba,d_{\ba,k})$-equitable partition of $G^{(k)}$ and 
\item[{\rm (G2)$_{\ref{lem: improve regularity partition}}$}]  $|G^{(k)}\triangle H^{(k)}|\leq \nu \binom{n}{k}$. 
\end{itemize}
\end{lemma}
\begin{proof}\COMMENT{`$1/n\ll \delta $' is unnecessary condition here.
We can just put $0< 1/n,\delta \ll \nu$ in the statement and add the following sentence at the beginning of the proof.
`By increasing the value of $\delta$ if necessary, we may assume that we have
$0< 1/n \ll \delta \ll \nu$.'
Note that $\sP$ is an $(\epsilon+\delta,\ba,d_{\ba,k})$-equitable partition of $H^{(k)}$ even after increasing the value of $\delta$.}
We define $m:= \lfloor n/a_1 \rfloor$.
Consider $j\in  [k]\setminus\{1\}$ and $\hat{\bx}\in \hat{A}(j,j-1,\ba)$.
We claim that 
\begin{align}\label{eq: polyad big}
|\cK_{j}(\hat{P}^{(j-1)}(\hat{\bx}))|\geq \frac{1}{2}\prod_{i=2}^{j-1}a_i^{-\binom{j}{i}} m^j \geq \epsilon^{1/2} m^{j}.
\end{align}
Indeed, this holds if $j=2$, so suppose that $j>2$.
As $R$ is a regularity instance, we have $\epsilon/3 \leq \|\ba\|_\infty^{-4^k}\epsilon_{\ref{lem: counting}}(\|\ba\|_\infty^{-1},\|\ba\|_\infty^{-1},j-1,j)$\COMMENT{Note that $\epsilon_{\ref{def: regularity instance}}(x,\cdot)$ is a decreasing function, 
so we can have $j-1,j$ here instead of $k-1,k$. 
Note that $\epsilon_{\ref{lem: counting}}$ doesn't have to be decreasing function in terms of third and fourth variables as long as $\epsilon_{\ref{def: regularity instance}}(x,\cdot)$ is a decreasing function.} and Lemma~\ref{lem: maps bijections}(ii) implies that $\hat{\cP}(\hat{\bx})$ is an $(\epsilon, (1/a_2,\dots, 1/a_{j-1}))$-regular $(m,j,j-1)$-complex.
Thus we can apply Lemma~\ref{lem: counting} with $\hat{\cP}(\hat{\bx})$ playing the role of $\cH$ to show \eqref{eq: polyad big}.

Note that Lemma~\ref{lem: maps bijections}(i) implies that $\hat{P}^{(j-1)}(\cdot):\hat{A}(j,j-1,\ba)\to \hat{\sP}^{(j-1)}$ and
$P^{(j)}(\cdot,\cdot):\hat{A}(j,j-1,\ba)\times [a_j]\to \sP^{(j)}$ are bijections 
(except in the case when $j=k$ for the latter).

We choose $\delta \ll \nu_2 \ll \nu_3 \ll \dots \ll \nu_{k+1}:=\nu^{2}$.
For each $\hat{\bx} \in \hat{A}(k,k-1,\ba)$, let
\begin{align}\label{eq: Pk1 Pk2 hat bx}
P^{(k)}(\hat{\bx},1):= H^{(k)}\cap \cK_{k}(\hat{P}^{(k-1)}(\hat{\bx})) \text{ and }P^{(k)}(\hat{\bx},2):=  \cK_{k}(\hat{P}^{(k-1)}(\hat{\bx}))\setminus H^{(k)}.
\end{align}
In addition, we set $a_k:=2$.

We proceed in an inductive manner.
Let $\sQ^{(1)}:=\sP^{(1)}$, and for each $\hat{\bx}\in \hat{A}(2,1,\ba)$, 
let $\hat{Q}^{(1)}(\hat{\bx}):= \hat{P}^{(1)}(\hat{\bx})$.
Assume that for $j\in [k]\sm\{1\}$, we have defined $\{\sQ^{(i)}\}_{i=1}^{j-1}$ such that
\begin{itemize}
\item[(Q1)$^j_{\ref{lem: improve regularity partition}}$] $\{\sQ^{(i)}\}_{i=1}^{j-1}$ is a $(1/a_1,\epsilon,(a_1,\dots, a_{j-1}))$-equitable family of partitions and
\item[(Q2)$^j_{\ref{lem: improve regularity partition}}$] $|\hat{P}^{(i-1)}(\hat{\bx})\triangle \hat{Q}^{(i-1)}(\hat{\bx})|\leq \nu_i^{1/2}  m^{i-1}$
for all $i\in [j]\setminus\{1\}$ and $\hat{\bx} \in \hat{A}(i,i-1,\ba)$.
\end{itemize}\vspace{0.3cm}
Note that this holds for $j=2$.
Suppose $\hat{\bx} \in \hat{A}(j,j-1,\ba)$.
If $j < k$, then for each $\hat{\bx} \in \hat{A}(j,j-1,\ba)$ and $b\in [a_j]$, we define
\begin{align*}
&Q'^{(j)}(\hat{\bx},b):= 
\left\{\begin{array}{ll}
P^{(j)}(\hat{\bx},b) \cap \cK_{j}(\hat{Q}^{(j-1)}(\hat{\bx})) &\text{ if } b\in [a_j-1],\\
\cK_{j}(\hat{Q}^{(j-1)}(\hat{\bx})) \setminus \bigcup_{b\in [a_j-1]} Q'^{(j)}(\hat{\bx},b) &\text{ if } b\in a_j.
\end{array} \right.
\end{align*}
If $j=k$, then we define
$$Q'^{(j)}(\hat{\bx},1):= H^{(k)}\cap \cK_{k}(\hat{Q}^{(j-1)}(\hat{\bx})),\text{ and } Q'^{(j)}(\hat{\bx},2):=  \cK_{k}(\hat{Q}^{(j-1)}(\hat{\bx}))\setminus H^{(k)}.$$
So  for each $j\in [k]\sm \{1\}$ and $b\in [a_j]$, we obtain\COMMENT{
If $b <a_j-1$ use lhs $\leq |P^{(j)}(\hat{\bx},b) \sm \cK_{j}(\hat{Q}^{(j-1)}(\hat{\bx}))|$.
If $b=a_j$ use $Q'^{(j)}(\hat{\bx},b)= (P^{(j)}(\hat{\bx},b) \cap \cK_{j}(\hat{Q}^{(j-1)}(\hat{\bx}))) \cup 
(\cK_{j}(\hat{Q}^{(j-1)}(\hat{\bx})) \setminus \bigcup_{b\in [a_j]} P^{(j)}(\hat{\bx},b))$.
}
\begin{eqnarray}\label{eq: P Q' diff}
|Q'^{(j)} (\hat{\bx},b) \triangle P^{(j)}(\hat{\bx},b)|\leq |\cK_{j}(\hat{P}^{(j-1)}(\hat{\bx}))\triangle \cK_j(\hat{Q}^{(j-1)}(\hat{\bx}))| 
\stackrel{{\rm(Q2)}^j_{\ref{lem: improve regularity partition}}}{\leq} 2\nu_j^{1/2} m^{j} .
\end{eqnarray}

Next, we will apply Lemma~\ref{lem: improve regularity} to find a partition $\sQ^{(j)}$ of $\cK_j(\sQ^{(1)})$. Fix $\hat{\bx} \in \hat{A}(j,j-1,\ba)$.
If $j \in[k-1]\sm\{1\}$, then for every $b\in [a_j]$ we define $d_b:=1/a_j$.
If $j=k$, then let $d_1:=d_{\ba,k}(\hat{\bx})$ and $d_2:=1-d_1$.

 Since $\sP$ is an $(\epsilon+\delta,\ba,d_{\ba,k})$-equitable partition of $H^{(k)}$, it follows from Definition~\ref{def: equitable family of partitions} that 
$P^{(j)}(\hat{\bx},b)$ is $(\epsilon+\delta,d_b)$-regular with respect to $\hat{P}^{(j-1)}(\hat{\bx})$ for each $b\in [a_j]$. (Here, we use Lemma~\ref{lem: simple facts 1}(i) in the case when $j=k$.)
Thus Lemma~\ref{lem: simple facts 2} together with \eqref{eq: polyad big}, \eqref{eq: P Q' diff} and (Q2)$^j_{\ref{lem: improve regularity partition}}$ implies that $Q'^{(j)}(\hat{\bx},b)$ is $(\epsilon+\delta+\nu_j^{1/9},d_b)$-regular with respect to $\hat{Q}^{(j-1)}(\hat{\bx})$. 

By using Lemma~\ref{lem: improve regularity} with $\hat{Q}^{(j-1)}(\hat{\bx}), Q'^{(j)}(\hat{\bx},1),\dots,Q'^{(j)}(\hat{\bx},a_j), a_j,d_b,\epsilon,\delta+\nu_j^{1/9},\nu_{j+1}/2$ playing the roles of $H^{(k-1)}, H^{(k)}_1,\dots, H^{(k)}_s, s,d_i, \epsilon,\delta,\nu$, respectively, 
we obtain 
$Q^{(j)}(\hat{\bx},1),$ $\dots,Q^{(j)}(\hat{\bx},a_j)$ forming a partition of $\cK_j(\hat{Q}^{(j-1)}(\hat{\bx}))$ such that 
$Q^{(j)}(\hat{\bx},b)$ is $(\epsilon,d_b)$-regular with respect to $\hat{Q}^{(j-1)}(\hat{\bx})$ and $|Q'^{(j)}(\hat{\bx},b)\triangle Q^{(j)}(\hat{\bx},b)|\leq \nu_{j+1} m^{j}/2$ for all $b\in [a_j]$.
(A similar argument as for~\eqref{eq: polyad big} shows that $\hat{Q}^{(j-1)}(\hat{\bx})$ satisfies the requirements of Lemma~\ref{lem: improve regularity}.)
By \eqref{eq: P Q' diff} we have that
\begin{align}\label{eq: P Q b diff}
|P^{(j)}(\hat{\bx},b)\triangle Q^{(j)}(\hat{\bx},b)|\leq 
 |Q'^{(j)} (\hat{\bx},b) \triangle P^{(j)}(\hat{\bx},b)| + |Q'^{(j)}(\hat{\bx},b)\triangle Q^{(j)}(\hat{\bx},b)| \leq \nu_{j+1}m^{j}.
 \end{align}
 
Suppose first that $j\in [ k-1]\sm\{1\}$. Define $$\sQ^{(j)}:= \{Q^{(j)}(\hat{\bx},b) : \hat{\bx}\in \hat{A}(j,j-1,\ba), b\in [a_j]\}.$$
Thus $\{\sQ^{(i)}\}_{i=1}^{j}$ forms a $(1/a_1, \epsilon, (a_1,\dots, a_{j}))$-equitable family of partitions since $d_b=1/a_j$ for all $b\in [a_j]$. Hence (Q1)$^{j+1}_{\ref{lem: improve regularity partition}}$ holds.

Furthermore, for each $\hat{\bx}\in \hat{A}(j+1,j,\ba)$ we obtain a polyad
$$\hat{Q}^{(j)}(\hat{\bx}):= \bigcup_{\hat{\by}\leq_{j,j-1} \hat{\bx}} Q^{(j)}(\hat{\by},\bx^{(j)}_{\hat{\by}^{(1)}_*}).$$
Then
$$|\hat{P}^{(j)}(\hat{\bx})\triangle \hat{Q}^{(j)}(\hat{\bx})|
\leq \sum_{\hat{\by}\leq_{j,j-1} \hat{\bx}}|P^{(j)}(\hat{\by},\bx^{(j)}_{\hat{\by}^{(1)}_*})\triangle Q^{(j)}(\hat{\by},\bx^{(j)}_{\hat{\by}^{(1)}_*}) | \stackrel{\eqref{eq: P Q b diff}}{\leq} (j+1) \nu_{j+1}m^{j} 
\leq \nu_{j+1}^{1/2}  m^{j},$$
and so (Q2)$^{j+1}_{\ref{lem: improve regularity partition}}$ holds.

Now suppose $j=k$. 
Let $$G^{(k)}:= \left(H^{(k)}\setminus \cK_{k}( \sQ^{(1)})\right) \cup
\bigcup_{\hat{\bx}\in \hat{A}(k,k-1,\ba)} Q^{(k)}(\hat{\bx},1).$$
By definition of $d_1,d_2$, 
we obtain that $\{\sQ^{(i)}\}_{i=1}^{k}$ is an $(\epsilon,\ba,d_{\ba,k})$-equitable partition of $G^{(k)}$.
Moreover, since $\sP^{(1)} = \sQ^{(1)}$, we have that 
\begin{eqnarray*}
H^{(k)} &=& \left(H^{(k)}\setminus \cK_{k}( \sQ^{(1)})\right) \cup \left( H^{(k)}\cap \bigcup_{\hat{\bx}\in \hat{A}(k,k-1,\ba)} \cK_{k}(\hat{P}^{(k-1)}(\hat{\bx} ))\right) \\
&\stackrel{\eqref{eq: Pk1 Pk2 hat bx}}{=}& \left(H^{(k)}\setminus \cK_k(\sQ^{(1)})\right) \cup \bigcup_{\hat{\bx}\in \hat{A}(k,k-1,\ba)} P^{(k)}(\hat{\bx},1).
\end{eqnarray*}
 Thus 
$$|H^{(k)}\triangle G^{(k)}|\leq \sum_{\hat{\bx} \in \hat{A}(k,k-1,\ba)} |P^{(k)}(\hat{\bx},1)\triangle Q^{(k)}(\hat{\bx},1)| \stackrel{\eqref{eq: P Q b diff}}{\leq} |\hat{A}(k,k-1,\ba)| \nu_{k+1} m^{k} \leq \nu \binom{n}{k}.$$
Indeed, the final inequality holds since $\nu_{k+1} = \nu^2$, 
since $|\hat{A}(k,k-1,\ba)|\leq \|\ba\|_{\infty}^{2^{k}}$ by Proposition~\ref{prop: hat relation}(viii), 
and the definition of a regularity instance implies that $\nu \ll \epsilon < \|\ba\|_{\infty}^{-4^k}$.
\end{proof}

\subsection{Refining a partition}\label{sec:3.4}
In this subsection we make two simple observations regarding refinements of a given partition.
The first one of these shows that we can refine a family of partitions 
without significantly affecting the regularity parameters.

\begin{lemma}\label{lem: partition refinement}
Suppose $0< 1/n \ll\epsilon \ll 1/t ,1/k$ with $k, t\in \N\sm \{1\}$, $0<\eta <1$, and $\ba\in \N^{k-1}$.
Suppose $\sP=\sP(k-1,\ba)$ is an $(\eta,\epsilon,\ba)$-equitable family of partitions on $V$ with $|V|=n$. 
Suppose $\bb\in [t]^{k-1}$  and $a_i \mid b_i$ for all $i\in [k-1]$.
Then there exists a family of partitions $\sQ=\sQ(k-1,\bb)$ on $V$ which is $(\eta,\epsilon^{1/3},\bb)$-equitable and $\sQ \prec \sP$.
\end{lemma}
It is easy to prove this by induction on $k$ via an appropriate application of the slicing lemma (Lemma~\ref{lem: slicing}.)
We omit the details.%%%%

The next observation shows that if $\sQ$ is an equitable partition of $H$ and $\sP\prec \sQ$,
then we can modify $H$ slightly to obtain $G$ so that $\sP$ is a equitable partition of $G$ (where the relevant densities are inherited from $\sQ$ and $H$).

\begin{proposition}\label{eq: sP prec sQ then sP is good partition}
Suppose that $0<1/n\ll \epsilon\ll \epsilon' \ll 1/T, 1/a^{\sQ}_1 \ll \nu\ll 1/k$ with $k\in \mathbb{N}\setminus\{1\}$, and $\ba^{\sP} \in [T]^{k-1}$.
Suppose that $\sP=\sP(k-1,\ba^{\sP})$ is a $(1/a_1^{\sP},\epsilon',\ba^{\sP})$-equitable family of partitions on $V$, that $\sQ=\sQ(k-1,\ba^{\sQ})$ is an $(\epsilon,\ba^{\sQ},d_{\ba^{\sQ},k})$-equitable partition of an $n$-vertex $k$-graph $H^{(k)}$ on $V$, and that $\sP\prec \sQ$. 
Let $d_{\ba^{\sP},k}$ be the density function defined by  
$$d_{\ba^{\sP},k}(\hat{\by}):= \left\{ \begin{array}{ll}
d_{\ba^{\sQ},k}(\hat{\bx}) & \text{ if } \enspace \exists \hat{\bx}\in \hat{A}(k,k-1,\ba^{\sQ}): \cK_k(\hat{P}^{(k-1)}(\hat{\by}))\subseteq \cK_k(\hat{Q}^{(k-1)}(\hat{\bx})), \\
0 & \text{ if }\enspace \cK_k(\hat{P}^{(k-1)}(\hat{\by}))\cap \cK_k(\sQ^{(1)}) = \emptyset.
\end{array}\right.$$
Let $G^{(k)}:= H^{(k)}\cap \cK_k(\sQ^{(1)})$.
Then
\begin{itemize}
\item[\rm (i)] $\sP$ is an $(\epsilon',\ba^{\sP},d_{\ba^{\sP},k})$-equitable partition of $G^\kk$,
\item[\rm (ii)] $|H^{(k)}\triangle G^{(k)}| \leq \nu \binom{n}{k}.$
\end{itemize}
\end{proposition}
\begin{proof}
Note that (ii) follows from \eqref{eq: eta a1}. 
We now verify (i).
Since $\sP\prec\sQ$, by Proposition~\ref{prop: hat relation}(vi)~and~(xi),
for each $\hat{\by}\in \hat{A}(k,k-1,\ba^{\sP})$, 
either there exists a unique $\hat{\bx}\in \hat{A}(k,k-1,\ba^{\sQ})$ 
such that $\cK_k(\hat{P}^{(k-1)}(\hat{\by}))\subseteq \cK_k(\hat{Q}^{(k-1)}(\hat{\bx}))$
or $\cK_k(\hat{P}^{(k-1)}(\hat{\by}))\cap \cK_k(\sQ^{(1)}) = \emptyset$.
Suppose first that there exists a unique $\hat{\bx}\in \hat{A}(k,k-1,\ba^{\sQ})$ 
such that $\cK_k(\hat{P}^{(k-1)}(\hat{\by}))\subseteq \cK_k(\hat{Q}^{(k-1)}(\hat{\bx}))$.
Then Lemma~\ref{lem: counting} implies that 
\begin{align*}%\label{eq: cKk hat P k-1 mu 1/2 |cKk|}
	|\cK_{k}(\hat{P}^{(k-1)}(\hat{\by}))| 
	\geq  (1- 1/4) \prod_{i=1}^{k-1} (a_i^{\sP})^{-\binom{k}{i}} n^{k}  
	\geq  \epsilon^{1/2} |\cK_{k}(\hat{Q}^{(k-1)}(\hat{\bx}))|.
\end{align*}
Together with Lemma~\ref{lem: simple facts 1}(ii) and the fact that $d_{\ba^{\sP},k}(\hat{\by}) = d_{\ba^{\sQ},k}(\hat{\bx})$ and $G^{(k)} \cap \cK_k(\hat{P}^{(k-1)}(\hat{\by})) = H^{(k)} \cap \cK_k(\hat{P}^{(k-1)}(\hat{\by}))$, this implies that $G^{(k)}$ is $(\epsilon^{1/2}, d_{\ba^{\sP},k}(\hat{\by}))$-regular with respect to $\hat{P}^{(k-1)}(\hat{\by})$.

Now suppose that $ \cK_k(\hat{P}^{(k-1)}(\hat{\by}))\cap \cK_k(\sQ^{(1)})=\emptyset$. 
Since $G^{(k)} \subseteq  \cK_k(\sQ^{(1)})$, we have $ \cK_k(\hat{P}^{(k-1)}(\hat{\by}))\cap G^{(k)} =\emptyset$.
Thus $G^{(k)}$ is also $(\epsilon^{1/2}, d_{\ba^{\sP},k}(\hat{\by}))$-regular with respect to $\hat{P}^{(k-1)}(\hat{\by})$ since $d_{\ba^{\sP},k}(\hat{\by})=0$.
Since $\epsilon^{1/2} \leq \epsilon'$, altogether this shows that
$\sP$ is an $(\epsilon',\ba^{\sP},d_{\ba^{\sP},k})$-equitable partition of $G^{(k)}$.
\end{proof}

\subsection{Small perturbations of partitions}\label{sec:3.5}
Here we consider the effect of small changes in a partition on the resulting parameters.
In particular, the next lemma implies that for any equitable family of partitions $\sP$,
every family of partitions that is close to $\sP$ in distance is an equitable family of partitions with almost the same parameters.

\begin{lemma}\label{lem: slightly different partition regularity}
Suppose $k\in \N\sm\{1\}$, $0<1/n \ll \nu \ll \epsilon$, and $0\leq \lambda \leq 1/4$.
Suppose $R=(\epsilon/3,\ba,d_{\ba,k})$ is a regularity instance.
Suppose $V$ is a vertex set of size $n$ and suppose $G^\kk,H^\kk$ are $k$-graphs on $V$ with $|G^{(k)}\triangle H^{(k)}|\leq \nu \binom{n}{k}$.
Suppose $\sP=\sP(k-1,\ba)$ is a $(1/a_1,\epsilon,\ba,\lambda)$-equitable family of partitions on $V$ which is an $(\epsilon,d_{\ba,k})$-partition of $H^{(k)}$.
Suppose $\sQ=\sQ(k-1,\ba)$ is a family of partitions on $V$ such that
for any $j\in [k-1]$, $\hat{\bx}\in \hat{A}(j,j-1,\ba)$, and $b\in [a_j]$, we have
\begin{align}\label{eq: P triangle Q leq nu}
	|P^{(j)}(\hat{\bx},b) \triangle Q^{(j)}(\hat{\bx},b)| \leq \nu \binom{n}{j}.
\end{align}
Then $\sQ$ is a $(1/a_1,\epsilon+ \nu^{1/6},\ba,\lambda+\nu^{1/6})$-equitable family of partitions which is an $(\epsilon+\nu^{1/6},d_{\ba,k})$-partition of $G^{(k)}$.
\end{lemma}
\begin{proof}
Let $t:=\|\ba\|_\infty$ and $m:= \lfloor n/a_1 \rfloor$.
Note that since $R=(\epsilon/3,\ba,d_{\ba,k})$ is a regularity instance, 
we have $\epsilon\leq 3\epsilon_{\ref{def: regularity instance}}(t,k) \leq t^{-4^k}\epsilon_{\ref{lem: counting}}(1/t, 1/t,k-1,k)$. 
Thus Lemma~\ref{lem: counting} (together with Lemma~\ref{lem: maps bijections}(ii)) implies for any $j\in [k-1]\sm\{1\}$ and $\hat{\bx}\in \hat{A}(j,j-1,\ba)$ that
\begin{align}\label{eq: hat polyad is not too much small...}
|\cK_{j}(\hat{P}^{(j-1)}(\hat{\bx}))| 
\geq   (1-1/t)\prod_{i=2}^{j-1}a_i^{-\binom{j}{i}} ((1-\lambda)m)^{j} \geq \epsilon^{1/2} m^{j}.
\end{align}
Since $\sP$ is a $(1/a_1, \epsilon,\ba,\lambda)$-equitable family of partitions, for each $j\in [k-1]\sm\{1\}$, $\hat{\bx}\in \hat{A}(j,j-1,\ba)$, and $b\in [a_j]$, 
the $j$-graph $P^{(j)}(\hat{\bx},b)$ is $(\epsilon,1/a_{j})$-regular with respect to $\hat{P}^{(j-1)}(\hat{\bx})$. 
Hence
$$|P^{(j)}(\hat{\bx},b)|\geq (1/a_j - \epsilon) |\cK_{j}(\hat{P}^{(j-1)}(\hat{\bx}))| \geq
\epsilon^{2/3} m^{j}.$$

Observe that $\sP$ does not provide a partition of $\cK_k(\sP^{(1)})$.
We define such a partition with $a_k:=2$ by setting, for each $\hat{\bx}\in \hat{A}(k,k-1,\ba)$,
$$P^{(k)}(\hat{\bx},1):= \cK_{k}(\hat{P}^{(k-1)}(\hat{\bx}))\cap H^{(k)} \text{ and } P^{(k)}(\hat{\bx},2):= \cK_{k}(\hat{P}^{(k-1)}(\hat{\bx}))\setminus H^{(k)}.$$
Similarly, let
$$Q^{(k)}(\hat{\bx},1):= \cK_{k}(\hat{Q}^{(k-1)}(\hat{\bx}))\cap G^{(k)} \text{ and } Q^{(k)}(\hat{\bx},2):= \cK_{k}(\hat{Q}^{(k-1)}(\hat{\bx}))\setminus G^{(k)}.$$
Fix some $j\in [k]\setminus\{1\}$, $\hat{\bx}\in \hat{A}(j,j-1,\ba)$, and $b\in [a_j]$. 
Depending on fixed choices of $j,\hat{\bx}$ and $b$, let 
$$d:= \left\{ \begin{array}{ll} 1/a_j  & \text{ if }j\in [k-1],\\
d_{\ba,k}(\hat{\bx}) & \text{ if } j=k , b=1, \text{ and} \\
1-d_{\ba,k}(\hat{\bx}) & \text{ if }j=k, b=2. \end{array}\right.$$
Hence $P^{(j)}(\hat{\bx},b)$ is $(\epsilon,d)$-regular with respect to $\hat{P}^{(j-1)}$. (Here we use Lemma~\ref{lem: simple facts 1}(i) if $j=k$ and $b=2$.)
Recall that \eqref{eq:hatPconsistsofP(x,b)} holds for both $\sP$ and $\sQ$. Together with \eqref{eq: P triangle Q leq nu} this implies that 
$$|\hat{P}^{(j-1)}(\hat{\bx})\triangle \hat{Q}^{(j-1)}(\hat{\bx})|
\leq \nu j\binom{n}{j-1} \leq \nu^{1/2} m^{j-1}.$$
Moreover, \eqref{eq: P triangle Q leq nu} also implies that $ |P^{(j)}(\hat{\bx},b) \triangle Q^{(j)}(\hat{\bx},b)| \leq \nu \binom{n}{j} \leq \nu^{1/2} m^{j}$.
Thus, by \eqref{eq: hat polyad is not too much small...}, we can use Lemma~\ref{lem: simple facts 2}\COMMENT{ Note that we have $\epsilon <1/100$ since $(\epsilon/3,\ba,d_{\ba,k})$ is a regularity instance. 
Note that $\| \ba\|_{\infty}\geq 2$, $k\geq 2$, thus $ \| \ba\|_{\infty}^{-4^k} \leq 2^{-16} < 1/100$ and $\epsilon \leq \| \ba\|_{\infty}^{-4^k} < 1/100$ from the definition of regularity instance.
} with $P^{(j)}(\hat{\bx},b), \hat{P}^{(j-1)}(\hat{\bx}),Q^{(j)}(\hat{\bx},b), \hat{Q}^{(j-1)}(\hat{\bx}), \nu^{1/2}$ playing the roles of $H^{(k)}, H^{(k-1)}, G^{(k)}, G^{(k-1)}, \nu$ to conclude that
$Q^{(j)}(\hat{\bx},b)$ is $(\epsilon+\nu^{1/6},d)$-regular with respect to $\hat{Q}^{(j-1)}(\hat{\bx},b)$.
Furthermore, \eqref{eq: P triangle Q leq nu} for $j=1$ implies that each part of $\sQ^{(1)}$ has size $(1\pm \lambda \pm \nu^{1/2}) n/a_1$. 
Thus $\sQ$ is an $(1/a_1,\epsilon+ \nu^{1/6},\ba,\lambda+\nu^{1/6})$-equitable family of partitions on $V$. 
Moreover, since $Q^{(k)}(\hat{\bx},1)= \cK_{k}(\hat{Q}^{(k-1)}(\hat{\bx}))\cap G^{(k)}$, we know that $\sQ$ is also an $(\epsilon+\nu^{1/6},d_{\ba,k})$-partition of $G^{(k)}$.
\end{proof}

\subsection{Distance of hypergraphs and density functions}\label{sec:3.6}
Recall that the distance between two density functions was defined in Section~\ref{sec: subsub density function}.
In this subsection we present two results that relate the distance between two hypergraphs to the distance between density functions of equitable families of partitions of these hypergraphs.
The first result shows that the distance of the density functions provides a lower bound on the distance between the two hypergraphs (we will use this in the proof of Theorem~\ref{thm: estimable}).
On the other hand,
Lemma~\ref{lem: density close then hypergraph close} shows that the lower bound given in Lemma~\ref{lem: hypergraph close then density close} is essentially tight (this will also be applied in the proof of Theorem~\ref{thm: estimable}).

\begin{lemma}\label{lem: hypergraph close then density close}
Suppose $0< 1/n \ll \epsilon  \ll \nu, 1/t, 1/k$ and $k\in \N\sm\{1\}$.
Suppose $\ba\in [t]^{k-1}$.
Suppose that $G^{(k)}$ and $H^{(k)}$ are $k$-graphs on vertex set $V$ with $|V|=n$.
Suppose that $\sP$ is an $(\epsilon,\ba,d^{G}_{\ba,k})$-equitable partition of $G^{(k)}$ and an $(\epsilon,\ba,d^{H}_{\ba,k})$-equitable partition of $H^{(k)}$. 
Then 
\begin{align*}
	\dist(d^{H}_{\ba,k},d^{G}_{\ba,k}) \leq |H^{(k)}\triangle G^{(k)}|\binom{n}{k}^{-1} + \nu.
\end{align*}
\end{lemma}
\begin{proof}
\COMMENT{`$1/n\ll \epsilon $' is unnecessary condition here.
We can just put $0< 1/n,\epsilon \ll \nu$ in the statement and add the following sentence at the beginning of the proof.
`By increasing the value of $\epsilon$ if necessary, we may assume that we have
$0< 1/m \ll \epsilon \ll \nu$.'
Note that  $\sP$ is an $(\epsilon,\ba,d^{G}_{\ba,k})$-equitable partition  even after increasing the value of $\epsilon$.
}
For each $\hat{\bx}\in \hat{A}(k,k-1,\ba)$, 
by definition, 
$G^{(k)}$ is $(\epsilon, d^{G}_{k,\ba}(\hat{\bx}))$-regular with respect to $\hat{P}^{(k-1)}(\hat{\bx})$ and $H^{(k)}$ is $(\epsilon, d^{H}_{k,\ba}(\hat{\bx}))$-regular with respect to $\hat{P}^{(k-1)}(\hat{\bx})$.
Thus
\begin{align*}
|G^{(k)}\cap  \cK_{k}(\hat{P}^{(k-1)}(\hat{\bx}))| &= (d^{G}_{k,\ba}(\hat{\bx})\pm \epsilon)|\cK_{k}(\hat{P}^{(k-1)}(\hat{\bx}))| \text{ and}\\
|H^{(k)}\cap  \cK_{k}(\hat{P}^{(k-1)}(\hat{\bx}))| &= (d^{H}_{k,\ba}(\hat{\bx})\pm \epsilon)|\cK_{k}(\hat{P}^{(k-1)}(\hat{\bx}))|.
\end{align*}
Hence
\begin{align}\label{eq: HG diff}
\left(|d^{G}_{\ba,k}(\hat{\bx})- d^{H}_{\ba,k}(\hat{\bx})|- 2\epsilon\right)|\cK_{k}(\hat{P}^{(k-1)}(\hat{\bx}))| 
\leq |(H^{(k)}\triangle G^{(k)})\cap \cK_{k}(\hat{P}^{(k-1)}(\hat{\bx}))|.
\end{align}
Since $1/n \ll \epsilon \ll \nu, 1/t,1/k$, Lemma~\ref{lem: counting} (together with Lemma~\ref{lem: maps bijections}(ii)) implies that 
\begin{align}\label{eq:cliques1}
	|\cK_{k}(\hat{P}^{(k-1)}(\hat{\bx}))| = \left(1\pm \frac{\nu}{2}\right) \prod_{i=1}^{k-1} a_i^{-\binom{k}{i}} n^{k}.
\end{align}
As $\{ \cK_{k}(\hat{P}^{(k-1)}(\hat{\bx})) :\hat{\bx}\in \hat{A}(k,k-1,\ba)\}$ partitions $\cK_{k}(\sP^{(1)})$ (by Proposition~\ref{prop: hat relation}(vi)),
we obtain
\begin{eqnarray*}
\dist(d^{G}_{\ba,k},d^{H}_{\ba,k})\hspace{-0.2cm}  &=& \hspace{-0.2cm} 
k! \prod_{i=1}^{k-1} a_i^{-\binom{k}{i}}\hspace{-0.4cm}  \sum_{\hat{\bx}\in \hat{A}(k,k-1,\ba)} \hspace{-0.4cm}|d^{G}_{\ba,k}(\hat{\bx})- d^{H}_{\ba,k}(\hat{\bx})|  \nonumber \\ 
\hspace{-0.2cm} &\stackrel{\eqref{eq: HG diff}}{\leq}& \hspace{-0.2cm} 
k! \prod_{i=1}^{k-1} a_i^{-\binom{k}{i}} \hspace{-0.4cm}
\sum_{\hat{\bx}\in \hat{A}(k,k-1,\ba)}  \hspace{-0.4cm}
\frac{|(H^{(k)}\triangle G^{(k)})\cap \cK_{k}(\hat{P}^{(k-1)}(\hat{\bx}))| + 2\epsilon |\cK_{k}(\hat{P}^{(k-1)}(\hat{\bx}))| }{|\cK_{k}(\hat{P}^{(k-1)}(\hat{\bx}))|} \nonumber \\
\hspace{-0.2cm} &\stackrel{(\ref{eq:cliques1})}{\leq}&\hspace{-0.2cm}  \sum_{\hat{\bx}\in \hat{A}(k,k-1,\ba)} \hspace{-0.4cm}\frac{k! |(H^{(k)}\triangle G^{(k)})\cap \cK_{k}(\hat{P}^{(k-1)}(\hat{\bx}))| +\epsilon^{1/2} n^{k}}{ (1- \nu/2) n^{k}}  \nonumber \\
\hspace{-0.2cm} &\leq& \hspace{-0.2cm}  \frac{ k! |H^{(k)}\triangle G^{(k)}|}{n^{k}} + \frac{2\nu}{3}  \leq |H^{(k)}\triangle G^{(k)}|\binom{n}{k}^{-1} + \nu,
\end{eqnarray*}
which completes the proof.
\end{proof}

\begin{lemma}\label{lem: density close then hypergraph close}
Suppose $0< 1/n\ll \epsilon \ll \nu , 1/t, 1/k$ and $k\in \N\sm\{1\}$.
Suppose $\ba\in [t]^{k-1}$.
Suppose that $H^{(k)}$ is an $n$-vertex $k$-graph and
suppose that $\sP$ is an $(\epsilon,\ba,d^{H}_{\ba,k})$-equitable partition of $H^{(k)}$.
Suppose $d^{G}_{\ba,k}$ is a density function of $\hat{A}(k,k-1,\ba)$. 
Then there exists a $k$-graph $G^{(k)}$ such that $\sP$ is a $(3\epsilon,\ba,d^{G}_{\ba,k})$-equitable partition of $G^{(k)}$ and 
\begin{align*}
	|H^{(k)}\triangle G^{(k)}|\leq (\dist(d^{H}_{\ba,k},d^{G}_{\ba,k})+\nu)\binom{n}{k}.
\end{align*}
\end{lemma}
\begin{proof}
\COMMENT{`$1/n\ll \epsilon $' is unnecessary condition here.
We can just put $0< 1/n,\epsilon \ll \nu$ in the statement and add the following sentence at the beginning of the proof.
`By increasing the value of $\epsilon$ if necessary, we may assume that we have
$0< 1/m \ll \epsilon \ll \nu$.'
Note that  $\sP$ is an $(\epsilon,\ba,d^{H}_{\ba,k})$-equitable partition  even after increasing the value of $\epsilon$.
}
Suppose $\hat{\bx}\in \hat{A}(k,k-1,\ba)$.
By Lemmas~\ref{lem: counting} and~\ref{lem: maps bijections}(ii), and the assumption $1/n, \epsilon \ll \nu, 1/t, 1/k$, we obtain
\begin{align}\label{eq: complex not small}
|\cK_k(\hat{P}^{(k-1)}(\hat{\bx}))|= \left(1\pm\frac{\nu}{2}\right)  \prod_{i=1}^{k-1} a_i^{-\binom{k}{i}} n^{k}.
\end{align}
We will distinguish the following cases depending on the values of $d^{H}_{\ba,k}(\hat{\bx})$ and $d^{G}_{\ba,k}(\hat{\bx})$.

\noindent
\textbf{Case 1:} $|d^{H}_{\ba,k}(\hat{\bx})-d^{G}_{\ba,k}(\hat{\bx})|\leq 2\epsilon $.
Let  $G^{(k)}(\hat{\bx}) := H^{(k)}(\hat{\bx})$.

\noindent
\textbf{Case 2:} $d^{H}_{\ba,k}(\hat{\bx})> d^{G}_{\ba,k}(\hat{\bx})+2\epsilon$. 
Let 
$$p_1:= \max\left\{ \frac{d^{G}_{\ba,k}(\hat{\bx})}{d^{H}_{\ba,k}(\hat{\bx})}, 1-\frac{d^{G}_{\ba,k}(\hat{\bx})}{d^{H}_{\ba,k}(\hat{\bx})} \right\}.$$
This definition implies that $1/2 \leq p_1 \leq 1$.
Note $d^{H}_{\ba,k}(\hat{\bx})\geq 2\epsilon$.
Because of this and \eqref{eq: complex not small}, we can apply the slicing lemma (Lemma~\ref{lem: slicing}) to $H^{(k)}\cap \cK_{k}(\hat{P}^{(k-1)}(\hat{\bx}))$
to%
\COMMENT{
{\small
\begin{tabular}{c|c|c|c|c}%|c|c|c|c|c|c|c|c}
object/parameter & $H^{(k)}\cap \cK_{k}(\hat{P}^{(k-1)}(\hat{\bx}))$ & $\hat{P}^{(k-1)}(\hat{\bx})$ & $d^{H}_{\ba,k}(\hat{\bx})$ & $p_1$ \\ \hline
playing the role of & $H^{(k)}$& $ H^{(k-1)}$ & $d$ &  $p_1$\\ 
\end{tabular}
}
}
obtain two $k$-graphs 
$F^{(k)}_0, F^{(k)}_1\subseteq H^{(k)}\cap \cK_{k}(\hat{P}^{(k-1)}(\hat{\bx}))$ such that 
\begin{itemize}
\item $F^{(k)}_1$ is $(3\epsilon,d^{H}_{\ba,k}(\hat{\bx})p_1)$-regular with respect to $\hat{P}^{(k-1)}(\hat{\bx})$, and
\item $F^{(k)}_0$ is $(3\epsilon,d^{H}_{\ba,k}(\hat{\bx})(1-p_1))$-regular with respect to $\hat{P}^{(k-1)}(\hat{\bx})$.
\end{itemize}
Let 
$$G^{(k)}(\hat{\bx}) :=\left\{\begin{array}{ll} F^{(k)}_1 & \text{ if } \frac{d^{G}_{\ba,k}(\hat{\bx})}{d^{H}_{\ba,k}(\hat{\bx})}\geq 1/2, \\
F^{(k)}_0 & \text{ otherwise.}\end{array} \right.$$
By construction, we have
$G^{(k)}(\hat{\bx})\subseteq H^{(k)}\cap \cK_{k}(\hat{P}^{(k-1)}(\hat{\bx})).$

\noindent
\textbf{Case 3:} $d^{H}_{\ba,k}(\hat{\bx})<  d^{G}_{\ba,k}(\hat{\bx})-2\epsilon$.
Let
$$p'_1:= \max\left\{ \frac{1-d^{G}_{\ba,k}(\hat{\bx})}{1-d^{H}_{\ba,k}(\hat{\bx})}, 1-\frac{1-d^{G}_{\ba,k}(\hat{\bx})}{1-d^{H}_{\ba,k}(\hat{\bx})} \right\}.$$
Similarly as before, $1/2 \leq p'_1 \leq 1$.
Note $1-d^{H}_{\ba,k}(\hat{\bx})\geq 2\epsilon$.
Because of this, Lemma~\ref{lem: simple facts 1}(i) and \eqref{eq: complex not small}, we can apply Lemma~\ref{lem: slicing} with $\cK_{k}(\hat{P}^{(k-1)}(\hat{\bx}))\setminus H^{(k)}$, $\hat{P}^{(k-1)}(\hat{\bx})$, $1-d^{H}_{\ba,k}(\hat{\bx})$ and $p'_1$ playing the roles of $H^{(k)}$, $ H^{(k-1)}$, $d$ and  $p_1$.
Then we obtain two $k$-graphs 
$F^{(k)}_0, F^{(k)}_1\subseteq \cK_{k}(\hat{P}^{(k-1)}(\hat{\bx}))\setminus H^{(k)}$ such that
\begin{itemize}
\item $F^{(k)}_1$ is $(3\epsilon,(1-d^{H}_{\ba,k}(\hat{\bx}))p'_1)$-regular with respect to $\hat{P}^{(k-1)}(\hat{\bx})$, and 
\item $F^{(k)}_0$ is $(3\epsilon,(1-d^{H}_{\ba,k}(\hat{\bx}))(1-p'_1))$-regular with respect to $\hat{P}^{(k-1)}(\hat{\bx})$.
\end{itemize}
We define
\begin{align*}
	G^{(k)}(\hat{\bx}) 
	:=\left\{\begin{array}{ll} \cK_{k}(\hat{P}^{(k-1)}(\hat{\bx}))\setminus F^{(k)}_1& \text{ if } \frac{1-d^{G}_{\ba,k}(\hat{\bx})}{1-d^{H}_{\ba,k}(\hat{\bx})}\geq 1/2,\\
	\cK_{k}(\hat{P}^{(k-1)}(\hat{\bx}))\setminus F^{(k)}_0 & \text{ otherwise.}\end{array} \right.
\end{align*}
Hence
$$H^{(k)}\cap \cK_{k}(\hat{P}^{(k-1)}(\hat{\bx})) \subseteq G^{(k)}(\hat{\bx}) \subseteq \cK_{k}(\hat{P}^{(k-1)}(\hat{\bx})).$$ 
In addition, $G^{(k)}(\hat{\bx})$ is $(3\epsilon,d^{G}_{\ba,k}(\hat{\bx}))$-regular with respect to $\hat{P}^{(k-1)}(\hat{\bx})$,
by Lemma~\ref{lem: simple facts 1}(i).

We can now combine the graphs from the above three cases and let
\begin{align*}
	G^\kk:= \left( \bigcup_{\hat{\bx}\in \hat{A}(k,k-1,\ba)}G^{(k)}(\hat{\bx})\right) \cup (H^\kk\sm \cK_k(\sP^{(1)})).
\end{align*}
By construction, $\sP$ is a $(3\epsilon,\ba,d^{G}_{\ba,k})$-equitable partition of $G^{(k)}$. 
In all three cases,
$$|(H^{(k)}\triangle G^{(k)})\cap \cK_{k}(\hat{P}^{(k-1)}(\hat{\bx}))| 
\leq (|d^{H}_{\ba,k}(\hat{\bx})- d^{G}_{\ba,k}(\hat{\bx})|+ 4\epsilon)|\cK_{k}(\hat{P}^{(k-1)}(\hat{\bx}))|.$$
Therefore, we conclude
\begin{eqnarray*}
|H^{(k)}\triangle G^{(k)}|&=& 
\sum_{\hat{\bx}\in \hat{A}(k,k-1,\ba)} |(H^{(k)}\triangle G^{(k)})\cap \cK_{k}(\hat{P}^{(k-1)}(\hat{\bx}))| \nonumber \\
&\leq& \sum_{\hat{\bx}\in \hat{A}(k,k-1,\ba)}(|d^{H}_{\ba,k}(\hat{\bx})- d^{G}_{\ba,k}(\hat{\bx})|+ 4\epsilon)|\cK_{k}(\hat{P}^{(k-1)}(\hat{\bx}))|\nonumber\\
&\stackrel{\eqref{eq: complex not small}}{\leq}& \left(1+ \frac{\nu}{2}\right)\sum_{\hat{\bx}\in \hat{A}(k,k-1,\ba)}  (|d^{H}_{\ba,k}(\hat{\bx})- d^{G}_{\ba,k}(\hat{\bx})|+ 4\epsilon) \prod_{i=1}^{k-1} a_i^{-\binom{k}{i}} n^{k} \\
&\leq& (\dist( d^{H}_{\ba,k},d^{G}_{\ba,k})  +\nu) \binom{n}{k}.
\end{eqnarray*}
\COMMENT{
We could add this as a second last line:
$\left(1+ \frac{\nu}{2}\right) (\dist( d^{H}_{\ba,k},d^{G}_{\ba,k})+ \epsilon^{1/2}  ) \frac{n^k}{k!}$.
}
\end{proof}

\section{Testable properties are regular reducible}\label{sec:testred}

In this section we show one direction of our main result (Lemma~\ref{lem: main 1}).
It states that every testable $k$-graph property can be (approximately) described by suitable regularity instances. (Recall that the formal definition of being regular reducible is given in Definition~\ref{def: regular reducible}.)

\begin{lemma}\label{lem: main 1}
If a $k$-graph property is testable, then it is regular reducible. 
\end{lemma}

Goldreich and Trevisan~\cite{GT03} proved that every testable graph property
is also testable in some canonical way.
It is an easy exercise to translate their results to the hypergraph setting,
as their arguments also work in this case.
Thus in the proof of Lemma~\ref{lem: main 1} we may restrict ourselves to such canonical testers.

To be precise, an $(n,\alpha)$-tester $\bT=\bT(n,\alpha)$ is \emph{canonical}
if, given an $n$-vertex $k$-graph $H$, it chooses
a set $Q$ of $q_k = q_k(n,\alpha)$ vertices of $H$ uniformly at random, 
queries all $k$-sets in $Q$,
and then accepts or rejects $H$ (deterministically) according to (the isomorphism class of) $H[Q]$. In particular, $\bT$ has query complexity $\binom{q_k}{k}$. Moreover, every canonical tester is non-adaptive.

\begin{lemma}[Goldreich and Trevisan~\cite{GT03}]\label{lem: canonical}
Suppose that $\bP$ is a $k$-graph property which is testable with query complexity at most $q_k= q_k(\alpha)$. Then for all $n\in \mathbb{N}$ and $\alpha\in (0,1)$, there is a canonical $(n,\alpha)$-tester $\bT=\bT(n,\alpha)$ for $\bP$ with query complexity at most $\binom{9kq_k}{k}$.
\end{lemma}
\COMMENT{
Lemma~4.4 in \cite{GT03} says that if we have a tester with error probability $1/6$, 
then we can get rid of its adaptiveness by increasing error probability by a factor of $2$. 
In order for this, we need to guarantees that initial tester has error probability less than $1/6$. 
To do this, we simply repeat the original tester $9$ times and accept or reject based on majority vote. 
Thus we can just choose $9kq$ vertices uniformly at random and check nine $kq$-vertex subgraphs and accept or reject based on majority vote. 
Take this tester and apply Lemma~4.4 to get rid of adaptivity, then we get canonical tester which choose $9kq_k$ vertices at random, and query complexity at most $\binom{9kq_k}{k}$.
}
To prove Lemma~\ref{lem: main 1}, we let $\cQ$ be the set of all $k$-graphs which are accepted by a canonical tester $\bT$ for $\bP$.
We then construct a (bounded size) set of regularity instances $\cR$
where the `induced density' of $\cQ$ is large for each $R\in \cR$.
We then apply  the counting lemma for induced graphs (Corollary~\ref{cor: counting collection}) to show that $\cR$ satisfies the requirements of Definition~\ref{def: regular reducible}.
\begin{proof}[Proof of Lemma~\ref{lem: main 1}]
Let $\bP$ be a testable $k$-graph property. Thus there exists a function $q'_k:(0,1)\rightarrow \mathbb{N}$ such that for every $n\in \mathbb{N}$ and $\beta \in (0,1)$,
there exists an $(n,\beta)$-tester $\bT = \bT(n,\beta)$ for $\bP$ with query complexity at most $q'_k(\beta)$. By Lemma~\ref{lem: canonical}, we may assume (by increasing $q'_k(\beta)$ if necessary) that $\bT$ is canonical.  
Since $\bT$ is canonical, there exists $q = q(n,\beta)$ with $\binom{q(n,\beta)}{k} \leq q'_k(\beta)$ such that 
given any $k$-graph $H$ on $n$ vertices $\bT$ samples a set $Q$ of $q$ vertices, 
considers $H[Q]$, and then deterministically accepts or rejects $H$ based on $H[Q]$.

Let $\cQ$ be the set of all the $k$-graphs on $q$ vertices such that $\bT$ accepts $H$ if and only if there is $Q'\in \cQ$ that is isomorphic to $H[Q]$.

In order to show that $\bP$ is regular reducible, it suffices to consider the case when $0<\beta <1/100$.
We fix some function $\overline{\epsilon}:\mathbb{N}^{k-1}\rightarrow (0,1)$ such that $\overline{\epsilon}(\ba) \ll \norm{\ba}^{-k}$
for all $\ba=(a_1,\dots,a_{k-1})\in \N^{k-1}$.
We choose constants $ \epsilon_*, \epsilon, \eta$, and $n_0, T\in \mathbb{N}$ such that $0< \epsilon_* \ll 1/n_0\ll \epsilon \ll 1/T \ll \eta \ll 1/q, \beta, 1/k$.\COMMENT{Formally, $q'_k$ would be better as $q$ depends on $n$ but $q$ seems clearer.}
 In particular, we have $n_0\geq n_{\ref{thm: RAL}}(\eta,\beta q^{-k},\overline{\epsilon})$, $T \geq t_{\ref{thm: RAL}}(\eta  ,\beta q^{-k},\overline{\epsilon})$ and  $\epsilon\ll \overline{\epsilon}(\ba)$ for any $\ba\in [T]^{k-1}$.

For each $\ell\in [n_0]\sm [k-1]$, 
we let $\cQ'(\ell)$ be the collection of $\ell$-vertex $k$-graphs satisfying the property $\bP$ and we let $\bI'(\ell)$ be the collection of regularity instances $R=( \epsilon_*, \ba,d_{\ba,k})$ 
such that
\begin{itemize}
\item[(R0)$_{\ref{lem: main 1}}$]  $\ba = (\ell,1,\dots, 1) \in \mathbb{N}^{k-1}$ and $d_{\ba,k}(\hat{\bx}) \in \{0,1\}$ for every $\hat{\bx}\in \hat{A}(k,k-1,\ba)$.
\end{itemize}

Let $\bI$ be the collection of regularity instances $R=(\epsilon'',\ba, d_{\ba,k})$ such that
\begin{itemize}
\item[(R1)$_{\ref{lem: main 1}}$] $\epsilon''\in \{\epsilon, 2\epsilon, \dots, \lceil(\overline{\epsilon}(\ba))^{1/2}\epsilon^{-1}\rceil \epsilon\}$,
\item[(R2)$_{\ref{lem: main 1}}$]  $\ba\in [T]^{k-1}$ and $a_1 > \eta^{-1}$, and
\item[(R3)$_{\ref{lem: main 1}}$]  $d_{\ba,k}(\hat{\bx}) \in \{0,\epsilon^2, 2\epsilon^2,\dots, 1\}$ for every $\hat{\bx}\in \hat{A}(k,k-1,\ba)$.
\end{itemize}

Observe that by construction
$|\bI|$ and $|\bI'(\ell)|$ are bounded by a function of $\beta$, $k$ and $q_k'(\beta)$ for every $\ell\in [n_0]\sm [k-1]$. 
Recall that $IC(\cQ, d_{\ba,k})$ was defined in \eqref{eq: def IC cF d a k}.
We define
$$\cR(n,\beta) := \left\{\begin{array}{ll} 
\{ R \in \bI'(n) :   R= (\epsilon_*,\ba,d_{\ba,k}) \text{ with } IC(\cQ'(n),d_{\ba,k}) >0 \} & \text{ if } n\leq n_0,\\
\{ R \in \bI :   R= (\epsilon'',\ba,d_{\ba,k}) \text{ with } IC(\cQ,d_{\ba,k}) \geq 1/2\} & \text{ if } n>n_0.
\end{array}\right.
$$

In order to show that the property $\bP$ is regular reducible, 
we need to show that for every $\alpha>\beta$ the following holds: 
every $n$-vertex $k$-graph $H$ that satisfies $\bP$ is $\beta$-close to satisfying $R$ for at least one $R\in \cR(n,\beta)$,
and every $n$-vertex $k$-graph $H$ that is $\alpha$-far from satisfying $\bP$ is $(\alpha-\beta)$-far from satisfying $R$ for all $R\in \cR(n,\beta)$.

First of all, we consider the case that $n \in  [n_0]\sm[k-1]$. 
Note that the only way for $H$ to satisfy $R\in \cR(n,\beta)$ is to have a partition $\sP^{(1)}$ of $V(H)$ into $n$ singleton clusters and the natural $(n,1,\dots,1)$-equitable partition $\sP$ arising from $\sP^{(1)}$.%
\COMMENT{Because, for each $\hat{\bx}\in \hat{A}(j+1,j,\ba)$ with $j\in [k-1]$, $|\cK_j(\hat{P}^{(k-1)}(\hat{\bx})|=1$. Thus in order $P^{(k-1)}(\hat{\bx},b)$ to be $\epsilon_*$-regular with respect to $\hat{P}^{(k-1)}(\hat{\bx})$, we must have $P^{(k-1)}(\hat{\bx},b) =\cK_j(\hat{P}^{(k-1)}(\hat{\bx})$. In this case, $IC(Q,d_{\ba,k})>0$ for an $n$-vertex graph $Q$ implies that there exists an isomorphism from $Q$ to $H$. Thus  $IC(\cQ'(n),d_{\ba,k}) >0$ if and only if $H \in \cQ'(n).$} 
In this case, it is easy to see that $IC(\cQ'(n),d_{\ba,k}) >0$ if and only if $H \in \cQ'(n).$ Thus we conclude that
every $n$-vertex $k$-graph $H$ that satisfies $\bP$ satisfies $R$ for at least one $R\in \cR(n,\beta)$, and every $k$-graph $H$ that is $\alpha$-far from satisfying $\bP$ is $\alpha$-far from satisfying $R$ for all $R\in \cR(n,\beta)$.

Now we suppose that $n>n_0$ and we let $\cR:= \cR(n,\beta)$.\COMMENT{Note that for $n > n_0$ and $\beta<1/100$, the set $\cR(n,\beta)$ is independent of $n$.} First, suppose that a $k$-graph $H$ satisfies $\bP$, and thus $\bT$ accepts $H$ with probability at least $2/3$. 
Hence
\begin{align}\label{eq: Pr 2/3}
\mathbf{Pr}(\cQ,H)\geq 2/3.
\end{align}

Since $|V(H)|\geq n_0$, by applying the regular approximation lemma (Theorem~\ref{thm: RAL}) with $H, \eta  ,\beta q^{-k},\overline{\epsilon}$ playing the roles of $H,\eta,\nu,\epsilon$, 
we obtain a $k$-graph $G$ and a family of partitions $\sP= \sP(k-1,\ba^{\sP})$ such that
\begin{itemize}
\item[(A1)$_{\ref{lem: main 1}}$] $\sP$ is $(\eta,\overline{\epsilon}(\ba^{\sP}),\ba^{\sP})$-equitable for some $\ba^{\sP} \in [T]^{k-1}$, 
\item[(A2)$_{\ref{lem: main 1}}$] $G$ is perfectly $\overline{\epsilon}(\ba^{\sP})$-regular with respect to $\sP$, and 
\item[(A3)$_{\ref{lem: main 1}}$] $|G\triangle H|\leq \beta q^{-k} \binom{n}{k}$.
\end{itemize}

Let $\epsilon':= \overline{\epsilon}(\ba^{\sP})$.
By the choice of $\overline{\epsilon}$ and $\eta$, we conclude that
$0<\epsilon' \ll 1/\norm{\ba^{\sP}}\leq  1/a_1^{\sP} \ll 1/q, \beta, 1/k$ and
by the choice of $\epsilon$, we obtain  $\epsilon \ll \epsilon'$.
Note that if a $k$-graph $F$ is $(\epsilon',d)$-regular with respect to a $(k-1)$-graph $F'$, 
then $F$ is $(\epsilon'',d')$-regular with respect to $F'$ 
for some $d'\in \{0,\epsilon^2, 2\epsilon^2,\dots, 1\}$ and $\epsilon''\in \{\epsilon, 2\epsilon,\dots,  \lceil \epsilon'^{1/2}\epsilon^{-1}\rceil \epsilon\}\cap [2\epsilon',3\epsilon']$.
Thus there exists 
\begin{align}\label{eq: RG in bIa}
R_G=(\epsilon'',\ba^{\sP} ,d^{G}_{\ba^{\sP},k})\in \bI
\end{align} 
such that $G$ satisfies $R_G$.

Proposition~\ref{prop: mathbf Pr doesn't change much} with (A3)$_{\ref{lem: main 1}}$ implies that 
\begin{align}\label{eq: Pr - beta}
\mathbf{Pr}(\cQ, G) \geq \mathbf{Pr}(\cQ, H) -\beta.
\end{align}
Since $0<\epsilon' \ll 1/\norm{\ba^{\sP}}\leq 1/a_1^{\sP} \ll 1/q,\beta, 1/k$, Corollary~\ref{cor: counting collection} implies that 
\begin{eqnarray*}
IC(\cQ,d^{G}_{\ba^{\sP},k} ) & \geq& \mathbf{Pr}(\cQ,G) - \beta 
\stackrel{\eqref{eq: Pr - beta}}{\geq}  \mathbf{Pr}(\cQ, H) -2\beta
\stackrel{\eqref{eq: Pr 2/3}}{\geq} 2/3 - 2\beta > 1/2.
\end{eqnarray*}
By the definition of $\cR$ and \eqref{eq: RG in bIa}, 
this implies that $R_G\in \cR$ and 
so $H$ is indeed $\beta$-close to a graph $G$ satisfying $R_{G}$, 
one of the regularity instances of $\cR$.\COMMENT{Here $\epsilon''\ll 1/\norm{\ba^\sP}$, thus it is legitimate regularity-instance. }

Suppose now that $H$ is $\alpha$-far from satisfying $\bP$.
Since $\bT$ is a canonical $(n,\beta)$-tester, we conclude
\begin{equation}\label{eq: Pr 1/3}
\begin{minipage}[c]{0.8\textwidth}\em
if a $k$-graph $H'$ is $\beta$-far from satisfying $\bP$,
then $\mathbf{Pr}(\cQ,H')\leq 1/3$.
\end{minipage}
\end{equation}
Assume for a contradiction that $H$ is $(\alpha-\beta)$-close to satisfying some regularity instance $R_F=(\epsilon'',\bb,d^{F}_{\bb,k})\in \cR$.
Thus there exists a $k$-graph $F$ which satisfies $R_F$ and which is $(\alpha-\beta)$-close to $H$. 
Since $R_F\in \cR$, we obtain
\begin{align*}
IC(\cQ, d_{\bb,k}^F) \geq 1/2.
\end{align*}

Then (R1)$_{\ref{lem: main 1}}$ and (R2)$_{\ref{lem: main 1}}$ guarantee that 
we can apply Corollary~\ref{cor: counting collection} with $\cQ, R_F, T, 1/10$ playing the roles of $\cF, R, t, \gamma$ to conclude that
$$\mathbf{Pr}(\cQ, F) \geq IC(\cQ, d_{\bb,k}^F) - 1/10 \geq 1/2 - 1/10 > 1/3.$$
Together with \eqref{eq: Pr 1/3} this implies that $F$ is $\beta$-close to satisfying $\bP$.
Thus there exists a $k$-graph $F'$ satisfying $\bP$ which is $\beta$-close to $F$. 
Since $H$ is $(\alpha-\beta)$-close to $F$, we conclude that $F'$ is $\alpha$-close to $H$, which contradicts the assumption that $H$ is $\alpha$-far from satisfying $\bP$.  This completes the proof that $\bP$ is regular reducible.
\end{proof}

\section{Satisfying a regularity instance is testable}\label{sec:regtest}

In this section we deduce from Lemma~\ref{lem: random choice2} that the property of satisfying a particular regularity instance is testable.
Suppose $H$ is a $k$-graph and $Q$ is a subset of the vertices chosen uniformly at random.
Lemma~\ref{lem:7.1} shows that if $H$ satisfies a regularity instance $R$,
then with high probability $H[Q]$ is close to satisfying $R$.
Lemma~\ref{lem:7.2} gives the converse:
if $H$ is far from satisfying $R$,
then with high probability $H[Q]$ is also far from satisfying $R$.

Lemmas~\ref{lem:7.1} and~\ref{lem:7.2} follow from Lemma~\ref{lem: random choice2}.
It implies that a family of partitions not only transfers from a hypergraph to its random samples with high probability,
but also vice versa.
Crucially, in both directions these transfer results allow only a small additive increase in the regularity parameters (which can then be eliminated via Lemma~\ref{lem: improve regularity partition}).
The proof of Lemma~\ref{lem: random choice2} is given in \cite{JKKO2}.

\begin{lemma}\label{lem: random choice2}
Suppose $0< 1/n< 1/q \ll c \ll \delta \ll \epsilon_0 \leq 1$ and $k\in \N\sm \{1\}$.
Suppose $R=(2\epsilon_0/3, \ba, d_{\ba,k})$ is a regularity instance.
Suppose $H$ is a $k$-graph on vertex set $V$ with $|V|=n$. 
Let $Q \in \binom{V}{q}$ be chosen uniformly at random. 
Then with probability at least $ 1 - e^{-cq}$ the following hold.
\begin{itemize}
\item[{\rm (Q1)$_{\ref{lem: random choice2}}$}] If there exists an $(\epsilon_0,\ba,d_{\ba,k})$-equitable partition $\sO_1$ of $H$, 
then there exists an $(\epsilon_0+\delta,\ba,d_{\ba,k})$-equitable partition $\sO_2$ of $H[Q]$.
\item[{\rm(Q2)$_{\ref{lem: random choice2}}$}] If there exists an $(\epsilon_0,\ba,d_{\ba,k})$-equitable partition $\sO_2$ of $H[Q]$, 
then there exists  an $(\epsilon_0+\delta,\ba,d_{\ba,k})$-equitable partition $\sO_1$ of $H$.
\end{itemize}
\end{lemma}

\begin{lemma}\label{lem:7.1}
Suppose that $0< 1/n < 1/q \ll c \ll \nu \ll \epsilon$ and $k\in \N\sm\{1\}$.
Suppose $R=(\epsilon,\ba,d_{\ba,k})$ is a regularity instance and
$H$ is an $n$-vertex $k$-graph on vertex set $V$ that satisfies $R$.
Suppose $Q\in \binom{V}{q}$ is chosen uniformly at random.
Then $H[Q]$ is $\nu$-close to satisfying $R$ with probability at least $1- e^{-c q}$.
\end{lemma}
\begin{proof}
Choose $\delta$ such that $ c \ll \delta \ll \nu.$
Since $H$ satisfies $R$, 
there exists an $(\epsilon,\ba,d_{\ba,k})$-equitable partition $\sO_1$ of $H$.  
Thus Lemma~\ref{lem: random choice2} implies that with probability at least $1- e^{-c q}$
there exists an $(\epsilon+\delta,\ba,d_{\ba,k})$-equitable partition $\sO_2$ of $H[Q]$.

Thus Lemma~\ref{lem: improve regularity partition}, with $H[Q], Q, q$ playing the roles of $H^{(k)}, V, n$ respectively, 
implies\COMMENT{Note that we can apply Lemma~\ref{lem: improve regularity partition} since if $R=(\epsilon,\ba,d_{\ba,k})$ is a regularity instance, 
then $R'=((\epsilon+\delta)/2,\ba,d_{\ba,k})$ is also a regularity instance.} that there exists a family of partitions $\sQ$ and a $k$-graph $G$ on vertex set $Q$ such that 
$\sQ$ is an $(\epsilon,\ba,d_{\ba,k})$-equitable partition of $G$ and $|G\triangle H[Q]|\leq \nu \binom{q}{k}$. 
Therefore, $G$ satisfies $R$ and $G[Q]$ is $\nu$-close to $H[Q]$.

Thus with probability at least $1- e^{-c q}$
the $k$-graph $H[Q]$ is $\nu$-close to satisfying $R$.
\end{proof}

\begin{lemma}\label{lem:7.2}
Suppose that $0< 1/n < 1/q \ll c \ll \nu \ll \epsilon, \alpha$ and $k\in \N\sm\{1\}$.
Suppose $R=(\epsilon,\ba,d_{\ba,k})$ is a regularity instance.
Suppose an $n$-vertex $k$-graph $H$ on vertex set $V$ is $\alpha$-far from satisfying $R$, 
and $Q\in \binom{V}{q}$ is chosen uniformly at random.
Then $H[Q]$ is $\nu$-far from satisfying $R$ with probability at least $1- e^{-c q}$.
\end{lemma}
We say that $R=(\epsilon',\ba,d_{\ba,k})$ is a \emph{relaxed regularity instance} if $(2\epsilon'/3,\ba,d_{\ba,k})$ is a regularity instance. 
In other words, $\epsilon'$ may exceed $\epsilon_{\ref{def: regularity instance}}(\|\ba\|_\infty, k)$ but is at most $\tfrac{3}{2} \cdot \epsilon_{\ref{def: regularity instance}}(\|\ba\|_\infty, k)$. Thus, if $R = (\epsilon', \ba, d_{\ba,k})$ is a relaxed regularity instance, then for any $\delta < \epsilon'$, the triple $\left((\epsilon' + \delta)/3, \ba, d_{\ba,k}\right)$ is a regularity instance. Consequently, we can apply certain lemmas—such as Lemma~\ref{lem: improve regularity partition}—to $(\epsilon' + \delta, \ba, d_{\ba,k})$-equitable partitions, provided the lemma only requires that $(\epsilon/3, \ba, d_{\ba,k})$ is a regularity instance.

\begin{proof}[Proof of Lemma~\ref{lem:7.2}]
Choose $\delta$ and $\alpha^*$ such that $c \ll \delta \ll \nu \ll \alpha^* \ll \epsilon, \alpha.$
Note that $2(\epsilon+\nu^{1/6}+\delta)/3\leq \epsilon$, thus 
$( 2(\epsilon+\nu^{1/6}+\delta)/3, \ba, d_{\ba,k})$ is a regularity instance.
Let $\cE$ be the event that the following holds:
\begin{equation}\label{eq: implication R' to R''}
\begin{minipage}[c]{0.8\textwidth}
\textit{If $H[Q]$ satisfies the relaxed regularity instance $R':=(\epsilon+\nu^{1/6},\ba,d_{\ba,k})$, 
then $H$ satisfies the relaxed regularity instance $R'':=(\epsilon+\nu^{1/6}+\delta,\ba,d_{\ba,k})$.}
\end{minipage}
\end{equation}

By Lemma~\ref{lem: random choice2},
we have 
$\mathbb{P}[\cE] \geq 1- e^{-c q}.$
We claim that if $\cE$ occurs and $H$ is $\alpha$-far from satisfying $R$, 
then $H[Q]$ is $\nu$-far from satisfying $R$. 
This clearly implies the lemma. 

Suppose $\cE$ holds and $H[Q]$ is $\nu$-close to satisfying $R$. 
Then there exists a $k$-graph $L$ on vertex set $Q$ which satisfies $R$ and $|H[Q]\triangle L|\leq \nu \binom{q}{k}$. This implies that there exists an $(\epsilon,\ba,d_{\ba,k})$-equitable partition $\sO_1$ of $L$.
Thus Lemma~\ref{lem: slightly different partition regularity} with $\sO_1$ playing the roles of both $\sP$ and $\sQ$ as well as $L, H[Q]$ playing the roles of $H^{(k)}$, $G^{(k)}$, respectively implies that $H[Q]$ satisfies the relaxed regularity instance $R'=(\epsilon+\nu^{1/6},\ba,d_{\ba,k})$.

By our assumption \eqref{eq: implication R' to R''}, this means that $H$ satisfies the relaxed regularity instance $R''=(\epsilon+\nu^{1/6}+\delta,\ba,d_{\ba,k})$. Thus Lemma~\ref{lem: improve regularity partition} with $\nu^{1/6}+\delta, \alpha^*$ playing the roles of $\delta, \nu$ respectively implies that there exists a family of partitions $\sQ$ and a $k$-graph $F$ on vertex set $V$ 
such that $\sQ$ is an $(\epsilon,\ba,d_{\ba,k})$-equitable partition of $F$ and  $|F\triangle H|\leq \alpha^* \binom{n}{k} \leq \alpha \binom{n}{k}$.
So $F$ satisfies the regularity instance $R$ and $F$ is $\alpha$-close to $H$. 
Thus $H$ is $\alpha$-close to satisfying $R$, 
a contradiction to our assumption.
Therefore, $H[Q]$ is $\nu$-far from satisfying $R$ with probability at least $\mathbb{P}[\cE] \geq 1- e^{-c q}$.
\end{proof}

\begin{theorem}\label{thm: regularity instance is testable}
For all $k\in \mathbb{N}\setminus\{1\}$ and
all regularity instances $R=(\epsilon,\ba,d_{\ba,k})$, 
the property of satisfying $R$ is testable.
\end{theorem}
\begin{proof}
Consider $\alpha\in (0,1)$ and $n\in \mathbb{N}$.
We choose $q= q(\alpha,\epsilon) \in \mathbb{N}$ and constants $\nu = \nu(\alpha,\epsilon),c = c(\alpha,\epsilon)>0$ such that $1/q \ll c \ll \nu \ll \epsilon, \alpha$.
Let $H$ be a $k$-graph on $n$ vertices. 
Without loss of generality, we assume that $n>q$.\COMMENT{Otherwise, we can test entire $k$-graph.}
We choose a set $Q$ of $q$ vertices of $H$ uniformly at random and accept $H$ if and only if $H[Q]$ is $\nu$-close to satisfying $R$. 
Then Lemmas~\ref{lem:7.1} and~\ref{lem:7.2} imply that with probability at least $ 1- e^{-cq} \geq 2/3$,
this algorithm distinguishes between the case that $H$ satisfies $R$ and that $H$ is $\alpha$-far from satisfying $R$.
This completes the proof.
\end{proof}

\section{Testable hypergraph properties are estimable}\label{sec:testest}

The notion of testability is similar to the notion of estimability.
In fact, we prove that for $k$-graphs these notions are equivalent.
This generalizes a result in \cite{FN07} where this is proved for graph properties.
The special case when $\bP$ is the property of satisfying a given regularity instance $R$ will be an important ingredient (together with Theorem~\ref{thm: regularity instance is testable}) in the proof of Theorem~\ref{thm:main} to show that every regular reducible property is testable.

\begin{theorem}\label{thm: estimable}
For every $k\in \N\sm\{1\}$,
a $k$-graph property $\bP$ is testable if and only if it is estimable. 
\end{theorem}

To prove Theorem~\ref{thm: estimable},
it suffices to show that for $k$-graphs every testable property $\bP$ is estimable.
As in Section~\ref{sec:testred} we can assume the existence of a family $\cF$ of bounded size
and a canonical tester $\bT$ for $\bP$
which accepts its input $H$ if and only if $\bT$ samples some $k$-graph in $\cF$.

We then consider the set $\bF$ of regularity instances which correspond to a high density of copies of the $k$-graphs in $\cF$.
Our estimator $\bT_E$ accepts $H$ if a random sample $H[Q]$ is sufficiently close to satisfying some $R\in \bF$ (see~\eqref{eq: algorithm}).

Now suppose $H$ is close to satisfying $\bP$.
Then there is a $k$-graph $F$ which satisfies $\bP$ and is close to $H$.
Via the partition version of the regular approximation lemma (Lemma~\ref{RAL(k)})
we can associate a high quality regularity partition $\sP$
with $F$ (actually we consider some $F'$ which is close to $F$) together with a suitable density function $d_{\ba^\sP,k}^F$. Since $F$ satisfies $\bP$, $d^{F}_{\ba^\sP,k}$ will give rise to a regularity instance $R_F \in \bF$ which is satisfied by $F$.
Via the regular approximation lemma and some additional arguments,
it also turns out that instead of $H$ we can actually consider some $G_*$
to which we can associate the same regularity partition $\sP$ but with a density function $d_{\ba^\sP,k}^G$ 
(which will also give rise to a regularity instance $R_{G_*}'$).
By Lemma~\ref{lem: random choice2}(i),
$\sP$ and $d^G_{\ba^\sP,k}$ are inherited by $G_*[Q]$ with high probability.
We can also construct $J$ on $Q$ which is close to $G_*[Q]$ and inherits $\sP$ and $d^F_{\ba^\sP,k}$ from  $F$, and also satisfies $R_F \in \bF$.
But this means that the estimator $\bT_E$ will indeed accept $H$.
See Figure~\ref{fig:estimable} for an illustration.

If $H$ is far from satisfying $\bP$,
then we argue via Lemma~\ref{lem: random choice2}(ii) rather than Lemma~\ref{lem: random choice2}(i).

\begin{proof}[Proof of Theorem~\ref{thm: estimable}]
Clearly, every estimable property is also testable.
Thus it suffices to show that given any $n\in \mathbb{N}$ and $\alpha>\beta>0$ and any testable $k$-graph property $\bP$, we can construct an $(n,\alpha,\beta)$-estimator.

Assume that $\bP$ is a testable $k$-graph property.
By Lemma~\ref{lem: canonical}, there exists a function $q_k: (0,1)\rightarrow \mathbb{N}$ such that the following holds: For any $n\in \mathbb{N}$,
there exist a canonical $(n,\beta/4)$-tester $\bT$ and an integer $q'=q'(n,\beta)\leq q_k(\beta/4)$ such that, given any $n$-vertex $k$-graph $H$, 
the tester $\bT$ chooses a random subset $Q'$ of $q'$ vertices of $H$ and (deterministically) accepts or rejects $H$ based on the isomorphism class of $H[Q']$.
Let $\cF$ be the collection of $q'$-vertex $k$-graphs 
such that $\bT$ accepts $H$ if and only if $H[Q']$ induces one of the $k$-graphs in $\cF$. 
Since $\bT$ is a $(n,\beta/4)$-tester, we conclude the following:
\begin{equation}\label{eq: Pr cF}
\begin{minipage}[c]{0.8\textwidth}
\textit{If $H$ satisfies $\bP$, 
then $\mathbf{Pr}(\cF,H)\geq 2/3$, and \\ 
if $H$ is $(\beta/4)$-far from satisfying $\bP$, 
then $\mathbf{Pr}(\cF,H)\leq 1/3$.}
\end{minipage}
\end{equation}

Let $\overline{\epsilon}:\mathbb{N}^{k-1}\rightarrow (0,1]$ be a function such that 
$\overline{\epsilon}(\ba)\ll \norm{\ba}^{-k}$ for every $\ba\in \N^{k-1}$.
We choose constants $\eta,\nu$ such that
\begin{align}\label{eq: def eta nu q' beta k const}
0< \eta \ll \nu \ll 1/q', \beta, 1/k.
\end{align}
Let $\mu:\mathbb{N}^{k-1}\rightarrow (0,1]$ be a function such that for any $\ba\in \mathbb{N}^{k-1}$, 
we have $\mu(\ba)\ll \norm{\ba}^{-k}, \eta, \nu$ and
\begin{align}\label{def: mu ba ll}
\mu(\ba) &\ll  \mu_{\ref{RAL(k)}}(k, \norm{\ba}, \norm{\ba}^{4^{k}},\eta,\nu,\overline{\epsilon}), 1/t_{\ref{RAL(k)}}(k, \norm{\ba}, \norm{\ba}^{4^{k}},\eta,\nu,\overline{\epsilon}),  \epsilon_{\ref{def: regularity instance}}(\norm{\ba},k).
\end{align}
In particular, we assume that 
\begin{align}\label{eq:defb}
	\mu(\ba)\ll \overline{\epsilon}(\bb) \text{ for all } \ba\in \mathbb{N}^{k-1} \text{ and } \bb\in [t_{\ref{RAL(k)}}(k, \norm{\ba}, \norm{\ba}^{4^{k}},\eta,\nu,\overline{\epsilon})]^{k-1}.
\end{align}
Let $T , \mu_*, c, q, n_0$ be numbers such that
\begin{align}\label{eq: def T mu* q}
&T:=t_{\ref{thm: RAL}}(\eta,\nu/2,\mu), \enspace\enspace
 \mu_* := \min_{\ba\in [T]^{k-1}} \mu(\ba), \\
&1/n_0\ll 1/q \ll c \ll \mu_*,\nu \enspace \text{ and }\enspace
n_0\geq n_{\ref{thm: RAL}}(\eta,\nu/2,\mu), n_{\ref{RAL(k)}}(k,T,T^{4^k},\eta,\nu,\overline{\epsilon}).\notag
\end{align}
Let $\bR$ be the collection of all regularity instances $R=(\epsilon,\ba,d_{\ba,k})$ satisfying the following.
\begin{itemize}
\item[(R1)$_{\ref{thm: estimable}}$] $\epsilon \in \{ \mu_*, 2\mu_*, \dots, 1\}$ and $\epsilon \leq \overline{\epsilon}(\ba)^{1/2}$,
\item[(R2)$_{\ref{thm: estimable}}$] $d_{\ba,k}(\hat{\bx}) \in \{0,\mu_*^2, 2\mu_*^2, \dots, 1\}$ for each $\hat{\bx}\in \hat{A}(k,k-1,\ba)$.
\end{itemize} \COMMENT{Note that since $R$ is a regularity instance, $\epsilon$ automatically satisfies the restriction in Definition~\ref{def: regularity instance} but this is not enough to apply Corollary~\ref{cor: counting collection} later on in the proof. For this reason we add $\epsilon \leq \overline{\epsilon}(\ba)^{1/2}$ to (R1)$_{\ref{thm: estimable}}$.}
Recall that $IC(\cF,d_{\ba,k})$ was defined in \eqref{eq: def IC cF d a k}. Let 
$$\bF:= \{ R=(\epsilon,\ba,d_{\ba,k}) \in \bR: IC(\cF,d_{\ba,k} ) \geq 1/2 \text{ and } a_1 \geq \eta^{-1}\}.$$
Note that $|\bF|$ is bounded by a function only depending on $\beta$, $k$, and $\bP$.\COMMENT{
The definition of regularity instance implies that $\norm{\ba}\leq 1/\epsilon\leq 1/\mu_*$.}

Fix $\alpha>\beta$.
As a next step we describe the algorithm $\bT_E(n,\alpha,\beta)$
which receives an $n$-vertex $k$-graph $H$ as an input:
\begin{equation}\label{eq: algorithm}
\begin{minipage}[c]{0.8\textwidth}
\textit{If $n< n_0$, 
then $\bT_E(n,\alpha,\beta)$ considers the entire $k$-graph $H$ and determines how close $H$ is to satisfying $\bP$. 
If $n\geq n_0$, then $\bT_E(n,\alpha,\beta)$ chooses a subset $Q$ of $q$ vertices of $H$ uniformly at random. 
If $H[Q]$ is $(\alpha-\beta/2)$-close to some $k$-graph which satisfies a regularity instance $R\in \bF$, 
then $\bT_E(n,\alpha,\beta)$ accepts $H$, otherwise it rejects $H$.
}\end{minipage}
\end{equation}
Recall that $q$ and $|\bF|$ are both bounded by a function of $\beta$, $k$, and $\bP$, so the query complexity of $\bT_E(n,\alpha,\beta)$ is also bounded by a function depending only on $\alpha, \beta$ and $k$.

In the following we verify that $\bT_E(n,\alpha,\beta)$ distinguishes $k$-graphs which are $(\alpha-\beta)$-close to satisfying $\bP$ and $k$-graphs which are $\alpha$-far from satisfying $\bP$
with probability at least $2/3$. If $n<n_0$, then this is trivial to verify since we consider the entire $k$-graph $H$. Thus we assume that $n\geq n_0$.

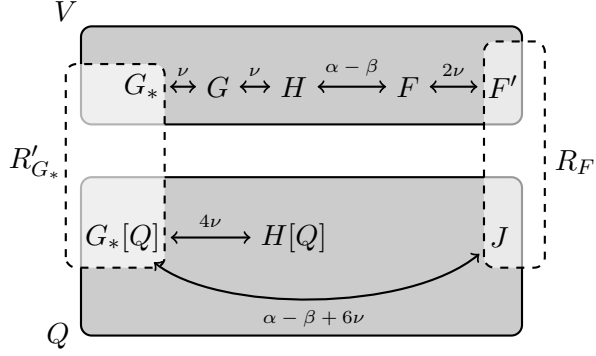
\begin{figure}[t]
\centering
\begin{tikzpicture}

\draw[rounded corners,thick,fill=gray!40] (-2.8,2.5) rectangle (3,3.8);
\draw[rounded corners,thick,fill=gray!40] (-2.8,-0.3) rectangle (3,1.8);

\node (b1) at (-3,4) {$V$};
\node (b2) at (-3.1,-0.3) {$Q$};

\draw[rounded corners,dashed,thick,fill opacity=0.6,fill=white!100] (2.5,0.6) rectangle (3.3,3.6);
\draw[rounded corners,dashed,thick,fill opacity=0.7,fill=white!100] (-3,0.6) rectangle (-1.7,3.3)
;

\node (R1) at (3.7,2) {$R_F$};
\node (R2) at (-3.4,2) {$R'_{G_*}$};

\node (v1) at (0,3) {$H$};
\node (v2) at (1.5,3) {$F$};
\node (v3) at (-1,3) {$G$};
\node (v4) at (-2,3) {$G_*$};
\node (v5) at (2.75,3) {$F'$};

\draw[thick,<->](v1)--(v2)
node[midway,above] {\tiny $\alpha-\beta$};
\draw[thick,<->](v1)--(v3)
node[midway,above] {\tiny $\nu$};
\draw[thick,<->](v4)--(v3)
node[midway,above] {\tiny $\nu$};
\draw[thick,<->](v2)--(v5)
node[midway,above] {\tiny $2\nu$};

\node (v6) at (0,1) {$H[Q]$};
\node (v7) at (-2.25,1) {$G_*[Q]$};
\node (v8) at (2.7,1) {$J$};

\draw[thick,<->](v6)--(v7)
node[midway,above] {\tiny $4\nu$};
\draw[thick,<->](v8) .. controls (1.5,0) and (-1,0) ..
node[midway,below] {\tiny $\alpha-\beta+6\nu$} (v7);

\end{tikzpicture}
\caption{The proof strategy for Theorem~\ref{thm: estimable} (Case 1).}\label{fig:estimable}
\end{figure}

Let us fix an $n$-vertex $k$-graph $H$.
By~\eqref{eq: def T mu* q},
 we can apply the regular approximation lemma (Theorem~\ref{thm: RAL}) with $H, \eta, \nu/2, \mu, T$ playing the roles of $H, \eta, \nu, \epsilon, t_0$. 
Hence there exists a $k$-graph $G$ on $V(H)$, $\ba^{\sQ} \in \mathbb{N}^{k-1}$ and a family of partitions $\sQ = \sQ(k-1,\ba^{\sQ})$ satisfying the following.
\begin{enumerate}[label=(G\Alph*)]
\item\label{item:GA} $\sQ$ is $(\eta,\mu(\ba^{\sQ}),\ba^{\sQ})$-equitable and $T$-bounded,
\item\label{item:GB} $G$ is perfectly $\mu(\ba^{\sQ})$-regular with respect to $\sQ$, and
\item\label{item:GC} $|G\triangle H|\leq \nu \binom{n}{k}$.
\end{enumerate}
%\COMMENT{Let $t:= \norm{\ba^{\sQ}}$. Here, $t$ may be much smaller than $T$. Thus $\mu(t,\dots,t) >1/T$ might happen while $\mu(t,\dots,t)\leq 1/t$. Thus we need to use $t$ instead of $T$ to avoid such incidence.}

Property (GB) implies that there exists a density function $d^{G}_{\ba^{\sQ},k}$ such that
for all $\hat{\bx}\in \hat{A}(k,k-1,\ba^{\sQ})$ the $k$-graph
 $G$ is $(\mu(\ba^{\sQ}), d^{G}_{\ba^{\sQ},k}(\hat{\bx}))$-regular with respect to $\hat{Q}^{(k-1)}(\hat{\bx})$.
Thus 
\begin{align}\label{eq:Qpart}
	\text{$\{\sQ^{(j)}\}_{j=1}^{k-1}$ is a $(\mu(\ba^{\sQ}), \ba^{\sQ},d^{G}_{\ba^{\sQ},k})$-equitable partition of $G$.}
\end{align}

\smallskip

\noindent
\textbf{Case 1:} \emph{$H$ is $(\alpha-\beta)$-close to satisfying $\bP$.}

\smallskip

\noindent
Hence there exists a $k$-graph $F$ which
satisfies $\bP$ and which is $(\alpha-\beta)$-close to $H$. 
By \eqref{eq: Pr cF}, we have the following:
\begin{align}\label{eq: Pr cF Fk}
F \text{ is } (\alpha-\beta)\text{-close to } H \text{ and }\mathbf{Pr}(\cF, F) \geq 2/3.
\end{align}

We choose a suitable partition of the edges of $F$ and its complement by setting
\begin{align*}
\{ F^{(k)}_1,\dots, F^{(k)}_s\} 
&:= \{ F\cap \cK_{k}(\hat{Q}^{(k-1)}(\hat{\bx}))  :\hat{\bx}\in \hat{A}(k,k-1,\ba^{\sQ}) \} \sm\{\emptyset\},\\
\{F^{(k)}_{s+1},\dots, F^{(k)}_{s+s'}\} 
&:= \left( \left\{ \cK_{k}(\hat{Q}^{(k-1)}(\hat{\bx}))  \setminus F:\hat{\bx}\in \hat{A}(k,k-1,\ba^{\sQ}) \right\} \right. \\
& \quad\cup \left. \left\{ F\setminus \cK_{k}(\sQ^{(1)}),\enspace\binom{V}{k}\setminus (\cK_{k}(\sQ^{(1)})\cup F )\right\}\right)\setminus \{\emptyset\}.
\end{align*}
Note that by Proposition~\ref{prop: hat relation}(viii), $s+s'\leq \norm{\ba^{\sQ}}^{4^{k}}$.
Let 
\begin{align}\label{eq: T' definition constant T'}
T':=t_{\ref{RAL(k)}}(k, \norm{\ba^{\sQ}},\norm{\ba^{\sQ}}^{4^{k}},\eta,\nu,\overline{\epsilon}).
\end{align}
Let 
$$\sQ^{(k)}:= \{ \cK_{k}(\hat{Q}^{(k-1)}( \hat{\bx})):\hat{\bx}\in \hat{A}(k,k-1,\ba^{\sQ})\}.$$
Clearly, $\{F^{(k)}_1,\dots, F^{(k)}_{s+s'}\} \prec \sQ^{(k)}.$
We apply the `partition version' of the regular approximation lemma (Lemma~\ref{RAL(k)}) with the following objects and parameters. (
Indeed this is possible by \ref{item:GA} and \eqref{def: mu ba ll}.) \newline

{\small
\begin{tabular}{c|c|c|c|c|c|c|c|c|c}%|c|c|c|c|c|c|c|c}
object/parameter & $ \{\sQ^{(j)}\}_{j=1}^{k}$ & $ \{F^{(k)}_1,\dots, F^{(k)}_{s+s'}\}$ & $k$&  $s+s'$ & $\eta$ & $\nu$ & $\overline{\epsilon}$ & $T'$ & $\norm{\ba^{\sQ}}$\\ \hline
playing the role of & $\{\sQ^{(j)}\}_{j=1}^{k}$ & $\sH^{(k)}$ & $k$ &  $s$ & $\eta$ & $\nu$ & $\epsilon$ & $t$ & $o$
 \\ 
\end{tabular}
}\newline \vspace{0.2cm}

\noindent
We obtain a family of partitions $\sP= \sP(k-1,\ba^{\sP})$ and $k$-graphs $\{F'^{(k)}_i\}_{i=1}^{s+s'}$ such that 
\begin{itemize}
\item[(P1)$_{\ref{thm: estimable}}$] $\sP$ is $(\eta,\overline{\epsilon}(\ba^{\sP}),\ba^{\sP})$-equitable and $T'$-bounded,
\item[(P2)$_{\ref{thm: estimable}}$] $\sP \prec \{\sQ^{(j)}\}_{j=1}^{k-1}$,
\item[(P3)$_{\ref{thm: estimable}}$] $F'^{(k)}_i$ is perfectly $\overline{\epsilon}(\ba^{\sP})$-regular with respect to $\sP$, and
\item[(P4)$_{\ref{thm: estimable}}$] $\sum_{i=1}^{s} |F^{(k)}_i\triangle F'^{(k)}_i| \leq \nu \binom{n}{k}$.
\end{itemize}
Note that by \eqref{eq: def eta nu q' beta k const}--\eqref{eq:defb}
\begin{align}\label{eq:muT}
	\mu(\ba^{\sQ})\ll 1/T', \overline{\epsilon}(\ba^{\sP})\enspace\text{ and }\enspace 
	\overline{\epsilon}(\ba^{\sP})\ll \norm{\ba^{\sP}}^{-1} \leq 1/a_1^{\sP}\leq 1/a_1^{\sQ}\leq \eta\ll \nu.
\end{align}
Let $F':=\bigcup_{i=1}^{s} F'^{(k)}_i$. Then by (P4)$_{\ref{thm: estimable}}$
\begin{align}\label{eq: F triangle F'}
|F \triangle F'| \leq \nu \binom{n}{k} + |F\setminus \cK_{k}(\sQ^{(1)})| \stackrel{(\ref{eq: eta a1})}{\leq} (\nu+k^2\eta)\binom{n}{k} \leq 2\nu \binom{n}{k} .
\end{align}
By (P3)$_{\ref{thm: estimable}}$ and Lemma~\ref{lem: union regularity} 
\begin{equation}\label{eq: F' is also s epsilon regular}
\begin{minipage}[c]{0.8\textwidth}\em
$F'$ is perfectly $s\overline{\epsilon}(\ba^{\sP})$-regular with respect to $\sP$.
\end{minipage}
\end{equation}
Together with \eqref{eq: Pr cF Fk} and \eqref{eq: F triangle F'}, Proposition~\ref{prop: mathbf Pr doesn't change much} implies that 
\begin{align}\label{eq: Pr F'}
\mathbf{Pr}(\cF,F')\geq 2/3 - 2q'^{k} \nu \geq 2/3 - \nu^{1/2}.
 \end{align}
Observe that so far we introduced a $k$-graph $G$ that is very close to $H$ and a $k$-graph $F'$ that is very close to $F$,
where $G$ and $F'$ satisfy very strong regularity conditions. 
We now modify $G$ slightly to obtain $G_{*}$ so that $\sP$ is an equitable partition of $G_{*}$.

By \eqref{eq:Qpart}, \eqref{eq:muT}, (P1)$_{\ref{thm: estimable}}$ and (P2)$_{\ref{thm: estimable}}$,
we can apply Proposition~\ref{eq: sP prec sQ then sP is good partition} with 
$\mu(\ba^{\sQ}), \overline{\epsilon}(\ba^{\sP}), \nu$, $\norm{\ba^{\sP}}, \sP$, $\{\sQ^{(j)}\}_{j=1}^{k-1}, G$ and $d^{G}_{\ba^{\sQ},k}$ playing the roles of $\epsilon, \epsilon', \nu, T, \sP, \sQ, H^{(k)}$ and $d_{\ba^{\sQ},k}$ to obtain a density function $d^{G}_{\ba^{\sP},k}:\hat{A}(k,k-1,\ba^{\sP}) \rightarrow [0,1]$ and an $n$-vertex $k$-graph $G_*$ such that 
\begin{enumerate}[label=(G$_*$\arabic*)]
\item\label{item:G*1} $\sP$ is an $(\overline{\epsilon}(\ba^{\sP}),\ba^{\sP},d^{G}_{\ba^{\sP},k})$-equitable partition of $G_*$,
\item\label{item:G*2}  $|G\triangle G_*| \leq \nu \binom{n}{k}.$
\end{enumerate}

Let $\epsilon_*:= \overline{\epsilon}(\ba^{\sP})$.
Since $\norm{\ba^{\sQ}}\leq T$ by \ref{item:GA}, it follows from \eqref{eq: def T mu* q} and \eqref{eq:muT} that $ \mu_*\leq \mu(\ba^{\sQ}) \ll\epsilon_* $.
For each $\hat{\by}\in \hat{A}(k,k-1,\ba^{\sP})$, 
let $d^{F}_{\ba^{\sP},k}(\hat{\by})\in \{0,\mu_*^2, 2\mu_*^2, \dots, 1\} $  such that 
$$d^{F}_{\ba^{\sP},k}(\hat{\by}) = d(F' \mid \hat{P}^{(k-1)}(\hat{\by}))\pm \mu_*^2.$$
Thus $d^{F}_{\ba^{\sP},k}(\hat{\by})$ satisfies (R2)$_{\ref{thm: estimable}}$.

Let $\epsilon^*$ be a number satisfying  (R1)$_{\ref{thm: estimable}}$ with $\ba^{\sP}$ playing the role of $\ba$ and such that $6s\epsilon_* \leq \epsilon^{*}\leq 7s\epsilon_*$.\COMMENT{Note that $s\leq \norm{\ba^{\sQ}}^{4^k}$ and $\epsilon_* \ll \norm{\ba^{\sP}}^{-4^k} \leq \norm{\ba^{\sQ}}^{-4^k}$. Thus $\epsilon_* \ll 1/s$. }
Then by \eqref{eq: F' is also s epsilon regular}, $\sP$ is an $(\epsilon^*, \ba^{\sP},d_{\ba^{\sP},k}^{F})$-equitable partition of $F'$. So 
\begin{equation}\label{eq: F' satisfies regularity instance RF}
\begin{minipage}[c]{0.8\textwidth}\em
$F'$ satisfies the regularity instance $R_F:=(\epsilon^{*},\ba^{\sP}, d^{F}_{\ba^{\sP},k})$ and $R_F \in \bR$.
\end{minipage}
\end{equation}
 Also, \ref{item:G*1} implies that $G_*$ satisfies the regularity instance $R_{G_*}=(\epsilon_{*}, \ba^{\sP},d^{G}_{\ba^{\sP},k})$.
 
Note that  \eqref{eq: Pr cF Fk}, \eqref{eq: F triangle F'}, \ref{item:GC} and \ref{item:G*2} together imply that $G_*$ and $F'$ are $(\alpha - \beta + 4\nu)$-close.
Thus Lemma~\ref{lem: hypergraph close then density close} implies that 
\begin{align}\label{eq: dist close}
\dist(d^{G}_{\ba^{\sP},k}, d^{F}_{\ba^{\sP},k}) \leq \alpha- \beta+5\nu.
\end{align}

Recall that the algorithm $\bT_E(n,\alpha,\beta)$ chooses a random $q$-set $Q$ of $V(H)$. 
Define events $\cE_0$ and $\cE_1$ as follows:
\begin{itemize}
	\item[($\cE_0$)] $|G_*[Q]\triangle H[Q]|= \frac{q^{k}}{n^{k}}|G_*\triangle H| \pm \nu\binom{q}{k}$,
	\item[($\cE_1$)] $G_*[Q]$ satisfies the regularity instance $R'_{G_*}:=( 2\epsilon_{*}, \ba^{\sP},d^{G}_{\ba^{\sP},k})$.
\end{itemize}
\COMMENT{It is a regularity instance even with $2\epsilon_*$ because of the definition of the function $\overline{\epsilon}$.}
Lemma~\ref{lem: random subset edge size} implies
$\mathbb{P}[\cE_0] \geq 1 - e^{-\nu^{3}q},$
and Lemma~\ref{lem: random choice2}\COMMENT{We can apply the lemma since $1/q \ll c\ll \mu_* \leq \mu(\ba^{\sQ}) \ll \overline{\epsilon}(\ba^{\sP}) = \epsilon_*$ by \eqref{eq: def T mu* q} and \eqref{eq:muT}.} implies 
$\mathbb{P}[\cE_1] \geq 1 - e^{-cq}.$
Hence, as $0<c\ll \nu$, 
\begin{align}\label{eq: probability cE0 wedge cE1}
\mathbb{P}[\cE_0\wedge \cE_1 ] \geq 1 - 2e^{-cq}\geq \frac{2}{3}.
\end{align}
It suffices to show that if $\cE_0\wedge \cE_1$ holds, then $\bT_E(n,\alpha,\beta)$ always accepts $H$. Suppose in the following that $\cE_0\wedge \cE_1$ holds.
Then ($\cE_0$), \eqref{eq: dist close} and Lemma~\ref{lem: density close then hypergraph close} with $q, G_*[Q],d^{G}_{\ba^{\sP},k},d^{F}_{\ba^{\sP},k}$ 
playing the roles of $n, H^{(k)},d^{H}_{\ba^{\sP},k},d^{G}_{\ba^{\sP},k}$ imply that there exists a $q$-vertex $k$-graph $J$ on vertex set $Q$ such that 
\begin{itemize}
\item[(J1)$_{\ref{thm: estimable}}$] $J$ is $(\alpha- \beta+6\nu)$-close to $G_*[Q]$, and
\item[(J2)$_{\ref{thm: estimable}}$] $J$ satisfies the regularity instance $R_F = (\epsilon^{*},\ba^{\sP}, d^{F}_{\ba^{\sP},k})$.
\end{itemize}
To obtain (J2)$_{\ref{thm: estimable}}$,  we use that $6s\epsilon_* \leq \epsilon^{*} \leq 7s\epsilon_*$.
Note that $(\cE_0)$, \ref{item:GC} and \ref{item:G*2} together imply 
\begin{align}\label{eq: diff between GQ HQ}
|G_*[Q]\triangle H[Q]| \leq 4\nu \binom{q}{k}.
\end{align}
Together with 
(J1)$_{\ref{thm: estimable}}$ and \eqref{eq: def eta nu q' beta k const}, this gives
\begin{equation*}%\label{eq: J close to H[Q]}
\begin{minipage}[c]{0.8\textwidth}\em
$J$ is $(\alpha-\beta/2)$-close to $H[Q]$.
\end{minipage}
\end{equation*}
Moreover, note $\epsilon^*\ll \norm{\ba^{\sP}}^{-1} \leq 1/a_1^{\sP}\ll \nu, 1/q',1/k$ by \eqref{eq: def eta nu q' beta k const} and \eqref{eq:muT}. Thus \eqref{eq: F' satisfies regularity instance RF} together with 
Corollary~\ref{cor: counting collection} implies that 
$$IC(\cF, d^{F}_{\ba^{\sP},k}) 
\geq \mathbf{Pr}(\cF, F') - \nu 
\stackrel{\eqref{eq: Pr F'}}{\geq} 2/3 -2\nu^{1/2} 
\geq 1/2.$$
Since we also have $R_F\in \bR$ by \eqref{eq: F' satisfies regularity instance RF} and $a^{\sP}_1 \geq 1/\eta$ by (P1)$_{\ref{thm: estimable}}$, it follows that $R_F \in \bF $. 
Thus there exists $J$ which is $(\alpha-\beta/2)$-close to $H[Q]$ and which satisfies $R_F\in \bF$. 
Hence \eqref{eq: algorithm} implies that $\bT_E(n,\alpha,\beta)$ accepts $H$.
Together with \eqref{eq: probability cE0 wedge cE1} this shows that whenever $H$ is $(\alpha-\beta)$-close to satisfying $\bP$, 
then  $\bT_E(n,\alpha,\beta)$ accepts $H$ with probability at least $2/3$.\COMMENT{at least $\mathbb{P}[\cE_0\wedge \cE_1] \geq 1 - 2e^{-cq} \geq 2/3.$}

\smallskip
\noindent
\textbf{Case 2:} \emph{$H$ is $\alpha$-far from satisfying $\bP$.}
\smallskip

\noindent
Again, we may use Theorem~\ref{thm: RAL} to show the existence of a family of partitions $\sQ=\sQ(k-1,\ba^{\sQ})$ and a $k$-graph $G$ on $V$ satisfying \ref{item:GA}--\ref{item:GC}. 
Then $G$ satisfies the regularity instance $R_{G} := (\mu(\ba^{\sQ}),\ba^{\sQ}, d^{G}_{\ba^{\sQ},k})$ for some density function $d^{G}_{\ba^{\sQ,k}}$ and \eqref{eq:Qpart} holds.
(By~\eqref{def: mu ba ll}, $R_G$ is indeed a regularity instance.)
Recall that $\bT_E(n,\alpha,\beta)$ chooses a random $q$-set $Q$ of $V(H)$. 
Define events $\cE_0'$ and $\cE'_1$ as follows.
\begin{itemize}
	\item[($\cE_0'$)] $|G[Q]\triangle H[Q]|= \frac{q^{k}}{n^{k}}|G\triangle H| \pm \nu\binom{q}{k}$,
	\item [($\cE_1'$)]$G[Q]$ satisfies the regularity instance $(2\mu(\ba^{\sQ}),\ba^{\sQ}, d^{G}_{\ba^{\sQ},k})$. 
\end{itemize}
Thus, if $\cE'_1$ occurs, 
then there exists a family of partitions $\sQ'=\sQ'(k-1,\ba^{\sQ})$ on $Q$ which is a $(2\mu(\ba^{\sQ}),\ba^{\sQ},d^{G}_{\ba^{\sQ},k})$-equitable partition of $G[Q]$. 
Note that $\sQ'$ is $\norm{\ba^{\sQ}}$-bounded.
Lemma~\ref{lem: random subset edge size} implies
$\mathbb{P}[\cE_0'] \geq 1 - e^{-\nu^{3}q},$
while Lemma~\ref{lem: random choice2} and~\eqref{eq: def T mu* q} imply that $\Pro[\cE'_1]\geq 1-e^{-cq}$.
Thus
\begin{align}\label{eq:probE'01}
	\mathbb{P}[\cE'_0\wedge \cE'_1]
	\geq 1 - 2e^{-cq} \geq 2/3.
\end{align}

 So it suffices to show that if $\cE'_0 \wedge \cE'_1$ holds, 
then $\bT_E(n,\alpha,\beta)$ rejects $H$. 
Assume for a contradiction that $\bT_E(n,\alpha,\beta)$ accepts $H$. 
This implies that there exists a $k$-graph $L$ on $Q$ so that 
\begin{equation}\label{eq: L is close and satisfies regularity instance}
\begin{minipage}[c]{0.8\textwidth}\em
$L$ is $(\alpha-\beta/2)$-close to $H[Q]$ and $L$ satisfies some regularity instance $R_L \in \bF$.
 \end{minipage}
 \end{equation}
Together with (R1)$_{\ref{thm: estimable}}$, \eqref{eq: def eta nu q' beta k const}, and the definition of $\bF$, 
Corollary~\ref{cor: counting collection} implies that 
\begin{align}\label{eq: Pr cF Qk 1/2 nu}
\mathbf{Pr}(\cF, L) \geq 1/2 - \nu.
\end{align}
\COMMENT{We can apply the corollary since the vector $\ba$ and constant $\epsilon_{L}$ `belonging' to $R_L$ satisfy $\epsilon_{L} \leq \overline{\epsilon}(\ba)^{1/2} \ll \norm{\ba}^{-1}\leq 1/a_1 \leq \eta\ll \nu,1/k,1/q'$.}
Our strategy is as follows.
We aim to construct a $k$-graph $M$ on $V$ whose structure is very similar to $L$.
However, this is hard to achieve directly as $L$ may not satisfy sufficiently strong regularity assumptions.
Thus we first approximate $L$ by a suitable $k$-graph $L'$.
Similarly, we also approximate $G$ (and thus $H$) by suitable $k$-graphs $G'$ and $G^*$.
Based on this, we construct $M$ which is $(\alpha-\beta/2+ 10\nu)$-close to $H$;
that is, almost as close as $L$ and $H[Q]$.
Based on \eqref{eq: L is close and satisfies regularity instance} and \eqref{eq: Pr cF Qk 1/2 nu},
we verify that $M$ is $(\beta/4)$-close to satisfying $\bP$, yielding a contradiction to the assumption that $H$ is $\alpha$-far from satisfying $\bP$.

Let 
\begin{align*}
\{ L_1,\dots, L_s\} 
&:= \{ L\cap \cK_{k}(\hat{Q'}^{(k-1)}(\hat{\bx}))  :\hat{\bx}\in \hat{A}(k,k-1,\ba^{\sQ}) \}\sm\{\es\},\\
\{L_{s+1},\dots, L_{s+s'}\}
&:= \left( \left\{ \cK_{k}(\hat{Q'}^{(k-1)}(\hat{\bx}))  \sm L:\hat{\bx}\in \hat{A}(k,k-1,\ba^{\sQ}) \right\} \right.\\
&\quad \cup \left.\left\{ L\setminus \cK_{k}(\sQ'^{(1)}),\enspace \binom{Q}{k}\setminus (\cK_{k}(\sQ'^{(1)})\cup  L )\right\}\right) \setminus\{\emptyset\}.
\end{align*}
Thus $s+s'\leq \norm{\ba^{\sQ}}^{4^{k}}$ by Proposition~\ref{prop: hat relation}(viii).
Let $\sQ'^{(k)}:= \{ \cK_{k}(\hat{Q'}^{(k-1)}( \hat{\bx})):\hat{\bx}\in \hat{A}(k,k-1,\ba^{\sQ})\}$. It follows that 
$\{L_1,\dots, L_{s+s'}\}\prec \sQ'^{(k)}$.
Again, by \eqref{def: mu ba ll} and \eqref{eq: def T mu* q}
we can apply the `partition version' of the regular approximation lemma (Lemma~\ref{RAL(k)}) with the following objects and parameters. \newline

{\small
\begin{tabular}{c|c|c|c|c|c|c|c|c|c|c}%|c|c|c|c|c|c|c|c}
object/parameter & $ \{\sQ'^{(j)}\}_{j=1}^{k}$ & $ \{L_1,\dots, L_{s+s'}\} $ & $k$&$q$&  $s+s'$ & $\eta$ & $\nu$ & $\overline{\epsilon}$ & $T'$  & $\norm{\ba^{\sQ}}$\\ \hline
playing the role of & $\{\sQ^{(j)}\}_{j=1}^{k}$ & $\sH^{(k)}$ &$k$& $n$ &  $s$ & $\eta$ & $\nu$ & $\epsilon$ & $t$ & $o$
 \\ 
\end{tabular}
}\newline \vspace{0.2cm}

\noindent
We obtain a family of partitions $\sP'=\sP'(k-1,\ba^{\sP'})$ and $k$-graphs $\{L'_i\}_{i=1}^{s+s'}$ such that 
\begin{itemize}
\item[(P$'$1)$_{\ref{thm: estimable}}$] $\sP'$ is $(\eta,\overline{\epsilon}(\ba^{\sP'}),\ba^{\sP'})$-equitable and $T'$-bounded, and $a^{\sQ}_i$ divides $a^{\sP'}_i$ for each $i\in [k-1]$,
\item[(P$'$2)$_{\ref{thm: estimable}}$] $\sP' \prec \{\sQ'^{(j)}\}_{j=1}^{k-1}$,
\item[(P$'$3)$_{\ref{thm: estimable}}$] $L'_i$ is perfectly $\overline{\epsilon}(\ba^{\sP'})$-regular with respect to $\sP'$,
\item[(P$'$4)$_{\ref{thm: estimable}}$] $\sum_{i=1}^{s} |L_i\triangle L'_i| \leq \nu \binom{q}{k}$.
\end{itemize}
Let $L' := \bigcup_{i=1}^{s} L'_i$.
Then the argument for \eqref{eq: F triangle F'} shows that
\begin{equation}\label{eq: L L' close}
\begin{minipage}[c]{0.8\textwidth}\em
$L'$ is $2\nu$-close to $L$.
\end{minipage}
\end{equation}
By (P$'$3)$_{\ref{thm: estimable}}$ and Lemma~\ref{lem: union regularity} we have
\begin{equation}\label{eq: L' is also s epsilon regular}
\begin{minipage}[c]{0.8\textwidth}\em
$L'$ is perfectly $s\overline{\epsilon}(\ba^{\sP'})$-regular with respect to $\sP'$.
\end{minipage}
\end{equation}

Thus there exists a density function $d^{L'}_{\ba^{\sP'},k}:\hat{A}(k,k-1,\ba^{\sP'})\to[0,1]$ such that $\sP'$ is an $(s\overline{\epsilon}(\ba^{\sP'}),d^{L'}_{\ba^{\sP'},k})$-partition of $L'$.
Hence
\begin{equation}\label{eq:L'reginstance}
\begin{minipage}[c]{0.8\textwidth}\em
$L'$ satisfies the regularity instance $R_{L'}:=(s\overline{\epsilon}(\ba^{\sP'}),\ba^{\sP'},d^{L'}_{\ba^{\sP'},k})$.
\end{minipage}
\end{equation}
From Proposition~\ref{prop: mathbf Pr doesn't change much}, \eqref{eq: Pr cF Qk 1/2 nu} and \eqref{eq: L L' close} we can conclude that
\begin{align}\label{eq: Pr cF j'}
 \mathbf{Pr}(\cF, L') 
 \geq \mathbf{Pr}(\cF, L) -2q'^{k}\nu  
 \geq 1/2 - \nu^{1/2} > 1/3+2\nu.
\end{align}

By ($\cE_1'$), (P$'$1)$_{\ref{thm: estimable}}$, (P$'$2)$_{\ref{thm: estimable}}$, and the analogue of \eqref{eq:muT},
we can apply Proposition~\ref{eq: sP prec sQ then sP is good partition} with 
$2\mu(\ba^{\sQ})$, $\overline{\epsilon}(\ba^{\sP'})$, $\nu$, $\sP'$, $\{\sQ'^{(j)}\}_{j=1}^{k-1}$, $G[Q]$, $q$ and $d^G_{\ba^{\sQ},k}$ playing the roles of $\epsilon$, $\epsilon'$, $\nu$, $\sP$, $\sQ$, $H^{(k)}$, $n$ and $d_{\ba^{\sQ},k}$ to obtain a $k$-graph $G'$ on $Q$ and a density function $d^{G'}_{\ba^{\sP'},k}:\hat{A}(k,k-1,\ba^{\sP'}) \rightarrow [0,1]$ such that 
\begin{enumerate}[label=(G$'$\arabic*)]
\item\label{item:G'1} $\sP'$ is an $(\overline{\epsilon}(\ba^{\sP'}),d^{G'}_{\ba^{\sP'},k})$-partition of $G'$,
\item\label{item:G'2} $|G' \triangle G[Q]| \leq \nu \binom{q}{k}.$
\item\label{item:G'3} $
d^{G'}_{\ba^{\sP'},k}(\hat{\by})= \left\{ \begin{array}{ll}
d^{G}_{\ba^{\sQ},k}(\hat{\bx}) & \text{ if } \enspace \exists \hat{\bx}\in \hat{A}(k,k-1,\ba^{\sQ}) : \cK_k(\hat{P'}^{(k-1)}(\hat{\by}))\subseteq \cK_k(\hat{Q'}^{(k-1)}(\hat{\bx})), \\
0 & \text{ if }\enspace \cK_k(\hat{P'}^{(k-1)}(\hat{\by}))\cap \cK_k(\sQ'^{(1)}) = \emptyset.
\end{array}\right.$
\end{enumerate}
\COMMENT{Proposition~\ref{prop: hat relation}(xi) implies that this defines $d^{G}_{\ba^{\sP'},k}(\hat{\by})$ for all $\hat{\by}\in \hat{A}(k,k-1,\ba^{\sP'})$.}
For each $\hat{\bx} \in \hat{A}(k,k-1,\ba^{\sQ})$, let 
\begin{align}\label{eq: definition of hat B hat bx}
\hat{B}(\hat{\bx}):= \{ \hat{\by}\in \hat{A}(k,k-1,\ba^{\sP'}) : \cK_{k}(\hat{P'}^{(k-1)}(\hat{\by}))\subseteq \cK_k(\hat{Q'}^{(k-1)}(\hat{\bx}))\}.
\end{align}
Note that \ref{item:GC}, ($\cE_0'$), \eqref{eq: L is close and satisfies regularity instance}, \eqref{eq: L L' close} and \ref{item:G'2} imply that $L'$ is $(\alpha - \beta/2 + 6\nu)$-close to $G'$. Since we also have \ref{item:G'1} we can apply Lemma~\ref{lem: hypergraph close then density close} to see that
\begin{align}\label{eq: G' J' density dist}
\dist(d^{G'}_{\ba^{\sP'},k},d^{L'}_{\ba^{\sP'},k})\leq \alpha-\beta/2 + 7\nu.
\end{align}

Recall that by (P$'$1)$_{\ref{thm: estimable}}$, $a_i^{\sQ}$ divides $a_i^{\sP'}$ for all $i\in [k-1]$,
and that $\sQ$ is an $(\eta, \mu(\ba^{\sQ}),\ba^{\sQ})$-equitable family of partitions on $V$ by \ref{item:GA}.
So (recalling \eqref{def: mu ba ll}) we can apply Lemma~\ref{lem: partition refinement} with $\sQ, \ba^{\sQ}, \ba^{\sP'}, \mu(\ba^{\sQ}), V$ playing the roles of $\sP, \ba, \bb, \epsilon, V$ to obtain  a family of partitions $\sP^*$ on $V$ satisfying the following.
\begin{itemize}
\item[(P$^*$1)$_{\ref{thm: estimable}}$] $\sP^*=\sP^*(k-1,\ba^{\sP'})$ is a $(1/a_1^{\sP'},  \mu(\ba^{\sQ})^{1/3}, \ba^{\sP'})$-equitable family of partitions on~$V$.
\item[(P$^*$2)$_{\ref{thm: estimable}}$] $\sP^* \prec \sQ$.
\end{itemize}
Moreover, we choose an appropriate $\ba^{\sP'}$-labelling so that for all $\hat{\bx} \in \hat{A}(k,k-1,\ba^{\sQ})$ and $\hat{\by}\in \hat{A}(k,k-1,\ba^{\sP'})$, we have
\begin{align}\label{eq: relation of hat B belong}
\cK_k(\hat{P^*}^{(k-1)}(\hat{\by}))\subseteq \cK_k(\hat{Q}^{(k-1)}(\hat{\bx}))
\text{ if and only if } \hat{\by}\in \hat{B}(\hat{\bx}).
\end{align}

By \eqref{eq:Qpart},  (P$^*$1)$_{\ref{thm: estimable}}$ and (P$^*$2)$_{\ref{thm: estimable}}$,
we can apply Proposition~\ref{eq: sP prec sQ then sP is good partition}\COMMENT{
To see that Proposition~\ref{eq: sP prec sQ then sP is good partition} can be applied,
recall that $\sQ$ is a $(\mu(\ba^{\sQ}), \ba^{\sQ},d_{\ba^{\sQ},k})$-equitable partition of $G$.
Moreover, $\mu(\ba^{\sQ})\ll \overline{\epsilon}(\ba^{\sP'})\ll \norm{\ba^{\sP'}}^{-1}\leq 1/a_1^{\sQ}\leq \eta\ll\nu.$ Also by (P$^*$1)$_{\ref{thm: estimable}}$, $\sP^{*}$ is $(1/a_1^{\sP'}, \overline{\epsilon}(\ba^{\sP'}), \ba^{\sP'})$-equitable.
}
%%%
with $\mu(\ba^{\sQ})$, $\overline{\epsilon}(\ba^{\sP'})$, $\sP^*$, $\sQ$, $G$ and $d^{G}_{\ba^{\sQ},k}$ 
playing the roles of $\epsilon$, $\epsilon'$, $\sP$, $\sQ$, $H$ and $d_{\ba^{\sQ},k}$ 
to obtain an $n$-vertex $k$-graph $G^*$ on $V$ and density function $d^{G^*}_{\ba^{\sP^*},k}:\hat{A}(k,k-1,\ba^{\sP'}) \rightarrow [0,1]$ such that 

\begin{enumerate}[label=(G$^*$\arabic*)]
\item\label{item:G**1} $\sP^*$ is an $(\overline{\epsilon}(\ba^{\sP'}),\ba^{\sP'},d^{G^*}_{\ba^{\sP'},k})$-equitable partition of $G^*$,
\item\label{item:G**2} $|G \triangle G^*| \leq \nu \binom{n}{k}.$
\item\label{item:G**3} $d^{G^*}_{\ba^{\sP'},k}(\hat{\by})= \left\{ \begin{array}{ll}
d^{G}_{\ba^{\sQ},k}(\hat{\bx}) & \text{ if } \enspace \exists \hat{\bx}\in \hat{A}(k,k-1,\ba^{\sQ}) : \cK_k(\hat{P^*}^{(k-1)}(\hat{\by}))\subseteq \cK_k(\hat{Q}^{(k-1)}(\hat{\bx})), \\
0 & \text{ if }\enspace \cK_k(\hat{P^*}^{(k-1)}(\hat{\by}))\cap \cK_k(\sQ^{(1)}) = \emptyset.
\end{array}\right.$
\end{enumerate}

This together with \ref{item:G'3}, \eqref{eq: definition of hat B hat bx} and \eqref{eq: relation of hat B belong} implies that $d^{G^*}_{\ba^{\sP'},k} = d^{G'}_{\ba^{\sP'},k}$.
Using this with \eqref{eq: G' J' density dist}, \ref{item:G**1}, and Lemma~\ref{lem: density close then hypergraph close} 
(applied with $G^*,\sP^*,\overline{\epsilon}(\ba^{\sP'}),\nu,d_{\ba^{\sP'},k}^{L'}$ playing the roles of $H^{(k)},\sP,\epsilon,\nu,d_{\ba,k}^G$),
we conclude that  there exists an $n$-vertex $k$-graph $M$ on vertex set $V$ such that 
\begin{itemize}
\item[(M1)$_{\ref{thm: estimable}}$] $M$ is $(\alpha-\beta/2 + 8\nu)$-close to $G^*$, and 
\item[(M2)$_{\ref{thm: estimable}}$] $M$ satisfies the regularity instance $R_{M}:=(3\overline{\epsilon}(\ba^{\sP'}),\ba^{\sP'}, d^{L'}_{\ba^{\sP'},k})$.
\end{itemize}
By \ref{item:GC}, \ref{item:G**2} and (M1)$_{\ref{thm: estimable}}$, we conclude that
\begin{equation}\label{eq: M and H are close}
\begin{minipage}[c]{0.8\textwidth}\em
$M$ is $(\alpha-\beta/2+ 10\nu)$-close to $H$.
\end{minipage}
\end{equation}
On the other hand, 
by \eqref{eq:L'reginstance}, (M2)$_{\ref{thm: estimable}}$, and two applications of Corollary~\ref{cor: counting collection}, \COMMENT{We can use this by the definition of $\overline{\epsilon}$. See previous comment for complete hierarchy.}%(applied with $M, R_M,q'$ playing the roles of $H, R,\ell$),
we have
\begin{eqnarray*}
\mathbf{Pr}(\cF, M) 
\geq IC(\cF, d^{L'}_{\ba^{\sP'},k}) -\nu
\geq \mathbf{Pr}(\cF, L') -2\nu 
\stackrel{\eqref{eq: Pr cF j'}}{>}  1/3.
\end{eqnarray*}
Thus \eqref{eq: Pr cF} implies that $M$ is not $(\beta/4)$-far from satisfying $\bP$;
that is, $M$ is $(\beta/4)$-close to satisfying $\bP$.
Together with \eqref{eq: M and H are close}, 
this implies that
$H$ is $(\alpha-\beta/2 + 10\nu + \beta/4)$-close to satisfying $\bP$.
But $\alpha-\beta/2+ 10\nu +\beta/4 < \alpha$ by \eqref{eq: def eta nu q' beta k const}, 
so $H$ is $\alpha$-close to satisfying $\bP$, a contradiction.
Thus as long as $\cE'_0\wedge \cE'_1$ holds (which, by \eqref{eq:probE'01}, happens with probability at least $2/3$), 
$\bT_E(n,\alpha,\beta)$ rejects $H$.
This completes the proof that $\bP$ is $(n,\alpha,\beta)$-estimable.
\end{proof}

\section{Regular reducible hypergraph properties are testable}\label{sec:regredtest}

In this section we derive our main theorem from results stated and proved so far.
\begin{proof}[Proof of Theorem~\ref{thm:main}]
By Lemma~\ref{lem: main 1}, (a) implies (c) and Theorem~\ref{thm: estimable} implies that (a) and (b) are equivalent. 
It remains to show (c) implies (a).

Suppose that a $k$-graph property $\bP$ is regular reducible. 
Fix $0<\alpha <1$ and $n\geq k$.
We will now introduce an algorithm which distinguishes $n$-vertex $k$-graphs satisfying $\bP$ from $n$-vertex $k$-graphs which are $\alpha$-far from satisfying $\bP$.
Since $\bP$ is a regular reducible property, 
by Definition~\ref{def: regular reducible},
there exist $r:=r_{\ref{def: regular reducible}}(\alpha/4,\bP)$ and a collection $\cR=\cR(n,\alpha/4,\bP)$ of at most $r$ regularity instances each of complexity at most $r$ satisfying the following for every $n$-vertex $k$-graph $H$. 
(Note that we may assume that $r\geq 100$.)
\begin{itemize}
\item[(R1)$_{\ref{thm:main}}$] If $H$ satisfies $\bP$, then $H$ is $\alpha/4$-close to satisfying $R$ for some $R\in \cR$.
\item[(R2)$_{\ref{thm:main}}$] If $H$ is $\alpha$-far from satisfying $\bP$, then $H$ is $3\alpha/4$-far from satisfying $R$ for all $R\in \cR$.
\end{itemize}
By Theorems~\ref{thm: regularity instance is testable}~and~\ref{thm: estimable}, 
for any $R\in \cR$, 
there exist a function $q_k:(0,1)\rightarrow \mathbb{N}$ and an algorithm $\bT_R=\bT(n,\alpha)$ which distinguishes $n$-vertex $k$-graphs which are $\alpha/4$-close to satisfying $R$
from $n$-vertex $k$-graphs which are $3\alpha/4$-far from satisfying $R$ with probability at least $2/3$, 
by making at most $g_k(\alpha)$ queries.

Now we let $\bT'_R$ be an algorithm which independently applies 
the algorithm $\bT_R$ exactly $6r+1$ times on an input $n$-vertex $k$-graph $H$ and accepts or rejects depending on the majority vote.
Let $\bT_{1},\dots, \bT_{6r+1}$ denote these independent repetitions.
Let $$X_i:=\left\{\begin{array}{ll} 1 & \text{ if $\bT_i$ accepts, }\\
0& \text{ if $\bT_i$ rejects}.\end{array}\right.$$
Let $X:=\sum_{i=1}^{6r+1}X_i$.
Suppose first that $H$ is $\alpha/4$-close to satisfying $R$. Then $\mathbb{P}[X_{i}]\geq 2/3$ by the definition of $\bT_R$, and so $\mathbb{E}[X ] \geq 4r$.
Thus by Lemma~\ref{lem: chernoff} and the fact that $r\geq 100$, we obtain
$$\mathbb{P}[ \bT'_R \text{ accepts }H] 
= \mathbb{P}\left[ X \geq 3r+1 \right] 
\geq 1 - 2e^{-\frac{2r^2}{6r+1}} \geq 1 - \frac{1}{3r}.$$\COMMENT{Here, we get the last inequality since $r\geq 100$. Actually $r\geq 11$ is enough.}
Similarly, if $H$ is $3\alpha/4$-far from satisfying $R$, then $\mathbb{E}[X ] \leq 2r+1$ and 
$$\mathbb{P}[ \bT'_R \text{ rejects }H] 
= \mathbb{P}[ X \leq 3r ] 
\geq 1 - 2e^{-\frac{2r^2}{6r+1}} \geq 1 - \frac{1}{3r}.$$
Observe that $\bT'_R$ makes at most $(6r+1)g_k(\alpha)$ queries.
We now describe our tester $\bT=\bT(n,\alpha)$ which receives as an input an integer $n\geq k$, a real $\alpha>0$ and an $n$-vertex $k$-graph $H$.
\begin{equation}\label{eq: Tester}
\begin{minipage}[c]{0.8\textwidth}
\textit{Run $\bT'_R$ on the input $(n, H)$ for every $R\in \cR$. 
If there exists $R\in \cR$ such that $\bT'_R$ accepts $H$, then $\bT$ also accepts $H$, 
and if $\bT'_R$ rejects $H$ for all $R\in \cR$, then $\bT$ also rejects $H$.}
\end{minipage}
\end{equation}
Let us show that $\bT$ is indeed an $(n,\alpha)$-tester for $\bP$. 
First, assume that $H$ satisfies $\bP$.
By (R1)$_{\ref{thm:main}}$, 
there exists $R\in \cR$ such that $H$ is $\alpha/4$-close to $R$.
So $\bT'_R$ accepts $H$ with probability at least $1 - 1/(3r)$
and hence $\bT$ accepts $H$ with probability at least $1 - 1/(3r)\geq 2/3$.

Now assume that $H$ is $\alpha$-far from satisfying $\bP$. 
By (R2)$_{\ref{thm:main}}$, 
the $k$-graph $H$ is $3\alpha/4$-far from satisfying $R$ for every $R\in \cR$. 
Thus for every $R\in \cR$, 
the tester
$\bT'_R$ accepts $H$ with probability at most $1/(3r)$. 
This in turn implies that $\bT$ accepts $H$ with probability at most $1/3$.

Therefore, $\bT$ is an $(n,\alpha)$-tester for $\bP$,
which in particular implies that $\bP$ is testable.
\end{proof}

\section{Applications}\label{sec:application}

In this section we illustrate how Theorem~\ref{thm:main} can be applied,
first to counting subgraphs,
then to the maximum cut problem.

\subsection{Testing the injective homomorphism density}
We first  show how to test the (injective) homomorphism density,
where a homomorphism of a $k$-graph $F$ into a $k$-graph $H$ is a function $f:V(F)\to V(H)$
that maps edges onto edges.
Let $\inj(F,H)$ be the number of (vertex-)injective homomorphisms from $F$ into $H$
and let $t_{\inj}(F,H):=\inj(F,H)/(n)_{|V(F)|}$.

\begin{corollary}\label{cor:hom}
Suppose $p,\delta\in (0,1)$, $k\in \N\sm\{1\}$, and $F$ is a $k$-graph.
Let $\bP$ be the property that a \mbox{$k$-graph~$H$} satisfies $t_{\inj}(F,H)=p\pm \delta$.
Then $\bP$ is testable.
\end{corollary}

Before we continue with the proof of Corollary~\ref{cor:hom},
we state a simple proposition.

\begin{proposition}\label{prop:changehom}
Suppose $0< 1/n \ll  \nu , 1/k, 1/\ell$ and $\nu\ll \alpha, 1-\alpha$.
Let $F$ be an $\ell$-vertex $k$-graph and $H$ be an $n$-vertex $k$-graph.
If $t_\inj(F,H) = \alpha \pm \nu$ for some $\alpha \in (0,1)$, 
then there exists an $n$-vertex $k$-graph $G$ with 
$t_\inj(F,G) = \alpha \pm 1/n$ and $|G\triangle H| \leq (\frac{2\nu}{\min\{\alpha,1-\alpha\}})^{1/\ell} \binom{n}{k}$.
\end{proposition}
\begin{proof}
Suppose first that $t_\inj(F,H) >\alpha +1/n$.
Let $\epsilon:= (\frac{2 \nu}{\alpha})^{1/\ell}$.
By an averaging argument, 
there exists a subgraph $H'$ of $H$ on $\epsilon n$ vertices such that $t_\inj(F,H') >\alpha +1/n$.
Clearly, $|H'|\leq \epsilon \binom{n}{k}$.
Moreover, after removing all edges contained in $H'$ from $H$,
we reduce the number of injective homomorphisms from $F$ to $H$ by at least $\inj(F,H')\geq \alpha(\epsilon n)_{\ell}\geq \nu (n)_\ell$.
Thus if instead we remove a suitable number of these edges iteratively, 
we can reach a spanning subgraph $G$ of $H$ with $t_\inj(F,G)=\alpha\pm 1/n$ 
as any single edge removal decreases the number of homomorphisms by at most $2 n^{\ell-2}$.
The case $t_\inj(F,H) <\alpha -1/n$ works similarly.%
\COMMENT{
Let $\epsilon:=(\frac{2\nu}{1-\alpha})^{1/\ell}$.
Choose a subgraph $H'$ of $H$ on $\epsilon n$ vertices with $t_\inj(F,H')\leq \alpha -1/n$.
Then $|\overline{H}'|\leq \epsilon \binom{n}{k}$.
Adding all edges of $\overline{H}'$ increases the number of injective homomorphisms by at least $(1-\alpha)(\epsilon n)_\ell\geq \nu(n)_\ell$.
}
\end{proof}

%\begin{proposition}\label{prop: mathbf Pr doesn't change much}
%Suppose $n,k,q\in \mathbb{N}$ with $k\leq q\leq n$ and $G$ and $H$ are $n$-vertex $k$-graphs on vertex set $V$ and $\cF$ is a collection of $q$-vertex $k$-graphs.
%If $|G \triangle H| \leq \nu \binom{n}{k}$, then 
%$$\mathbf{Pr}(\cF, G) = \mathbf{Pr}(\cF, H) \pm q^k \nu.$$
%\end{proposition}

\begin{proof}[Proof of Corollary~\ref{cor:hom}]
Let $\ell:=|V(F)|$.
We may assume that $|F|>0$ as otherwise $t_\inj(F,H)=1$ for every $n$-vertex graph $H$ with $n\geq \ell$.
By Theorem~\ref{thm:main}, it suffices to verify that $\bP$ is regular reducible.

Suppose $\beta>0$. 
We may assume that $\beta \ll p-\delta, 1/\ell$ if $p-\delta>0$ and $\beta \ll 1-(p+\delta), 1/\ell$ if $p+\delta<1$.\COMMENT{This will enable us to apply Proposition~\ref{prop:changehom} at the end of the proof.}
We write $\beta':=\beta^{\ell+1}$ and $\beta'':=2^{-\binom{\ell}{k}}\beta'$.
We fix some function $\overline{\epsilon}:\mathbb{N}^{k-1}\rightarrow (0,1)$ such that $\overline{\epsilon}(\ba) \ll \norm{\ba}^{-k}$
for all $(a_1,\dots,a_{k-1})=\ba\in \N^{k-1}$.
We choose constants $  \epsilon, \eta$, and $n_0, T\in \N$
such that $ 1/n_0\ll \epsilon \ll 1/T \ll \eta \ll \beta, 1/k,1/\ell$.
In particular,
we have $n_0\geq n_{\ref{thm: RAL}}(\eta,\beta''\ell^{-k}/2,\overline{\epsilon})$,
$T\geq t_{\ref{thm: RAL}}(\eta,\beta''\ell^{-k}/2,\overline{\epsilon})$
and $\epsilon\ll \overline{\epsilon}(\ba)$ for all $\ba\in [T]^{k-1}$.
For simplicity, we consider only $n$-vertex $k$-graphs $H$ with $n\geq n_0$.

Let $\bI$ be the collection of regularity instances $R=(\epsilon'',\ba, d_{\ba,k})$ such that
\begin{itemize}
\item[(R1)$_{\ref{cor:hom}}$] $\epsilon''\in \{\epsilon, 2\epsilon, \dots, \lceil(\overline{\epsilon}(\ba))^{1/2}\epsilon^{-1}\rceil \epsilon\}$,
\item[(R2)$_{\ref{cor:hom}}$]  $\ba\in [T]^{k-1}$ and $a_1 > \eta^{-1}$, and
\item[(R3)$_{\ref{cor:hom}}$]  $d_{\ba,k}(\hat{\bx}) \in \{0,\epsilon^2, 2\epsilon^2,\dots, 1\}$ for every $\hat{\bx}\in \hat{A}(k,k-1,\ba)$.
\end{itemize}
Observe that by construction
$|\bI|$ is bounded by a function of $\beta$, $k$ and $\ell$.
We define
$$\cR:= \left\{  (\epsilon'',\ba,d_{\ba,k}) \in \bI :   \sum_{J\colon |V(J)|=\ell }\inj(F,J)\cdot IC(J,d_{\ba,k})/\ell!
=p \pm (\delta + \beta')\right\}.$$
First, suppose that an $n$-vertex $k$-graph $H$ satisfies $\bP$.
Then
\begin{align}\label{eq: Pr F}
\frac{1}{\ell!}\sum_{J\colon |V(J)|=\ell }\inj(F,J)\cdot \bPr(J, H)=t_\inj(F,H)= p \pm \delta.
\end{align}
By applying the regular approximation lemma (Theorem~\ref{thm: RAL}) 
with $H, \eta  ,\beta'' \ell^{-k}/2,\overline{\epsilon}$ playing the roles of $H,\eta,\nu,\epsilon$, 
we obtain a $k$-graph $G$ and a family of partitions $\sP= \sP(k-1,\ba^{\sP})$ such that
\begin{enumerate}[label=(\Roman*)]
\item\label{item:I} $\sP$ is $(\eta,\overline{\epsilon}(\ba^{\sP}),\ba^{\sP})$-equitable for some $\ba^{\sP} \in [T]^{k-1}$, 
\item\label{item:II} $G$ is perfectly $\overline{\epsilon}(\ba^{\sP})$-regular with respect to $\sP$, and 
\item\label{item:III} $|G\triangle H|\leq \beta'' \ell^{-k} \binom{n}{k}/2$.
\end{enumerate}
Let $\epsilon':= \overline{\epsilon}(\ba^{\sP})$.
By the choice of $\overline{\epsilon}$ and $\eta$, we conclude that
$0<\epsilon' \ll 1/\norm{\ba^{\sP}}\leq  1/a_1^{\sP} \ll  \beta, 1/k,1/\ell$ and
by the choice of $\epsilon$, we obtain  $\epsilon\ll \epsilon'$.
Note that if a $k$-graph $J$ is $(\epsilon',d)$-regular with respect to a $(k-1)$-graph $J'$, 
then $J$ is $(\epsilon'',d')$-regular with respect to $J'$ 
for some $d'\in \{0,\epsilon^2, 2\epsilon^2,\dots, 1\}$ and $\epsilon''\in \{\epsilon, 2\epsilon,\dots,  \lceil \epsilon'^{1/2}\epsilon^{-1}\rceil \epsilon\}\cap [2\epsilon',3\epsilon']$.
Thus there exists 
\begin{align}\label{eq: RG in bII}
R_G=(\epsilon'',\ba^{\sP} ,d^{G}_{\ba^{\sP},k})\in \bI
\end{align} 
such that $G$ satisfies $R_G$.

For every $\ell$-vertex $k$-graph $J$,
Proposition~\ref{prop: mathbf Pr doesn't change much} with \ref{item:III} 
and Corollary~\ref{cor: counting collection} imply that 
\begin{align}\label{eq: Pr - beta1}
IC(J,d^{G}_{\ba^{\sP},k}) = \bPr(J, G)\pm  \beta''/2 = \bPr(J, H) \pm \beta''.
\end{align}
Hence
\begin{eqnarray}\notag
	\frac{1}{\ell!}\sum_{J\colon |V(J)|=\ell }\inj(F,J)\cdot IC(J,d_{\ba^{\sP},k}^G)
	&\stackrel{(\ref{eq: Pr - beta1})}{=}&
	\frac{1}{\ell!}\sum_{J\colon |V(J)|=\ell }\inj(F,J)\cdot (\bPr(J, H) \pm \beta'')\\
	&\stackrel{(\ref{eq: Pr F})}{=}&\label{eq:approx}
	p\pm (\delta+\beta').
\end{eqnarray}
By the definition of $\cR$ and~\eqref{eq: RG in bII}, 
this implies that $R_G\in \cR$ and 
so $H$ is indeed $\beta$-close to a graph $G$ satisfying $R_{G}$, 
one of the regularity instances of $\cR$.

Now we show that if $\alpha>\beta$ and $H$ is $\alpha$-far from $\bP$, 
then $H$ is $(\alpha-\beta)$-far from all $R\in \cR$.
We prove this by verifying the following statement: 
if $H$ is $(\alpha-\beta)$-close to some $R \in \cR$, 
then it is $\alpha$-close to~$\bP$.

Suppose $H$ is $(\alpha-\beta)$-close to some $R= (\epsilon'',\ba, d_{\ba,k})\in \cR$.
Then there exists a $k$-graph $G_R$ such that $G_R$ satisfies $R$ and $|H\triangle G_R|\leq (\alpha-\beta)\binom{n}{k}$.
By the definition of $\cR$, we have $\sum_{J\colon |V(J)|=\ell }\inj(F,J)\cdot IC(J,d_{\ba,k})/\ell!=p\pm (\delta + \beta')$.
Similarly to the calculations leading to~\eqref{eq:approx}, we obtain $t_\inj(F,G_R)=p\pm (\delta+2\beta')$.%
\COMMENT{
Note that $IC(J,d_{\ba,k})=\bPr(J, G_R)\pm \beta''$ and so
\begin{align*}
t_\inj(F,G_R)
&=\frac{1}{\ell!}\sum_{J\colon |V(J)|=\ell }\inj(F,J)\cdot \bPr(J, G_R)\\
&\stackrel{(\ref{eq: Pr - beta})}{=}\frac{1}{\ell!}\sum_{J\colon |V(J)|=\ell }\inj(F,J)\cdot (IC(J,d_{\ba,k})\pm \beta'')\\
&=p\pm (\delta + \beta') \pm \beta'' \cdot 2^{\binom{\ell}{k}}\\
&=p\pm (\delta + 2\beta')
\end{align*}
}

By Proposition~\ref{prop:changehom}, there exists a $k$-graph $G$ such that
$t_\inj(F,G)= p \pm \delta $ and $|G\triangle G_R| \leq (\beta/2)\cdot \binom{n}{k}$.
Therefore, $G$ satisfies $\bP$ and 
$|H\triangle G| \leq |H\triangle G_R| + |G_R\triangle G|  < \alpha \binom{n}{k}$
which implies that $H$ is $\alpha$-close to satisfying $\bP$.
Thus, $\bP$ is indeed regular reducible.
\end{proof}

\subsection{Testing the maximum cut size}

We proceed with another corollary of Theorem~\ref{thm:main}.
For a given $n$-vertex $k$-graph $H$, we define  the following parameter measuring the size of a largest $\ell$-partite subgraph:
\begin{align*}
	\maxcut_\ell(H) 
	:=\binom{n}{k}^{-1} \max_{\substack{\{V_1,\dots, V_\ell\} \text{ is a}\\\text{ partition of } V(H)}} \left\{ |\cK_{k}( V_1,\dots, V_{\ell}) \cap H|  \right\}.
\end{align*}
We let
\begin{align*}
c_{\ell,k}(n) := \binom{n}{k}^{-1} \sum_{\Lambda\in \binom{[\ell]}{k}} \prod_{\lambda \in \Lambda} \left\lfloor \frac{n+\lambda-1}{\ell} \right\rfloor.
\end{align*} 
Thus $c_{\ell,k}(n) \binom{n}{k}$ is the number of edges of the complete $\ell$-partite $k$-graph on $n$ vertices whose vertex class sizes are as equal as possible. In particular, any $n$-vertex $k$-graph $H$ satisfies $\maxcut_{\ell}(H) \leq c_{\ell,k}(n)$. 

\begin{corollary}\label{cor:cut}
Suppose $\ell,k \in \N\sm \{1\}$ and $c= c(n)$ is such that $0\leq c\leq c_{\ell,k}(n)$.
Let $\bP$ be the property that an $n$-vertex $k$-graph $H$ satisfies $\maxcut_\ell(H)\geq c$. Then $\bP$ is testable.
\end{corollary}
Note that since the property of having a given edge density is trivially testable,
it follows from Corollary~\ref{cor:cut} that the property of being strongly $\ell$-colourable is also testable
(in a strong colouring, we require all vertices of an edge to have distinct colours).

Before we prove Corollary~\ref{cor:cut},
we need to introduce some notation and make a few observations.
For a given vector $\ba \in \mathbb{N}^{k-1}$, 
a density function $d_{\ba,k} : \hat{A}(k,k-1,\ba)\rightarrow [0,1]$, 
and a partition $\cL=\{ \Lambda_1,\dots, \Lambda_{\ell}\}$ of $[a_1]$, we define
\begin{align*}
	\cut(d_{\ba,k}, \cL )
	&:= k! \prod_{i=1}^{k-1} a_i^{-\binom{k}{i}}\sum_{\substack{ \hat{\bx}\in \hat{A}(k,k-1,\ba)\colon \\ \bx^{(1)}_*\in \cK_{k}( \cL) } } d_{\ba,k}(\hat{\bx})\enspace\text{ and }\\
	\maxcut_{\ell}(d_{\ba,k}) 
	&:= \max_{\substack{\cL \text{ is a partition}\\\text{  of } [a_1] \text{ with }|\cL|=\ell } } \cut(d_{\ba,k}, \cL ).
\end{align*}
Recall that if $\sP=\sP(k-1,\ba)$ is a family of partitions, then $P^{(1)}(1,1),\dots, P^{(1)}(a_1,a_1)$ denote the parts of $\sP^{(1)}$.
\begin{proposition}\label{eq: cut counting edges}
Suppose $0<1/n \ll \epsilon \ll \gamma,1/T, 1/k, 1/\ell$.
Suppose that $\ba\in [T]^{k-1}$ and $\sP$ is an $(\epsilon,\ba,d_{\ba,k})$-equitable partition of an $n$-vertex $k$-graph $H$. Let $\cL= \{ \Lambda_1,\dots, \Lambda_{\ell}\}$ be a partition of $[a_1]$ and for each $i\in [\ell]$, let $U_i:= \bigcup_{\lambda\in \Lambda_i} P^{(1)}(\lambda,\lambda).$ Then 
$$| \cK_k(U_1,\dots, U_\ell) \cap H| = (\cut(d_{\ba,k},\cL)  \pm \gamma)\binom{n}{k}.$$
\end{proposition}
Note that the $|\cK_k(U_1,\dots, U_{\ell}) \cap H|\binom{n}{k}^{-1}$ is a lower bound for $\maxcut_\ell (H)$.
\begin{proof}
For each $\hat{\bx} \in \hat{A}(k,k-1,\ba)$ with $\bx^{(1)}_* \in \cK_k(\cL)$, 
we apply Lemma~\ref{lem: counting} to  $\hat{\cP}(\hat{\bx}) $. (Recall that $\hat{\cP}(\hat{\bx})$ was defined in \eqref{eq: complex definition by address} and is an $(\epsilon, (1/a_2,\dots, 1/a_{k-1}))$-regular complex by Lemma~\ref{lem: maps bijections}(ii).)
Since $ H$ is $(\epsilon,d_{\ba,k} (\hat{\bx}) )$-regular with respect to $\hat{P}^{(k-1)}(\hat{\bx})$, we obtain
\begin{align*}
|\cK_{k}(U_1,\dots, U_{\ell}) \cap H|
& = \sum_{ \substack{ \hat{\bx}\in \hat{A}(k,k-1,\ba)\colon\\ \bx^{(1)}_*\in \cK_{k}( \cL) }}  | H \cap \cK_k(\hat{P}^{(k-1)}(\hat{\bx}))|  \\
&= \sum_{ \substack{ \hat{\bx}\in \hat{A}(k,k-1,\ba)\colon\\ \bx^{(1)}_*\in \cK_{k}( \cL) }}  (d_{\ba,k} (\hat{\bx}) \pm \epsilon) |\cK_k(\hat{P}^{(k-1)}(\hat{\bx}))|  \\
& = \sum_{ \substack{ \hat{\bx}\in \hat{A}(k,k-1,\ba)\colon\\ \bx^{(1)}_*\in \cK_{k}( \cL) }} (d_{\ba,k} (\hat{\bx}) \pm \epsilon) (1\pm \gamma/2 ) \prod_{i=1}^{k-1} a_i^{-\binom{k}{i}} n^{k}\\ 
&= (\cut(d_{\ba,k}, \cL ) \pm \gamma)\binom{n}{k}.
\end{align*}
 We conclude the final equality since $ \cut(d_{\ba,k}, \cL ) \leq 1$.
\end{proof}

\begin{proposition}\label{lem: cut size small change}
Suppose that $0<1/n \ll \nu \ll  \beta, 1/k, 1/\ell$, and $c=c(n)\in [0,c_{\ell,k}(n)]$. 
If $H$ is an $n$-vertex $k$-graph with $\maxcut_{\ell}(H)\geq c-\nu$, then there exists an $n$-vertex $k$-graph $G$ with $\maxcut_{\ell}(G) \geq c$ and $|H\triangle G| \leq \beta \binom{n}{k}$.
\end{proposition}
\begin{proof}
Since $\maxcut_{\ell}(H)\geq c-\nu$, 
there is a partition $\{U_1,\dots, U_{\ell}\}$ of $V(H)$ such that
\begin{align}\label{eq: UUUUU size}
|\cK_{k}(U_1,\dots, U_{\ell})\cap H| \geq (c-\nu) \binom{n}{k}.
\end{align}
It is easy to see that there exists a partition $\{U'_1,\dots, U'_{\ell}\}$ of $V(H)$ such that
\begin{itemize}
\item[(U$'$1)] $\sum_{i=1}^{\ell} |U_i\triangle U'_i| \leq \nu^{\frac{1}{5k}}n$,  and 
\item[(U$'$2)] $|\cK_{k}(U'_1,\dots, U'_{\ell})| \geq c \binom{n}{k}.$
\end{itemize}
Since $\nu\ll \beta$, 
we conclude
\begin{eqnarray*}
|\cK_k(U'_1,\dots, U'_{\ell})\cap H| 
\geq |\cK_k(U_1,\dots, U_{\ell})\cap H| - \sum_{i=1}^{\ell}|U_i\triangle U'_i| n^{k-1}
\stackrel{\eqref{eq: UUUUU size},  {\rm (U}'{\rm 1)} }{\geq}   (c-\beta) \binom{n}{k}.
\end{eqnarray*}
Together with (U$'$2), this shows that we can add at most $\beta \binom{n}{k}$ $k$-sets from $\cK_k(U'_1,\dots, U'_{\ell})\setminus H$ 
to $H$ to obtain a $k$-graph $G$ with 
$\maxcut_\ell(G) \geq  c$ and $|H\triangle G| \leq \beta \binom{n}{k}.$
\end{proof}

\begin{proof}[Proof of Corollary~\ref{cor:cut}]
By Theorem~\ref{thm:main}, we only need to show that $\bP$ is regular reducible.  
We assume that $\ell\geq k$, otherwise $\maxcut_{\ell}(H)=0$ for all $k$-graphs $H$. Suppose $\beta>0$. 

 Let $\overline{\epsilon}:\mathbb{N}^{k-1}\rightarrow (0,1)$ be a function such that $\overline{\epsilon}(\ba) \ll \norm{\ba}^{-1}, 1/k, 1/\ell$.
We choose constants $  \epsilon,\eta, \nu$, and $n_0, T\in \mathbb{N}$ such that $0< 1/n_0\ll \epsilon \ll 1/T \ll \eta, \nu \ll  \beta, 1/k, 1/\ell$.
For simplicity, we consider only $n$-vertex $k$-graphs $H$ with $n\geq n_0$.

Let $\bI$ be the collection of regularity instances $R=(\epsilon'',\ba, d_{\ba,k})$ such that
\begin{itemize}
\item[(R1)$_{\ref{cor:cut}}$] $\epsilon''\in \{\epsilon, 2\epsilon, \dots, \lceil(\overline{\epsilon}(\ba))^{1/2}\epsilon^{-1}\rceil \epsilon\}$,
\item[(R2)$_{\ref{cor:cut}}$]  $\ba\in [T]^{k-1}$, and
\item[(R3)$_{\ref{cor:cut}}$]  $d_{\ba,k}(\hat{\bx}) \in \{0,\epsilon^2, 2\epsilon^2,\dots, 1\}$ for every $\hat{\bx}\in \hat{A}(k,k-1,\ba)$.
\end{itemize}
Observe that by construction
$|\bI|$ is bounded by a function of $\beta$, $k$ and $\ell$. We define
$$\cR:= \left\{  (\epsilon'',\ba,d_{\ba,k}) \in \bI :   \maxcut_{\ell}(d_{\ba,k}) \geq c-\nu^{1/2} \right\}.$$
First, suppose that an $n$-vertex $k$-graph $H$ satisfies $\maxcut_{\ell}(H)\geq c$.
Then there exists a partition $\{V_1,\dots, V_{\ell}\}$ of $V(H)$ such that 
\begin{align}\label{eq: good cut V}
|\cK_k( V_1,\dots, V_{\ell}) \cap H| \geq c\binom{n}{k}.
\end{align}
Let 
$$\sO^{(1)}:= \{V_1,\dots, V_{\ell}\}.$$
For each $j\in [k-2]$ and given $\sO^{(j)}$, \eqref{eq:sP} naturally defines $\hat{\sO}^{(j)}$, and 
we define 
$$\sO^{(j+1)} := \{ \cK_{j+1}( \hat{O}^{(j)}) : \hat{O}^{(j)} \in \hat{\sO}^{(j)} \}.$$
By repeating this for each $j\in [k-2]$ in increasing order, we define a family of partitions $\sO:=\sO(k-1,\ba^{\sO}) = \{\sO^{i}\}_{i=1}^{k-1}$ with $\ba^{\sO}=(\ell,1,\dots,1)\in \mathbb{N}^{k-1}$.
Let $\sQ=\sQ(k,\ba^{\sQ})$ be an arbitrary $(1/a_1^{\sQ}, 1/n_0, \ba^{\sQ})$-equitable family of partitions on $V(H)$, where $\ba^{\sQ}=(\ell,1,\dots, 1) \in \mathbb{N}^{k}$. It is easy to see that such $\sQ$ indeed exists.%
\COMMENT{We take an $\ell$-equipartition of $V(H)$, and we let all $j$-graphs are complete with respect to the lower polyad.}
Let 
\begin{align*}
\{H_1,\dots, H_{s}\} &:= \left( \left\{  Q^{(k)} \cap H : Q^{(k)}\in \sQ^{(k)}  \right\} \cup \left\{  H \setminus \cK_{k}(\sQ^{(1)})\right\} \right) \setminus \{\emptyset\},\\
\{H_{s+1},\dots, H_{s'}\} &:= \left( \left\{  Q^{(k)} \setminus H : Q^{(k)}\in \sQ^{(k)}  \right\} \cup \left\{  \binom{V(H)}{k} \setminus (\cK_{k}(\sQ^{(1)}) \cup H) \right\} \right) \setminus \{\emptyset\},\\
\sH&:= \{H_1,\dots, H_{s'}\}.
\end{align*}
Note that $s' \leq 2\binom{\ell}{k}+2$.
Since $|V(H)|\geq n_0$, we can apply~\cite[Lemma~6.1]{JKKO2} with the following objects and parameters. \newline

{\small
\begin{tabular}{c|c|c|c|c|c|c|c|c|c|c}%|c|c|c|c|c|c|c|c}
object/parameter & $V(H)$ & $\sO$ & $ \sH $  & $\sQ$ &
$\ell$ & $s'$ & $\eta$ & $\nu$ & $\overline{\epsilon}$ &$T$ \\ \hline
playing the role of & $V$ & $\sO$ & $\sH^{(k)}$ & $\sQ$ & $o$ & $s$ & $\eta$ & $\nu$ & $\epsilon$ &$t$
 \\ 
\end{tabular}
}\newline \vspace{0.2cm}

Then we obtain $k$-graphs $G_1,\ldots,G_{s'}$ partitioning $\binom{V(H)}{k}$ and a family of partitions $\sP= \sP(k-1,\ba^{\sP})$ such that
\begin{enumerate}[label=(\Roman*)]
\item\label{item2:I} $\sP$ is $(\eta,\overline{\epsilon}(\ba^{\sP}),\ba^{\sP})$-equitable for some $\ba^{\sP} \in [T]^{k-1}$, 
\item\label{item2:II} $\sP^{(1)} \prec_{\nu} \sO^{(1)}$, 
\item\label{item2:III} for each $i\in [s]$, $G_i$ is perfectly $\overline{\epsilon}(\ba^{\sP})$-regular with respect to $\sP$, and 
\item\label{item2:IV} $\sum_{i=1}^{s} |G_i\triangle H_i|\leq \nu \binom{n}{k}$.
\end{enumerate}
Let $\epsilon':= \overline{\epsilon}(\ba^{\sP})$ and $G:= \bigcup_{i=1}^{s} G_i$.
Lemma~\ref{lem: union regularity} together with \ref{item2:III} implies that 
$G$ is perfectly $s\epsilon'$-regular with respect to $\sP$.
Also \ref{item2:IV} implies that 
\begin{align}\label{eq: GH tri nu}
|G\triangle H| \leq  \nu  \binom{n}{k}.
\end{align}
By the choice of $\epsilon$, $\overline{\epsilon}$ and $\eta$, we conclude that
$0<\epsilon\ll\epsilon' \ll 1/\norm{\ba^{\sP}}\leq 1/a_1^{\sP} \leq \eta \ll \beta, 1/k, 1/\ell$.%
\COMMENT{Note that if a $k$-graph $J$ is $(s\epsilon',d)$-regular with respect to a $(k-1)$-graph $J'$, 
then $J$ is $(\epsilon'',d')$-regular with respect to $J'$ 
for some $d'\in \{0,\epsilon^2, 2\epsilon^2,\dots, 1\}$ and $\epsilon''\in \{\epsilon, 2\epsilon,\dots,  \lceil \epsilon'^{1/2}\epsilon^{-1}\rceil \epsilon\}\cap [2s\epsilon',3s\epsilon']$.}
Similarly as in the proof of Corollary~\ref{cor:hom}, this implies that there exists 
\begin{align}\label{eq: RG in bI}
R_G=(\epsilon'',\ba^{\sP} ,d^{G}_{\ba^{\sP},k})\in \bI
\end{align} 
such that $G$ satisfies $R_G$.

Note that \ref{item2:II} implies that there exists a partition $\cL:=\{\Lambda_1,\dots, \Lambda_\ell\}$ of $[a^{\sP}_1]$ such that 
\begin{align}\label{eq: well prec nu}
\sum_{i=1}^{\ell}\sum_{\lambda \in \Lambda_i} |P^{(1)}(\lambda,\lambda) \setminus V_i| \leq \nu n.
\end{align}
For each $i\in [\ell]$, let $U_i:= \bigcup_{\lambda \in \Lambda_i} P^{(1)}(\lambda,\lambda).$ 
Then we obtain
\begin{eqnarray*}
\cut(d^{G}_{\ba^{\sP},k}, \cL) &\hspace{-0.15cm} \stackrel{{\rm Prop.}~\ref{eq: cut counting edges}}{=}& \hspace{-0.3cm} \binom{n}{k}^{-1} |\cK_{k}(U_1,\dots, U_{\ell}) \cap G| \pm  \nu\\
&\hspace{-0.15cm} \stackrel{\eqref{eq: GH tri nu}}{=}&\hspace{-0.3cm}  \binom{n}{k}^{-1} \hspace{-0.1cm} \left( |\cK_{k}(V_1,\dots, V_{\ell}) \cap H| \pm \sum_{i=1}^{\ell} \sum_{\lambda\in \Lambda_i}|P^{(1)}(\lambda,\lambda) \setminus V_i| n^{k-1} \hspace{-0.1cm} \right)  \hspace{-0.1cm} \pm 2 \nu\\
&\hspace{-0.15cm} \stackrel{\eqref{eq: well prec nu}}{=}&\hspace{-0.3cm}   \binom{n}{k}^{-1}|\cK_{k}(V_1,\dots, V_{\ell}) \cap H| \pm \nu^{1/2} \\
&\hspace{-0.15cm} \stackrel{\eqref{eq: good cut V}}{\geq}& \hspace{-0.3cm}  c-\nu^{1/2}.
\end{eqnarray*}
\COMMENT{We may add as a second line $\binom{n}{k}^{-1}\left( |\cK_{k}(V_1,\dots, V_{\ell}) \cap G| \pm \sum_{i=1}^{\ell} |U_i \setminus V_i| n^{k-1} \right) \pm  \nu$}
By the definition of $\cR$ and \eqref{eq: RG in bI}, this implies that $R_G\in \cR$ and so $H$ is indeed $\beta$-close to a graph $G$ satisfying $R_{G}$, one of the regularity instances of $\cR$.

Now we show that if $H$ is $\alpha$-far from satisfying $\bP$, 
then $H$ is $(\alpha-\beta)$-far from all $R\in \cR$. Suppose $H$ is $(\alpha-\beta)$-close to some $R= (\epsilon'',\ba, d_{\ba,k})\in \cR$.
Then there exists a $k$-graph $G_R$ such that $G_R$ satisfies $R$ and $|H\triangle G_R|\leq (\alpha-\beta)\binom{n}{k}$.
Thus there is an $(\epsilon'',\ba,d_{\ba,k})$-equitable partition $\sP'$ of $G_R$. 
By the definition of $\cR$, we have $\maxcut_{\ell}(d_{\ba,k}) \geq c- \nu^{1/2}$.
By applying Proposition~\ref{eq: cut counting edges} with $G_R, c-\nu^{1/2}, \nu^{1/2}, d_{\ba,k}$ playing the roles of $H, c, \gamma, d_{\ba,k}$, we obtain that $\maxcut_{\ell}(G_R) \geq c- 2\nu^{1/2}$.\COMMENT{Here, the lemma gives us a partition $\cL$ of $V(G_R)$. By the definition of $\maxcut_{\ell}$, such partition gives us $\maxcut_{\ell}(G_R)\geq c-2\nu^{1/2}$.}
Since $\nu \ll \beta , 1/k, 1/\ell$ and $c\in [0, c_{\ell,k}(n)]$, we can apply Proposition~\ref{lem: cut size small change} with $G_R, 2\nu^{1/2}, \beta/2, c$ playing the roles of $H, \nu, \beta, c$ to obtain a $k$-graph $G'$ such that $|G'\triangle G_R |\leq (\beta/2)\binom{n}{k}$ and $\maxcut_{\ell}(G) \geq c$.
Then 
$$|H\triangle G'| \leq |H\triangle G_R| + |G'\triangle G_R| \leq (\alpha-\beta + \beta/2)\binom{n}{k} < \alpha\binom{n}{k}.$$
Thus $H$ is $\alpha$-close to satisfying $\bP$.
Therefore, $\bP$ is indeed regular reducible.
\end{proof}

\bibliographystyle{amsplain}
\bibliography{littesting}

\providecommand{\bysame}{\leavevmode\hbox to3em{\hrulefill}\thinspace}
\providecommand{\MR}{\relax\ifhmode\unskip\space\fi MR }
\providecommand{\MRhref}[2]{%
  \href{http://www.ams.org/mathscinet-getitem?mr=#1}{#2}
}
\providecommand{\href}[2]{#2}
\begin{thebibliography}{10}

\bibitem{AB16}
N.~Alon and O.~Ben-Eliezer, \emph{Efficient removal lemmas for matrices}, Order
  \textbf{37} (2020), 83--101.

\bibitem{AB10}
N.~Alon and E.~Blais, \emph{Testing {B}oolean function isomorphism},
  Approximation, randomization, and combinatorial optimization, Lecture Notes
  in Comput. Sci., vol. 6302, Springer, Berlin, 2010, pp.~394--405.

\bibitem{ADPR03}
N.~Alon, S.~Dar, M.~Parnas, and D.~Ron, \emph{Testing of clustering}, SIAM J.
  Discrete Math. \textbf{16} (2003), 393--417.

\bibitem{AFKS00}
N.~Alon, E.~Fischer, M.~Krivelevich, and M.~Szegedy, \emph{Efficient testing of
  large graphs}, Combinatorica \textbf{20} (2000), 451--476.

\bibitem{AFNS09}
N.~Alon, E.~Fischer, I.~Newman, and A.~Shapira, \emph{A combinatorial
  characterization of the testable graph properties: it's all about
  regularity}, SIAM J. Comput. \textbf{39} (2009), 143--167.

\bibitem{AF15}
N.~Alon and J.~Fox, \emph{Easily testable graph properties}, Combin. Probab.
  Comput. \textbf{24} (2015), 646--657.

\bibitem{AS03}
N.~Alon and A.~Shapira, \emph{Testing satisfiability}, J. Algorithms
  \textbf{47} (2003), 87--103.

\bibitem{AS05}
\bysame, \emph{Linear equations, arithmetic progressions, and hypergraph
  property testing}, Theory Comput. \textbf{1} (2005), 177--216.

\bibitem{AS08}
\bysame, \emph{A characterization of the (natural) graph properties testable
  with one-sided error}, SIAM J. Comput. \textbf{37} (2008), 1703--1727.

\bibitem{AS08a}
\bysame, \emph{Every monotone graph property is testable}, SIAM J. Comput.
  \textbf{38} (2008), 505--522.

\bibitem{AT10}
T.~Austin and T.~Tao, \emph{Testability and repair of hereditary hypergraph
  properties}, Random Structures Algorithms \textbf{36} (2010), 373--463.

\bibitem{ARS07}
C.~Avart, V.~R\"odl, and M.~Schacht, \emph{Every monotone 3-graph property is
  testable}, SIAM J. Discrete Math. \textbf{21} (2007), 73--92.

\bibitem{BCLR08}
M.~Ben~{O}r, D.~Coppersmith, M.~Luby, and R.~Rubinfeld, \emph{Non-abelian
  homomorphism testing, and distributions close to their self-convolutions},
  Random Structures Algorithms \textbf{32} (2008), 49--70.

\bibitem{BSS10}
I.~Benjamini, O.~Schramm, and A.~Shapira, \emph{Every minor-closed property of
  sparse graphs is testable}, Adv. Math. \textbf{223} (2010), 2200--2218.

\bibitem{BCLSSV06}
C.~Borgs, J.~Chayes, L.~Lov\'asz, V.~T. S\'os, B.~Szegedy, and K.~Vesztergombi,
  \emph{Graph limits and parameter testing}, S{TOC}'06: {P}roceedings of the
  38th {A}nnual {ACM} {S}ymposium on {T}heory of {C}omputing, ACM, New York,
  2006, pp.~261--270.

\bibitem{ES12}
G.~Elek and B.~Szegedy, \emph{A measure-theoretic approach to the theory of
  dense hypergraphs}, Adv. Math. \textbf{231} (2012), 1731--1772.

\bibitem{EJKO:19}
A.~Espuny~D\'{\i}az, F.~Joos, D.~K\"{u}hn, and D.~Osthus, \emph{Edge
  correlations in random regular hypergraphs and applications to subgraph
  testing}, SIAM J. Discrete Math. \textbf{33} (2019), 1837--1863.

\bibitem{Fish05}
E.~Fischer, \emph{The difficulty of testing for isomorphism against a graph
  that is given in advance}, SIAM J. Comput. \textbf{34} (2005), 1147--1158.

\bibitem{FN07}
E.~Fischer and I.~Newman, \emph{Testing versus estimation of graph properties},
  SIAM J. Comput. \textbf{37} (2007), 482--501.

\bibitem{FPS16}
J.~Fox, J.~Pach, and A.~Suk, \emph{A polynomial regularity lemma for
  semialgebraic hypergraphs and its applications in geometry and property
  testing}, SIAM J. Comput. \textbf{45} (2016), 2199--2223.

\bibitem{FR92}
P.~Frankl and V.~R\"odl, \emph{The uniformity lemma for hypergraphs}, Graphs
  Combin. \textbf{8} (1992), 309--312.

\bibitem{FIS05}
K.~Friedl, G.~Ivanyos, and M.~Santha, \emph{Efficient testing of groups},
  S{TOC}'05: {P}roceedings of the 37th {A}nnual {ACM} {S}ymposium on {T}heory
  of {C}omputing, ACM, New York, 2005, pp.~157--166.

\bibitem{Gol:17}
O.~Goldreich, \emph{Introduction to property testing}, Cambridge University
  Press, Cambridge, 2017.

\bibitem{GGR98}
O.~Goldreich, S.~Goldwasser, and D.~Ron, \emph{Property testing and its
  connection to learning and approximation}, J. ACM \textbf{45} (1998),
  653--750.

\bibitem{GT03}
O.~Goldreich and L.~Trevisan, \emph{Three theorems regarding testing graph
  properties}, Random Structures Algorithms \textbf{23} (2003), 23--57.

\bibitem{Gow97}
W.~T. Gowers, \emph{Lower bounds of tower type for {S}zemer\'edi's uniformity
  lemma}, Geom. Funct. Anal. \textbf{7} (1997), 322--337.

\bibitem{Gow07}
\bysame, \emph{Hypergraph regularity and the multidimensional {S}zemer\'edi
  theorem}, Ann. of Math. \textbf{166} (2007), 897--946.

\bibitem{JLR00}
S.~Janson, T.~\L{}uczak, and A.~Rucinski, \emph{Random graphs},
  Wiley-Interscience Series in Discrete Mathematics and Optimization,
  Wiley-Interscience, New York, 2000.

\bibitem{JKKO2}
F.~Joos, J.~Kim, D.~K\"uhn, and D.~Osthus, \emph{Hypergraph regularity and
  random sampling}, Random Structures Algorithms \textbf{62} (2023), 956--1015.

\bibitem{KM15a}
M.~Karpinski and R.~Mark\'o, \emph{Explicit bounds for nondeterministically
  testable hypergraph parameters}, arXiv:1509.03046 (2015).

\bibitem{KM15}
\bysame, \emph{On the complexity of nondeterministically testable hypergraph
  parameters}, arXiv:1503.07093 (2015).

\bibitem{KNR02}
Y.~Kohayakawa, B.~Nagle, and V.~R\"odl, \emph{Efficient testing of hypergraphs
  (extended abstract)}, Automata, languages and programming, Lecture Notes in
  Comput. Sci., vol. 2380, Springer, Berlin, 2002, pp.~1017--1028.

\bibitem{KRS02}
Y.~Kohayakawa, V.~R\"odl, and J.~Skokan, \emph{Hypergraphs, quasi-randomness,
  and conditions for regularity}, J. Combin. Theory Ser. A \textbf{97} (2002),
  307--352.

\bibitem{GS14}
G.~Lior and A.~Shapira, \emph{Deterministic vs non-deterministic graph property
  testing}, Israel J. Math. \textbf{204} (2014), 397--416.

\bibitem{Lov12}
L.~Lov\'asz, \emph{Large networks and graph limits}, American Mathematical
  Society Colloquium Publications, vol.~60, American Mathematical Society,
  Providence, RI, 2012.

\bibitem{LS10}
L.~Lov\'asz and B.~Szegedy, \emph{Testing properties of graphs and functions},
  Israel J. Math. \textbf{178} (2010), 113--156.

\bibitem{LV13}
L.~Lov\'asz and K.~Vesztergombi, \emph{Non-deterministic graph property
  testing}, Combin. Probab. Comput. \textbf{22} (2013), 749--762.

\bibitem{NS13}
I.~Newman and C.~Sohler, \emph{Every property of hyperfinite graphs is
  testable}, SIAM J. Comput. \textbf{42} (2013), 1095--1112.

\bibitem{RS07STOC}
V.~R\"odl and M.~Schacht, \emph{Property testing in hypergraphs and the removal
  lemma}, S{TOC}'07---{P}roceedings of the 39th {A}nnual {ACM} {S}ymposium on
  {T}heory of {C}omputing, ACM, New York, 2007, pp.~488--495.

\bibitem{RS07count}
\bysame, \emph{Regular partitions of hypergraphs: counting lemmas}, Combin.
  Probab. Comput. \textbf{16} (2007), 887--901.

\bibitem{RS07}
\bysame, \emph{Regular partitions of hypergraphs: regularity lemmas}, Combin.
  Probab. Comput. \textbf{16} (2007), 833--885.

\bibitem{RS09}
\bysame, \emph{Generalizations of the removal lemma}, Combinatorica \textbf{29}
  (2009), 467--501.

\bibitem{RS04}
V.~R\"odl and J.~Skokan, \emph{Regularity lemma for {$k$}-uniform hypergraphs},
  Random Structures Algorithms \textbf{25} (2004), 1--42.

\bibitem{RS96}
R.~Rubinfeld and M.~Sudan, \emph{Robust characterizations of polynomials with
  applications to program testing}, SIAM J. Comput. \textbf{25} (1996),
  252--271.

\bibitem{RS78}
I.~Z. Ruzsa and E.~Szemer{\'e}di, \emph{Triple systems with no six points
  carrying three triangles}, Combinatorics ({P}roc. {F}ifth {H}ungarian
  {C}olloq., {K}eszthely, 1976), {V}ol. {II}, Colloq. Math. Soc. J\'anos
  Bolyai, vol.~18, North-Holland, Amsterdam-New York, 1978, pp.~939--945.

\bibitem{Tao06}
T.~Tao, \emph{A variant of the hypergraph removal lemma}, J. Combin. Theory
  Ser. A \textbf{113} (2006), 1257--1280.

\end{thebibliography}

\end{document}